\documentclass[12pt,english]{article}
\usepackage[T1]{fontenc}
\usepackage[latin9]{inputenc}
\usepackage{geometry}
\usepackage{fancyhdr}
\usepackage{verbatim}
\usepackage{color}
\usepackage{mathrsfs}
\usepackage{amsmath}
\usepackage{amsthm}
\usepackage{amssymb}
\usepackage[misc]{ifsym}
\usepackage{hyperref}
\usepackage{stmaryrd}
\makeatletter
\theoremstyle{plain}
\newtheorem{theorem}{\protect\theoremname}[section] 
\theoremstyle{definition}
\newtheorem{definition}[theorem]{\protect\definitionname}
\theoremstyle{plain}
\newtheorem{lemma}[theorem]{\protect\lemmaname}
\theoremstyle{remark}
\newtheorem{remark}[theorem]{\protect\remarkname}
\theoremstyle{plain}
\newtheorem{corollary}[theorem]{\protect\corollaryname}
\theoremstyle{plain}
\newtheorem{proposition}[theorem]{\protect\propositionname}
\theoremstyle{plain}
\newtheorem{example}[theorem]{\protect\examplename}

\@ifundefined{date}{}{\date{}}
\usepackage{fancyhdr}

\usepackage{babel}
\providecommand{\definitionname}{Definition}
\providecommand{\lemmaname}{Lemma}
\providecommand{\theoremname}{Theorem}
\providecommand{\corollaryname}{Corollary}
\providecommand{\remarkname}{Remark}
\providecommand{\propositionname}{Proposition}
\providecommand{\examplename}{Example}

\numberwithin{equation}{section}  %%%%????????

\makeatother
\def\rr{{\mathbb R}}
\def\rn{{{\rr}^n}}
\def\D{\mathcal{D}}
\def\d{\mathrm{d}}

\def\H{\mathcal{H}}
\def\M{\mathcal{M}}

\def\sgn{\operatorname{sgn}}
\def\BMO{\operatorname{BMO}(\rn)}
\def\CMO{\operatorname{CMO}(\rn)}
\begin{document}
	\title
	{\bf\Large
		 Bourgain-Morrey-Lorentz spaces and  operators on them
		\footnotetext{The corresponding author Jingshi Xu is supported by the National Natural Science Foundation of China (Grant No. 12161022) and the Science and Technology Project of Guangxi (Guike AD23023002).
		Pengfei Guo is supported by Hainan Provincial Natural Science Foundation of China (Grant No. 122RC652).
	}
%		\footnotetext{{\it Key words and phrases}: pseudodifferential operator, Orlicz space, weight
%			\newline\indent\hspace{1mm} {\it 2020 Mathematics Subject Classification}: }
		}
	
	\date{}
	
	\author{Tengfei Bai\textsuperscript{a}, Pengfei Guo\textsuperscript{a},  Jingshi Xu\textsuperscript{b,c,d}\footnote{Corresponding author, E-mail: jingshixu@126.com}  \\
		{\scriptsize \textsuperscript{a} \scriptsize College of Mathematics and Statistics, Hainan Normal University, Haikou, Hainan 571158,
			China}\\
		{\scriptsize  \textsuperscript{b} School of Mathematics and Computing Science, Guilin University of Electronic Technology, Guilin 541004, China} \\
		{\scriptsize  \textsuperscript{c} Center for Applied Mathematics of Guangxi (GUET), Guilin 541004, China}\\
		{\scriptsize  \textsuperscript{d} Guangxi Colleges and Universities Key Laboratory of Data Analysis and Computation, Guilin 541004, China}
	}
	
	\pagestyle{myheadings}\markboth{\footnotesize\rm\sc Tengfei Bai, Pengfei Guo,  Jingshi Xu}
	{\footnotesize\rm\sc  }
	
\maketitle
\begin{abstract}
We introduce  Bourgain-Morrey-Lorentz spaces and give a description of the predual of Bourgain-Morrey-Lorentz spaces via the block spaces. As an application of duality, we obtain the   boundedness of Hardy-Littlewood maximal operator, sharp maximal operator, Calder\'on-Zygmund operator, fractional integral operator, commutator on Bourgain-Morrey-Lorentz spaces. Moreover, we obtain a weak Hardy factorization terms of Calder\'on-Zygmund operator in  Bourgain-Morrey-Lorentz spaces. Using this result, we obtain a characterization of functions in $\BMO$ (the functions of ``bounded mean oscillation'') via the boundedness of commutators generated by them and  a homogeneous  Calder\'on-Zygmund operator. In the last, we show that the commutator generated by a function $b$ and a homogeneous  Calder\'on-Zygmund operator  is a compact operator on Bourgain-Morrey-Lorentz spaces if and only if $b$ is the
limit of compactly supported  smooth functions in $\BMO$.
\end{abstract}

%secondary 

\section{Introduction}

Morrey spaces were first introduced by Morrey in \cite{M43},
which include Lebesgue spaces, Lipschitz spaces, bounded mean oscillation spaces.
A special case of Bourgain-Morrey spaces was introduced by Bourgain in \cite{B91}. Over the past years, Bourgain-Morrey spaces have been applied to
various partial differential equations, especially the Strichartz estimate and nonlinear Schr\"odinger equations, see, e.g., \cite{BPV07,B91,M16,MV98,MVV99}.

In \cite{M16}, Masaki proposed a predual space of the Bourgain-Morrey space and characterized total boundedness of a bounded set in the Bourgain-Morrey space.
Further properties of Bourgain-Morrey spaces such as inclusion, dilation, translation, nontriviality, diversity, approximation,  density and duality were obtained by  Hatano, Nogayama, Sawano and Hakim in \cite{HNSH23}.
The boundedness of operators such as the Hardy-Littlewood maximal operator, fractional integral operators, fractional maximal operators, and singular integral operators on Bourgain-Morrey spaces was also  given in \cite{HNSH23}.

After then, some function spaces extending Bourgain-Morrey spaces were established.
For example, in \cite{ZSTYY23}, Zhao et al.  introduced Besov-Bourgain-Morrey spaces which connect Bourgain-Morrey spaces with amalgam-type spaces.  They obtained predual, dual spaces and complex interpolation  of these spaces. They also gave an equivalent norm with an integral expression  and  obtained the continuity  on these spaces
of the Hardy-Littlewood maximal operator, the fractional integral and the Calder\'on-Zygmund operator. 

Immediately after \cite{ZSTYY23},
 Hu, Li, and Yang  introduced  Triebel-Lizorkin-Bourgain-Morrey spaces which connect Bourgain-Morrey spaces and global Morrey spaces in \cite{HLY23}. 
They considered the embedding relations between Triebel-Lizorkin-Bourgain-Morrey spaces and Besov-Bourgain-Morrey spaces.  
  They studied various fundamental real-variable properties of these spaces. They obtained the sharp mapping on these spaces
of  the Hardy-Littlewood maximal operator, the Calder\'on-Zygmund operator and the fractional integral.

Inspired by the generalized grand Morrey spaces and Besov-Bourgain-Morrey spaces, Zhang et al. introduced generalized grand Besov-Bourgain-Morrey spaces in \cite{ZYZ24}. They obtained predual spaces and the Gagliardo-Peetre, the $\pm$ interpolation theorems,  extrapolation theorem.  The boundedness of the Hardy-Littlewood maximal operator, the fractional integral and the Calder\'on-Zygmund operator is also proved. 

In \cite{Dai24}, Dai et al. obtained the Bourgain-Brezis-Mironescu characterization on ball Banach function spaces on $\rn$.
In \cite{ZYY24}, Zhu, Yang and Yuan obtained the Bourgain-Brezis-Mironescu-type characterization of  the inhomogeneous ball Banach Sobolev space on extension domains. And they used the characterization to ten kinds of function spaces, including Besov-Bourgain-Morrey spaces. 
In particular, they obtained the boundedness of  the Hardy-Littlewood maximal operator on the preduals of Besov-Bourgain-Morrey spaces.
In \cite{ZYY242}, Zhu, Yang and Yuan obtained the Brezis-Seeger-Van Schaftingen-Yung-type characterization of  the homogeneous ball Banach Sobolev space on uniform domains. The spaces also includes  Besov-Bourgain-Morrey spaces.

In \cite{R12}, Ragusa studied the Embeddings of Morrey-Lorentz spaces.
In \cite{SE18}, Sawano and El-Shabrawy studied the  weak Morrey spaces and their preduals. As an application, they  obtained the boundedness of singular integral operators in weak Morrey spaces.
Dao and Krantz
introduced  the block space, which is the predual  of the Morrey-Lorentz space in  \cite{DK24}. They obtained the Morrey-Lorentz boundedness of powered Hardy-Littlewood maximal function, sharp maximal function, linear Calder\'on-Zygmund operator, commutator. They also extended the  Hardy factorization in terms of linear Calder\'on-Zygmund operators to Morrey-Lorentz spaces  and obtained  the  compactness characterization of  commutator  in Morrey-Lorentz spaces.

Very recently, the first author and the third author of the paper \cite{BX25} introduced the weighted homogeneous Bourgain-Morrey Besov spaces and Triebel-Lizorkin associated with the operator. They obtained the Hardy-Littlewood maximal function and Fefferman-Stein maximal inequality on weighted Bourgain-Morrey sapces. Some characterizations of homogeneous Bourgain-Morrey Besov and Triebel-Lizorkin spaces were also obtained, such as  Peetre maximal function, compactly supported  functions,  atomic decompositions and molecular decompositions. 

Inspired by the above literature, we introduce   Bourgain-Morrey-Lorentz spaces and obtain the predual of Bourgain-Morrey-Lorentz spaces via  block spaces. 
The article is organized as follows. In Section \ref{preliminaries}, we give the definition of Bourgain-Morrey-Lorentz spaces and  some fundamental properties such as embedding, dilation, translation invariance, nontriviality and completeness. We also show that the Bourgain-Morrey-Lorentz space include the sequence space $\ell^r$ by a geometric property. As an application, we get  the Lorentz space differs from Bourgain-Morrey-Lorentz space.
In Section \ref{block space}, we introduce the block space $\mathcal H_{ p', q'}^{t',r'} (\rn) $, which is the predual of the Bourgain-Morrey-Lorentz space.
In Section \ref{Applicaiton operator}, we obtain the  boundedness of Hardy-Littlewood maximal operator,  sharp maximal operator, Calder\'on-Zygmund operator, fractional integral operator, commutator on Bourgain-Morrey-Lorentz spaces and the block spaces.
In Section \ref{Hardy factorization Section}, we extend a weak Hardy factorization in terms of Calder\'on-Zygmund operator on  Bourgain-Morrey-Lorentz spaces. As an application, we obtain a characterization for functions $b$ in $\BMO$ via the boundedness of commutator $[b,T]$, where $T$ is a homogeneous  Calder\'on-Zygmund operator.
In the last Section \ref{Compactness}, a  Bourgain-Morrey-Lorentz compactness characterization of comutator generated by a function $b$ and a homogeneous  Calder\'on-Zygmund operator  is obtained. Precisely, the commutator generated by a function $b$ and  a homogeneous  Calder\'on-Zygmund operator  is a compact operator on Bourgain-Morrey-Lorentz spaces if and only if $b$ is the limit of compactly supported  smooth functions in $\BMO$.

Throughout the paper, we use the following notations.
For $x \in \rn$ and  $r >0$, we define $B(x,r) = \{ y\in \rn: |y-x| <r \}$   the open ball in $\rn$ equipped with the Euclidean norm $|\cdot|$. Denote by  $Q = Q(c_Q, \ell(Q)/2)$ the cube with center $c_Q$ and the side length $\ell(Q)$.
Let $a Q $ be the cube concentric with $Q$, having  the side length $a \ell (Q)$.
For $1\le p \le \infty$, let $ p' \in [1,\infty]$  such that $1/p + 1/p' =1$ and $p'$ is called the  conjugate exponent of $p$. 
For $j\in\mathbb{Z}$, $m\in\mathbb{Z}^{n}$, let $Q_{j,m}:=\prod_{i=1}^{n}[2^{-j}m_{i},2^{-j}(m_{i}+1))$.
 We
denote by $\mathcal{D} := \{Q_{j,m} : j\in \mathbb Z, m \in \mathbb Z ^n \}$ the the family of all dyadic cubes in $\mathbb{R}^{n}$,
while $\mathcal{D}_{j}$ is the set of all dyadic cubes with $\ell(Q)=2^{-j},j\in\mathbb{Z}$.
Let $\chi_{E}$ be the characteristic function of the set $E\subset\mathbb{R}^{n}$.
Let $\mathbb N := \{ 1,2,3, \cdots\} $ and $\mathbb{N}_{0}:=\mathbb{N\cup}\{0\}$.
Let $L^0(\rn)$ stands for the set of all measurable functions over $\rn$.
Let $L_c^\infty (\rn)$ be the set of   bounded functions with compact support on $\rn$.
Let $C_c^\infty (\rn)$ be the set of   smooth functions with compact support on $\rn$.
%denote the Schwartz space on $\rn$, and let $\mathscr{S}'(\mathbb{R}^{n})$
%be its dual. 
%Let $\mathcal{P}(\mathbb{R}^{n})$ be the class of the
%polynomials on $\mathbb{R}^{n}$. Denote by $\mathscr{S}'/\mathcal{P}(\mathbb{R}^{n})$
%the space of tempered distributions modulo polynomials. Let $\mathscr{S}_{0}(\mathbb{R}^{n}):=\{f\in\mathscr{S}(\mathbb{R}^{n}):\partial^{\alpha}\mathcal{F}(g)(0)=0$
%for all multi-indices $\alpha$\}. Recall that $\mathscr{S}'/\mathcal{P}(\mathbb{R}^{n})$
%is the dual of $\mathscr{S}_{0}(\mathbb{R}^{n})$. 
We use the symbol
$A\lesssim B$ to denote that there exists a positive constant $c$
such that $A\le cB$. If $A\lesssim B$ and $B\lesssim A$, then we
denote $A\approx B$. The letters $c,C$ will denote various positive
constants and may change in different lines. 
%The notation $C_{p,q, \ldots}$ stands for the  constant $0<C <\infty$ depending only on $p,q,$.

\section{Preliminaries} \label{preliminaries}
In this section, we study some properties of the  Bourgain-Morrey-Lorentz space.      We first recall the definition of Lorentz spaces.
\subsection{Lorentz spaces}
\begin{definition}
	For a measurable function $f$ on $(X, \mu)$, the {\it distribution function} of $f$ is the function $d_f$ defined on $[0,\infty)$ by $d_f (\alpha) = \mu (  \{ x\in X: |f(x)|  > \alpha \}  )$.
	The {\it decreasing rearrangement} of $f$ is the function $f^*$  defined on $[0,\infty)$ by
	\begin{equation*}
		f^*(t) = \inf\{  s>0 : d_f (s)  \le t    \}  = \inf\{  s\ge 0 : d_f (s)  \le t    \}
	\end{equation*}
with the usual convention that $\inf \emptyset = \infty$.
\end{definition}

Given  a measurable function $f$ and $0< p, q \le \infty$, define
\begin{equation*}
	\| f\|_{L^{p,q} (X) } = \begin{cases}
		\left(   \int_0^\infty   \left(t^{1/p  }f^* (t) \right)^q \frac{ \d t}{t}     \right)^{1/q}, & \operatorname{if} \; q <\infty , \\
		\sup_{t>0}  t^{1/p}  f^* (t), & \operatorname{if}\; q = \infty. 
	\end{cases}
\end{equation*}
\begin{remark}\label{basic Lorentz}
From \cite[Theorem 1.4.13]{G14}, for all $0< p, q \le \infty$ the spaces $L^{p,q} (X) $ are complete with respect to there quasi-norm and they are quasi-Banach spaces. $L^{p,q} (X) $ are normable when $ 1 < p \le \infty  $, and  $1 \le q \le \infty$, see \cite[page 83, line -12]{G14}. Note that $L^{1,\infty} (\rn)$ is not normable (\cite[Exercise 1.1.13]{G14}).

%The following dilation operator on Lorentz spaces can be found in \cite[page 52]{G14} and we will use it to obtain the dilation property of Bourgain-Morrey-Lorentz spaces. 
For all $0<p,r<\infty$ and $0<q \le \infty$, we have
\begin{equation*}
	\| |f|^r \|_{L^{p,q} (X)}  = \| f \|_{L^{pr,qr} (X)} ^r.
 \end{equation*} 
%Let $\epsilon >0$ and $\delta^\epsilon (f) (x) =f (\epsilon x)$. From $ d _{\delta^\epsilon f  } (\alpha) = \epsilon^{-n} d_f (\alpha) $ and $ (\delta^\epsilon (f) )^* (t) =f^* (\epsilon^n t) $, we have
%\begin{equation*}
%	\| \delta^\epsilon (f) \|_{L^{p,q}  (X)} = \epsilon^{-n/p} \| f\|_{L^{p,q} (X)}.
%\end{equation*}
%Note that the above equation holds when $ \epsilon x  $ is still in $X$.  
If $\mu (X) <\infty$ and $0< p_1 \le p_2 <\infty$, then
\begin{equation*}
	\| f\|_{L^{p_1,q} (X) }  \le \mu(X)^{1/p_1 - 1/p_2} \| f\|_{L^{p_2,q} (X) }.
\end{equation*} 

\end{remark}
We sometimes use the distribution function to calculate the Lorentz norm. The following result from \cite[Proposition 1.4.9]{G14}. 

\begin{lemma}
	For $0<p<\infty$ and $0<q \le \infty $, we have 
	\begin{equation*}
			\| f\|_{L^{p,q} (X)} = \begin{cases}
			 p^{1/q} \left(   \int_0^\infty   \left(d_f(s)^{1/p } s \right)^q \frac{ \d s}{s}     \right)^{1/q}, & \operatorname{if} \; q <\infty , \\
			\sup_{s>0}  s  d_f (s)^{1/p}, & \operatorname{if}\; q = \infty. 
		\end{cases}
	\end{equation*}
\end{lemma}

The following result is the H\"older-type inequality in Lorentz spaces.
\begin{lemma}[Theorem 2.9, \cite{CC21}] \label{holder lorentz}
	Let $(X,\mathcal A, \mu)$ be a measure space. Let $1\le p, q \le \infty$ and let $p' , q'$ be their conjugate exponents. If $f \in L^{p,q} (X)$ and $g \in L^{p',q'} (X)$, then 
	\begin{equation*}
		\|fg\|_{L^1 (X)}  \le  \| f \|_{  L^{p,q} (X) }  \| g \|_{  L^{p',q'} (X) }.
	\end{equation*}
\end{lemma}

The following lemma  is the embedding of Lorentz spaces.

\begin{lemma}[Proposition 1.4.10, \cite{G14}]\label{embed Lorentz}
	Suppose that $ 0 < p \le \infty$ and $ 0< q_1 < q_2 \le \infty $. Then there exists a constant $C_{p, q_1,q_2} >0$ such that 
	\begin{equation*}
		\| f \|_{L^{p, q_2} (X)}  \le C_{p, q_1,q_2} \| f \|_{  L^{p, q_1} (X)}.
	\end{equation*}
\end{lemma}

Denote by $\M$  the Hardy-Littlewood maximal function of $f$:
\begin{equation*}
	\M f (x) = \sup_{B \ni x } \frac{1}{\mu(B)} \int_B | f (y) | \d \mu (y) 
\end{equation*}
where the sup is taken over all balls $B$ containing $x \in X$.
Then from \cite[Theorem 1.4.19]{G14},  we have the following result.
\begin{lemma}\label{HL Lorentz}
	For $1< p <\infty$ and $ 1 \le q \le \infty$, the maximal operator $\M$ is bounded on  $L^{p,q} (X)$. 
\end{lemma}
The duals of Lorentz spaces can be seen in \cite[Theorem 1.4.16]{G14}. 
For more details on the Lorentz spaces, we refer the reader to \cite{ CC21, G14}.

\subsection{Bourgain-Morrey-Lorentz spaces}
Now we define the Bourgain-Morrey-Lorentz spaces.
\begin{definition}
	Let $ 0 < q \le \infty$. Let $ 0 <p \le t< \infty $ and $ 0 < r  \le \infty $. Then the Bourgain-Morrey-Lorentz space $M_{p,q}^{t,r} (\rn) $ is the set of all $f\in L^0 (\rn)$ such that 
	\begin{equation*}
		\| f \|_{M_{p,q}^{t,r} (\rn)   } = \left(  \sum_{Q \in \D} |Q|^{r /t- r/p} \| f\|_{ L^{p,q} (Q)} ^r \right)^{1/r} <\infty.
	\end{equation*}
\end{definition}
\begin{remark}
	If $q =p$, then $M_{p,q}^{t,r} (\rn)$  is the Bourgain-Morrey space $M_{p}^{t,r} (\rn)$  in \cite{HNSH23}.
	
	If $r=\infty$,  then $M_{p,q}^{t,\infty} (\rn)$  is the Morrey-Lorentz space; for example, see \cite{DK24}.
	
	If $q = \infty$, then $M_{p,\infty}^{t,r} (\rn)$ becomes the weak Bourgain-Morrey space. As far as the authors know, it is a new space.
\end{remark}

\subsection{Fundamental properties of the Bourgain-Morrey-Lorentz space $M_{p,q}^{t,r} (\rn) $}
In this subsection, we get some basic properties of Bourgain-Morrey-Lorentz spaces $M_{p,q}^{t,r} (\rn) $, such as embedding, dilation, translation invariance, nontriviality and completeness.

Given two quasi-Banach spaces $X$ and $Y$, we write $X \hookrightarrow Y $ if $X \subset  Y$ and the natural embedding is bounded.
The embedding of Morrey-Lorentz spaces can be found in \cite{R12}.
The following lemma is obtained from $ \ell ^{r_1}  \hookrightarrow  \ell ^{r_2}$ for $ 0 < r_1 \le r_2 \le \infty$. 
\begin{lemma}
	Let $0\le q \le \infty$.  Let $ 0 <p \le t < r_1 \le r_2 \le \infty $. Then $M_{p,q}^{t,r_1} (\rn)   \hookrightarrow  M_{p,q}^{t,r_2} (\rn) $.
\end{lemma}
Using  Lemma \ref{embed Lorentz}, we  obtain the following  embedding of  Bourgain-Morrey-Lorentz spaces.

\begin{lemma}
	Let $ 0< q_1 < q_2 \le \infty $. Let $ 0< p \le t < r \le \infty $. Then $M_{p,q_1}^{t,r} (\rn) \hookrightarrow M_{p,q_2}^{t,r} (\rn)  $.
\end{lemma}

\begin{lemma}[Lemma 23,  \cite{SFH20}]
	Let $0 < p_1 < p_2 \le t <\infty$. Then $ M_{p_2,\infty}^{t,\infty} (\rn) \hookrightarrow M_{p_1,p_1}^{t,\infty} (\rn)   $.
\end{lemma}

\begin{remark}
	From \cite[Example 17]{SFH20}, the weak Morrey spaces $M_{p,\infty}^{t,\infty} (\rn) $ are strict bigger than   Morrey spaces $M_{p,p}^{t,\infty} (\rn) $. Then from
	\begin{equation} \label{embedding p - epsilon}
		M_{p,q}^{t,r} (\rn) \hookrightarrow M_{p,\infty}^{t,r} (\rn) \hookrightarrow M_{p,\infty}^{t,\infty} (\rn) \hookrightarrow M_{p- \epsilon,p- \epsilon}^{t,\infty} (\rn)
	\end{equation}
and the completeness of Morrey spaces $M_{p- \epsilon,p- \epsilon}^{t,\infty} (\rn)$, by the Fatou lemma, it is not hard to show  the completeness of	$M_{p,q}^{t,r} (\rn)$.
	\end{remark}

Now we turn to consider the translation invariance of spaces $ M_{p,q}^{t,r} (\rn) $.
\begin{lemma}\label{translation BML}
			Let $ 0 < q \le \infty$. Let $ 0 <p \le t< \infty $ and $ 0 < r  \le \infty $.  There exists a constant $C_{q,p,n,r} >0$ depends on $q,p,n,r $ such that for $y \in \rn$  and $f \in  M_{p,q}^{t,r} (\rn) $, 
		\begin{equation*}
			\| f (\cdot -y)  \|_{   M_{p,q}^{t,r} (\rn)  }  \le C_{q,p,n,r} \| f   \|_{   M_{p,q}^{t,r} (\rn)  }.
		\end{equation*}
\end{lemma}

\begin{proof}
	We only prove $r<\infty$  since $r = \infty$  is similar. 
	We use the idea from \cite[Lemma 2.5]{HNSH23}. Then
	\begin{align*}
		\| f (\cdot -y)  \|_{   M_{p,q}^{t,r} (\rn)  }   & = \left\| \left\{ |Q_{v,m}|^{1/t-1/p}  \| f (\cdot -y) \chi_{ Q_{v,m}} \|_{ L^{p,q} (\rn) }    \right\}_{(v,m) \in \mathbb Z ^{1+n}}          \right\|_{\ell^r} \\
		 & = \left\| \left\{ |Q_{v,m}|^{1/t-1/p}  \| f  \chi_{ Q_{v,m} -y} \|_{ L^{p,q} (\rn) }    \right\}_{(v,m) \in \mathbb Z ^{1+n}}          \right\|_{\ell^r} .
	\end{align*}
For each $v\in \mathbb Z $ and $y \in \rn$, there exists a finite collection $m_1 (v,y ), \ldots , m_{2^n} (v,y )  \in \mathbb Z^n $ such that 
\begin{equation*}
	Q_{v,m} -y  \subset \bigcup_{k=1}^{2^n} Q_{v, m+m_k (v,y ) }
\end{equation*}
for all $m\in \mathbb Z ^n$. Therefore
\begin{align*}
& \| f (\cdot -y)  \|_{   M_{p,q}^{t,r} (\rn)  }  
\\
&
\le C_{q,p,n} \left\| \left\{ |Q_{v,m}|^{1/t-1/p} \sum_{k=1}^{2^n} \| f  \chi_{ Q_{v,m+m_k (v,y )} } \|_{ L^{p,q} (\rn) }    \right\}_{(v,m) \in \mathbb Z ^{1+n}}          \right\|_{\ell^r} \\
&
\le C_{q,p,n,r}\sum_{k=1}^{2^n} \left\| \left\{ |Q_{v,m}|^{1/t-1/p}  \| f  \chi_{ Q_{v,m+m_k (v,y )} } \|_{ L^{p,q} (\rn) }    \right\}_{(v,m) \in \mathbb Z ^{1+n}}          \right\|_{\ell^r} \\
& = 2^n C_{q,p,n,r} \| f\|_{  M_{p,q}^{t,r} (\rn)   }.
\end{align*}
This provides the desired inequality.
\end{proof}
Using this property and the triangle inequality for $ M_{p,q}^{t,r} (\rn)$, we obtain the following convolution inequality.
\begin{corollary}\label{convolution BML}
	Let $1 \le q \le \infty$. 
	Let $1 < p \le t < \infty$  and $ 1 \le r \le \infty$. Then there exists a constant $C_{q,p,n,r} >0$  such that
	\begin{equation*}
		\| g * f \|_{ M_{p,q}^{t,r} (\rn) }  \le C_{q,p,n,r}  \|g\|_{L^1} \| f\|_{ M_{p,q}^{t,r} (\rn)}
	\end{equation*}
for all $f \in  M_{p,q}^{t,r} (\rn)$  and $g \in L^{1} (\rn)$.
\end{corollary}

For the boundedness of Hardy-Littlewood maximal operator in subsection \ref{HL section},
we define a equivalent norm of $M_{p,q}^{t,r} (\rn)  $ below. To do so, we use a dyadic grid $\mathcal D_{k,\vec a}$, $k \in \mathbb Z$, $\vec a \in \{ 0,1,2\}^n$. More precisely, let
\begin{equation*}
	\mathcal D_{k,\vec a}^0  \equiv  \{ 2^{-k} [ m +a/3, m+a/3 +1) : m\in \mathbb Z   \}
\end{equation*}
for $k \in \mathbb Z$  and $a = 0,1,2$. Consider
\begin{equation*}
		\mathcal D_{k,\vec a} \equiv \{ Q_1 \times Q_2 \times \cdots \times Q_n :Q_j \in	\mathcal D_{k,\vec a_j}^0 ,  j =1,2,\ldots , n \}
\end{equation*}
for $k \in \mathbb Z$  and  $\vec a = (a_1, a_2, \ldots, a_n) \in  \{ 0,1,2\}^n$. Hereafter,  a  dyadic grid is the family $\mathcal D _{\vec a} \equiv \cup_{k\in \mathbb Z} \mathcal D_{k,\vec a} $ for $\vec a \in \{0,1,2\}^n$.

The following result is an important property of the dyadic grids to prove the equivalent norm of the Bourgain-Morrey-Lorentz space.
\begin{lemma}[Lemma 2.8, \cite{HNSH23}] \label{cube be covered}
	For any cube $Q$ there exists $R\in \bigcup_{\vec a \in \{0,1,2\}^n}  \mathcal D _{\vec a}$ such that $Q \subset R$  and $|R| \le 6^n |Q|$.
\end{lemma}

Fix the dyadic grid $\mathfrak{D} \in \{  \mathcal D _{\vec a} : \vec a \in \{0,1,2\}^n \} $. Then we define
\begin{equation*}
	\| f \|_{  M_{p,q}^{t,r} (\mathfrak{D}) } : = \left\|\left\{|Q|^{1/t-1/p}  \| f\|_{L^{p,q} (Q) }  \right\}_{Q \in \mathfrak{D} }    \right\|_{\ell^r}
\end{equation*}
for all $f \in L^0 (\rn)$.
\begin{lemma} \label{D equivalence}
		Let $ 0 < q \le \infty$. Let $ 0 <p <t<r <\infty $ or $ 0 <p \le t < r =\infty $. 
		Then $ \|\cdot \|_{M_{p,q}^{t,r} (\rn)  }$  and $ \|\cdot \|_{M_{p,q}^{t,r} (\mathfrak{D})  } $ are equivalent for any dyadic grid $\mathfrak{D} \in \{  \mathcal D _{\vec a} : \vec a \in \{0,1,2\}^n \} $. That is, for each $f \in L^0 (\rn)$, $ \|f\|_{M_{p,q}^{t,r} (\rn)  }   \approx   \|f\|_{M_{p,q}^{t,r} (\mathfrak{D})  }  $.
\end{lemma}
\begin{proof}
	We use the idea from the proof of \cite[ (2.1)]{HNSH23}.
	By Lemma \ref{cube be covered}, for all cubes $Q$, we can find $R \in \bigcup_{\vec a \in \{0,1,2\}^n}  \mathcal D _{\vec a}$ such that $Q \subset R$ and $|R| \le 6^n|Q|$. Then
	\begin{align*}
		\| f \|_{ M_{p,q}^{t,r} (\rn) } & = \left(  \sum_{Q \in \D} |Q|^{r /t- r/p} \| f\|_{ L^{p,q} (Q)} ^r \right)^{1/r} \\
		& \le 6^{n/p-n/t} \left(  \sum_{R \in \mathfrak D} |R|^{r /t- r/p} \| f\|_{ L^{p,q} (R)} ^r \right)^{1/r} =  6^{n/p-n/t} 	\| f \|_{ M_{p,q}^{t,r} (\mathfrak D) } .
	\end{align*}
Meanwhile, since $Q \subset R$ and $|R| \le 6^n |Q|$, we have $R \subset 13 Q$. Thus,
\begin{equation*}
	\| f \|_{ M_{p,q}^{t,r} (\mathfrak D) }  \lesssim \| f \|_{ M_{p,q}^{t,r} (\rn) }.
\end{equation*}
Thus the proof is finished.
\end{proof}

Next  Bourgain-Morrey-Lorentz spaces $M_{p,q}^{t,r} (\rn) $ have the same  dilation property as Bourgain-Morrey spaces $M_{p}^{t,r} (\rn) = M_{p,p}^{t,r} (\rn) $ and Lebesgue spaces $L^t (\rn)$.
\begin{proposition}
	Let $ 0 < q \le \infty$. Let $ 0 <p \le t< \infty $ and $ 0 < r  \le \infty $.
	For $\epsilon >0$, we have	$ \| f (\epsilon\cdot ) \|_{ M_{p,q}^{t,r} (\rn) } \approx \epsilon^{- n/t} \| f \|_{ M_{p,q}^{t,r} (\rn) }  $.	
\end{proposition}

\begin{proof}
	The proof is similar to \cite[Lemma 2.4]{HNSH23}. We omit the  details here.
\end{proof}

%\subsection{Nontriviality of function spaces}
Below, we investigate the necessary and sufficient conditions for the Bourgain-Morrey-Lorentz space $ M_{p,q}^{t,r} (\rn)$ to be nontrivial.  That is $M_{p,q}^{t,r} (\rn)  \neq \{0\} $.
 We first consider the following example.
\begin{example} \label{example chi_Q}
	Let $ 0 < q \le \infty$. 	Let $0 < p \le t <\infty$  and $0 <r \le \infty$.
	Let $ f = \chi_{ Q_{j,0}}$. Then $f \in M_{p,q}^{t,r} (\rn)$  if and only if $ 0 <p <t<r <\infty $ or $ 0 <p \le t < r =\infty $. 
%	Furthermore,  
%	\begin{equation*}
%		\| f\|_{L^{p,q} ( Q_{j,0})} = \begin{cases}
%			p^{1/q} \left(   \int_0^1   \left( 2^{-jn /p } s \right)^q \frac{ \d s}{s}     \right)^{1/q} = 2^{-jn/p}  \big(  \frac{p}{q} \big) ^{1/q} , & \operatorname{if} \; q <\infty , \\
%			\sup_{1>s>0}  s  2^{-jn/p} =  2^{-jn/p}, & \operatorname{if}\; q = \infty. 
%		\end{cases}
%	\end{equation*}
\end{example}
\begin{proof}
	We only prove $r<\infty$  since $r =\infty$ is similar.
	Let $k  \le j$. Then there is only one cube $Q_k  \in \mathcal D_k$ such that $  Q_{j,0} \subset Q_k  $. Hence
	\begin{equation*}
		\| f\|_{L^{p,q} ( Q_k)} = \begin{cases}
			p^{1/q} \left(   \int_0^1   \left( 2^{-jn /p } s \right)^q \frac{ \d s}{s}     \right)^{1/q} = 2^{-jn/p}  \big(  \frac{p}{q} \big) ^{1/q} , & \operatorname{if} \; q <\infty , \\
			\sup_{1>s>0}  s  2^{-jn/p} =  2^{-jn/p}, & \operatorname{if}\; q = \infty. 
		\end{cases}
	\end{equation*}	
	Let $  k >j $. Then there exists $2^{ (k-j)n }$ cubes $Q_k^i \in \mathcal D_k$ such that $ Q_{j,0} = \cup_{i=1}^{2^{ (k-j)n } } Q_k^i $. Hence
	\begin{equation*}
		\| f\|_{L^{p,q} ( Q_k^i)} = \begin{cases}
			p^{1/q} \left(   \int_0^1   \left( 2^{-kn /p } s \right)^q \frac{ \d s}{s}     \right)^{1/q} = 2^{-kn/p}  \big(  \frac{p}{q} \big) ^{1/q} , & \operatorname{if} \; q <\infty , \\
			\sup_{1>s>0}  s  2^{-kn/p} =  2^{-kn/p}, & \operatorname{if}\; q = \infty. 
		\end{cases}
	\end{equation*}
	Let 
	\begin{equation*}
		C_{p,q} = \begin{cases}
			\big(  \frac{p}{q} \big) ^{1/q} , & \operatorname{if} \; q <\infty , \\
			1,	& \operatorname{if}\; q = \infty. 
		\end{cases}
	\end{equation*}
	Then
	\begin{align*}
		\| f \|_{M_{p,q}^{t,r} (\rn)   }^r 
		& = C_{p,q} \left(  \sum_{ k \le j } 2^{-kn (r/t-r/p)} 2^{-jnr/p}   +  \sum_{ k> j} 2^{-kn (r/t-r/p)} 2^{(k-j)n }  2^{-knr/p} \right) .
	\end{align*}
	Then $ \| f \|_{M_{p,q}^{t,r} (\rn)   }  <\infty$  if and only if $0< p <t <r <\infty$. And when $0< p <t <r <\infty$,
	$ \| f \|_{M_{p,q}^{t,r} (\rn)   }  \approx  2^{-jn/t} = \ell( Q_{j,0}) ^{n/t} $ where the implicit constant depends only on $ p,q,t,r,n $. 
%	Case $r=\infty$. Then $M_{p,q}^{t,\infty} (\rn) $ is the Morrey-Lorentz space. The proof is similar to the Bourgain-Morrey-Lorentz space and	
%	\begin{equation*}
%		\| f \|_{M_{p,q}^{t,\infty} (\rn)   } \approx 2^{-jn(1/t - 1/p) }  2^{-jn/p} = 2^{-jn/t} = \ell( Q_{j,0}) ^{n/t}.
%	\end{equation*}
%	if $  0<p \le t < r=\infty$.
\end{proof}

\begin{theorem}
	Let $ 0 < q \le \infty$.  
	Let $0 < p \le t <\infty$  and $0 <r \le \infty$.
	 Then $M_{p,q}^{t,r} (\rn)$  is not trivial if and only if $0< p <t <r <\infty$ or  $  0<p \le t < r=\infty$.
\end{theorem}
\begin{proof}
	We use the idea from \cite[Theorem 2.10]{HNSH23}.
	 Example \ref{example chi_Q} implies  `if' part. Thus, let us prove the `only if' part. Since
	\begin{equation*}
		\| |f|^u \|_{ M_{p,q}^{t,r} (\rn) }  = \left( \|f\|_{  M_{pu,qu}^{tu,ru} (\rn)  } \right) ^u
	\end{equation*}
for all $f \in L^0 (\rn)$, we can assume that $  1 <p \le t<\infty  $ and $1 < r <\infty $. In this case, $M_{p,q}^{t,r} (\rn) $ is a Banach space. Let $f  \in M_{p,q}^{t,r} (\rn) \backslash \{0\}$. By replacing $f$ with $|f|$, we can assume that $f$ is non-negative. By  replacing $f$ with $\min(1,|f|)$, we also assume that $f \in L^\infty (\rn)$. By Corollary \ref{convolution BML},  $  \chi_{ [0,1]^n } * f  \in M_{p,q}^{t,r} (\rn) \backslash \{0\}$. Since $ \chi_{ [0,1]^n } * f$  is a non-negative non-zero continuous function, there exists $x_0 \in \rn $, $\epsilon >0$ and $R >0 $  such that $  \chi_{ [0,1]^n } * f  \ge \epsilon \chi_{B(x_0 ,R) }$. Thus $  \chi_{B(x_0 ,R)} \in  M_{p,q}^{t,r} (\rn)$. Since $  M_{p,q}^{t,r} (\rn)$  is invariant under translation and dilation, we have $ \chi_{ [0,1]^n } \in  M_{p,q}^{t,r} (\rn) $. Thus,  from Example \ref{example chi_Q}, we obtain $0< p <t <r <\infty$ or  $  0<p \le t < r=\infty$.
\end{proof}

Next we get a geometric property of the space $M_{p,q}^{t,r} (\rn)$.
\begin{theorem} \label{tau y f +g le}
		Let $ 0 < q \le \infty$.  Let $0< p <t <r \le \infty$. Suppose that we have $f, g \in M_{p,q}^{t,r} (\rn)$ supported on a dyadic cube $Q \in \D$.
		
		{\rm (i)} If $r =\infty$, then for all $\epsilon >0$, there exists $ y \in \rn$ such that 
		\begin{align*}
			\| f(\cdot -y) +g \|_{M_{p,q}^{t,\infty} (\rn)  }  
			& \le \max \left( \| f\|_{M_{p,q}^{t,\infty} (\rn)} ,  \| g\|_{M_{p,q}^{t,\infty} (\rn)} \right) \\
			& \le(1+\epsilon) \max \left( \| f\|_{M_{p,q}^{t,\infty} (\rn)} ,  \| g\|_{M_{p,q}^{t,\infty} (\rn)}  \right)  .
		\end{align*}
	
		{\rm (ii)} If $r <\infty$, then for all $\epsilon >0$, there exists $ y \in \rn$ such that 
	\begin{equation*}
		\| f(\cdot -y) +g \|_{M_{p,q}^{t,r} (\rn)  }  \le (1+\epsilon) \left( \| f\|_{M_{p,q}^{t,r} (\rn)} ^r + \| g\|_{M_{p,q}^{t,r} (\rn)} ^r \right)^{1/r} .
	\end{equation*}
	 \end{theorem}

\begin{proof}
	We use the idea from \cite[Theorem 2.14]{HNSH23}.
	We first recall the inequality from \cite[(1.4.9)]{G14}. For all $0<p,q \le\infty$,
	\begin{equation*}
		\|f+g\|_{L^{p,q} (X)} \le c_{p,q} ( 	\|f\|_{L^{p,q} (X)}  + 	\|g\|_{L^{p,q} (X)}  )
	\end{equation*}
with $  c_{p,q} = 2^{1/p} \max(1, 2^{(1-q) /q})$.
	Let $y \in \rn $ be such that $y+Q \in \D$ and $Q \cap (y+Q) = \emptyset$.
	
	(i) We write out the norm in full:
	\begin{align*}
		&	\| f(\cdot -y) +g \|_{M_{p,q}^{t,\infty} (\rn)  } \\
	&	= \sup_{R\in \D:  R \cap ( Q \cup (y+Q)}  |R|^{1/t-1/p}  \|f(\cdot -y) +g  \|_{ L^{p,q}  (  R \cap ( Q \cup (y+Q) )      ) }  \\
	& \le \max \bigg( \| f\|_{M_{p,q}^{t,\infty} (\rn)} ,  \| g\|_{M_{p,q}^{t,\infty} (\rn)},  \\
	& \quad  \sup_{R\in \D:  R \cap ( Q \cup (y+Q)}  |R|^{1/t-1/p} c_{p,q} \left(  \|f(\cdot -y)   \|_{ L^{p,q}  (  y+Q )      ) } +   \|g  \|_{ L^{p,q}  (  Q )      ) } \right)      \bigg) \\
	& \le \max \left( \| f\|_{M_{p,q}^{t,\infty} (\rn)} ,  \| g\|_{M_{p,q}^{t,\infty} (\rn)}\right)    \sup_{R\in \D:  R \cap ( Q \cup (y+Q) )  \neq \emptyset }  
	\max \left( 1,  2 c_{p,q} \frac{|R|^{1/t-1/p}}{|Q|^{1/t-1/p}} \right)  
   \\
   & \le \max \left( \| f\|_{M_{p,q}^{t,\infty} (\rn)} ,  \| g\|_{M_{p,q}^{t,\infty} (\rn)}\right)    \sup_{R\in \D:  R \cap ( Q \cup (y+Q) )  \neq \emptyset }  
   \max \left( 1, 2^{1/n} c_{p,q} ^{1/n} \frac{(\ell(Q) +|y|)^{1/t-1/p}}{(\ell(Q))^{1/t-1/p}} \right) ^{n} .
	\end{align*}
Since $p<t$, choose $|y|$  sufficiently large such that 
\begin{equation*}
	\left( 2^{1/n }c_{p,q} ^{1/n} \frac{(\ell(Q) +|y|)^{1/t-1/p}}{(\ell(Q))^{1/t-1/p}} \right) ^{n} \le 1.
\end{equation*}
Thus, we obtain the desired result.

(ii) This is similar to (i). 
We write out the norm in full as
	\begin{align*}
	\| f(\cdot -y) +g \|_{M_{p,q}^{t,r} (\rn)  }^r 
	&	=\sum_{R\in \D:  R \cap ( Q \cup (y+Q)}  |R|^{r/t-r/p}  \|f(\cdot -y) +g  \|_{ L^{p,q}  (  R \cap ( Q \cup (y+Q) )  ) }^r \\
= &  \sum_{R\in \D:  R \subset Q} |R|^{r/t-r/p}  \| g  \|_{ L^{p,q}  (  R  ) }^r   
 + \sum_{R\in \D:  R \subset (y+Q)}|R|^{r/t-r/p}  \|f \|_{ L^{p,q}  (  R   ) }^r   \\
& \quad +  \sum_{R\in \D:  ( Q \cup (y+Q) \subset R} |R|^{r/t-r/p}  \|f(\cdot -y) +g  \|_{ L^{p,q}  (  R \cap ( Q \cup (y+Q) )  ) }^r  \\
& \le 	\| f \|_{M_{p,q}^{t,r} (\rn)  } ^r + \| g \|_{M_{p,q}^{t,r} (\rn)  } ^r  \\
& \quad +  \sum_{R\in \D:  ( Q \cup (y+Q) \subset R}  |R|^{r/t-r/p}  c_{p,q}^r \left( \|f\|_{ L^{p,q}  (  Q) }  +   \|g \|_{ L^{p,q}  (  Q ) }  \right)^r \\
& \le 	\| f \|_{M_{p,q}^{t,r} (\rn)  } ^r + \| g \|_{M_{p,q}^{t,r} (\rn)  } ^r  \\
& \quad +  \sum_{R\in \D:  ( Q \cup (y+Q) \subset R}  |R|^{r/t-r/p}  c_{p,q}^r 2^{1/r'}  \left( \|f\|_{ L^{p,q}  (  Q) }^r  +   \|g \|_{ L^{p,q}  (  Q ) } ^r\right)  \\
& \le \left( 	\| f \|_{M_{p,q}^{t,r} (\rn)  } ^r + \| g \|_{M_{p,q}^{t,r} (\rn)  } ^r \right)  \\
& \quad \times \left( 1 + \sum_{R\in \D:  ( Q \cup (y+Q) \subset R}  |R|^{r/t-r/p}  c_{p,q}^r 2^{1/r'}  |Q|^{r/p-r/t} \right) .
\end{align*}
Next, we estimate $ I:= \sum_{R\in \D:  ( Q \cup (y+Q) \subset R}  |R|^{r/t-r/p}  c_{p,q}^r 2^{1/r'}  |Q|^{r/p-r/t} $.
Let $ m = \lfloor\log_2  (|y|+ \ell(Q)) \rfloor +1$ where $ \lfloor a \rfloor $ is the integer part of real number $a$.
Note that there at most one cube $R \in \D _k$ such that $ ( Q \cup (y+Q) \subset R $ for $ k \le m$.
Since $t >p$ and $\ell(R) \ge |y|+ \ell(Q)$, we have
\begin{align*}
	I & \le c_{p,q}^r 2^{1/r'} \ell (Q) ^{nr (1/p-1/t) } \sum_{k =-\infty }^{ m} 2^{ -knr (1/t-1/p)  } \\
	& \lesssim  c_{p,q}^r 2^{1/r'} \ell (Q) ^{nr (1/p-1/t) }  \left(  |y|+ \ell(Q)  \right) ^{nr (1/t-1/p)  }  
\end{align*}
Choose $|y|$ large enough such that $I \le \epsilon '$ where $\epsilon' >0$ such that $(1 + \epsilon')^{1/r} = (1+\epsilon)  $.
Then 
\begin{align*}
 	\| f(\cdot -y) +g \|_{M_{p,q}^{t,r} (\rn)  } \le  \left( 	\| f \|_{M_{p,q}^{t,r} (\rn)  } ^r + \| g \|_{M_{p,q}^{t,r} (\rn)  } ^r \right)^{1/r} (1 + \epsilon) .
\end{align*}
So, we complete the proof.
\end{proof}
An important geometric property of the space $M_{p,q}^{t,r} (\rn)$ is that it includes $\ell^r$.
\begin{theorem} \label{ell r M}
	Let $ 0 < q \le \infty$.  Let $0< p <t <r \le \infty$. Then there exists a linear mapping $\Phi:\ell^r \to M_{p,q}^{t,r} (\rn) $ such that 
	\begin{equation*}
		C^{-1} \|a \|_{\ell^r} \le \| \Phi(a) \|_{ M_{p,q}^{t,r} (\rn)}  \le C\|a \|_{\ell^r} 
	\end{equation*}
for all $a \in \ell^r$ for some constant $C>0$ independent of $a$.
\end{theorem}

\begin{proof}
	We use the idea from \cite[Theorem 2.15]{HNSH23}.
	Let $f_1 = \chi_{Q_{0,1}}$. Then $ \|f_1\|_{M_{p,q}^{t,r} (\rn)  }  \approx |Q_{0,1}|^{1/t} =1$. We inductively construct a sequence $\{ f_j\}_{j\ge 1}$ such that
	\begin{equation} \label{induce}
		\left(  \|f_1 +f_2 + \cdots + f_j \|_{M_{p,q}^{t,r} (\rn)}     \right)^r \le \prod_{l = 1}^j  (1+ 2^{-l}) \times \sum_{l=1}^j  ( \|f_l \|_{M_{p,q}^{t,r} (\rn)}    )^r
	\end{equation}
and that each $f_j$ is the indicator function of the cube with volume $1$. Once this is achieved, let 
\begin{equation*}
	\Phi (  \{ a_j \}_{j\ge 1}  ) = \sum_{j\ge 1} a_j f_j.
\end{equation*}
Then  $\Phi$ is the desired function since it is obvious that $ \|a\|_{\ell^r}  \lesssim  \| \Phi (a) \|_{M_{p,q}^{t,r} (\rn)} $  and that  $\prod_{l = 1}^\infty  (1+ 2^{-l})$  is finite.

Suppose that we construct $f_1, \ldots, f_j$ such that (\ref{induce}) is true.  By Theorem \ref{tau y f +g le}, we can select $f_{j+1}$, which is the indicator function of the cube of volume $1$, such  that
\begin{align*}
		& \left(  \|f_1 +f_2 + \cdots + f_j +f_{j+1} \|_{M_{p,q}^{t,r} (\rn)}     \right)^r \\
		&\le  (1+ 2^{-j-1}) \left( 	 \|f_1 +f_2 + \cdots + f_j  \|_{M_{p,q}^{t,r} (\rn)}  ^r +  \| f_{j+1} \|_{M_{p,q}^{t,r} (\rn)}  ^r   \right) 
\end{align*}
If we use the induction assumption, then a function  $f_{j+1}$ satisfying (\ref{induce}) holds for $j+1$.
\end{proof}

Then we get that the  Lorentz space $L^{p,w} (\rn)$ is different with the Bourgain-Morrey-Lorentz space $M_{p,q}^{t,r} (\rn) $. 
We remark here that  the  Lorentz space $L^{p,w} (\rn)$ differs the Bourgain-Morrey space $M_{p,p}^{t,r} (\rn) $   in \cite[Corollary 2.16]{HNSH23}. 

\begin{corollary}
	Let $ 0<q<\infty $  and $ 0< w <\infty $. Let $0< p <t <r \le \infty$. Then $L^{p,w} (\rn)$  and $M_{p,q}^{t,r} (\rn)$ are not isomorphic.
More precisely, $M_{p,q}^{t,r} (\rn) $ can not be embedded into $L^{p,w} (\rn)$. 
\end{corollary}

\begin{proof}
For the function $\Phi$ constructed in the proof of Theorem \ref{ell r M}, for any sequence $a$ such that $a_j = 0 $ or $1$ for all $j \in \mathbb N$, we have
\begin{align*}
	\| \Phi (a) \|_{ L^{p,w} (\rn) } & =\left\|  \sum_{j=1}^\infty a_j f_j \right\|_{L^{p,w} (\rn) } = \left\|  \chi_{ [0,1)^{n-1}  \times [0, \|a\|_{\ell^1} ] } \right\|_{L^{p,w} (\rn) } \\
	& = C \left| [0,1)^{n-1}  \times [0, \|a\|_{\ell^1} ]  \right|^{1/p}  =C \| a\|_{\ell^p}
\end{align*}
for some constant $C >0$ independent of $a$. But by Theorem \ref{ell r M}, $ \|a\|_{\ell^r} \approx \| \Phi (a) \|_{ M_{p,q}^{t,r} (\rn) }  $. Thus, $ L^{p,w} (\rn) $  and  $M_{p,q}^{t,r} (\rn) $  are not isomorphic.
\end{proof}

\section{Block spaces} \label{block space}
In this section, we introduce the block space $\mathcal H_{ p', q'}^{t',r'} (\rn)$, which is the predual of the  Bourgain-Morrey-Lorentz space  $M_{p,q}^{t,r} (\rn)$.
Given a quasi-Banach space $Y$  with norm $\|\cdot\|_{Y}$, its dual $Y^\ast$  is the space of all continuous linear functionals $T$ on $Y$ equipped with the norm 
\begin{equation*}
	\|T\|_{Y^\ast } = \sup_{ \|x\|_Y  \le 1}  |T(x)|.
\end{equation*}
Note that the dual of a quasi-Banach space is always a Banach space.
The predual of a Banach space $Y$ is actually a Banach space  $X$ such that $X^*$ is isomorphic to $Y$.

Next, we define  block functions.
\begin{definition}
		Let $ 1 \le q \le \infty$. Let $ 1 <p <t<r <\infty $ or $ 1 <p \le t < r =\infty $. A function $b$ is called $(p', q', t')$-block if there exists a cube $Q$ that supports $b$ such that
		\begin{equation*}
			\|b\|_{ L^{ p', q'}  }  \le |Q|^{ 1/ p' - 1/ t'}.
		\end{equation*}
	If we need to indicate $Q$, we say that $b$ is a $(p', q', t')$-block supported on $Q$.
\end{definition}
Now, we define the space $\mathcal H_{ p', q'}^{t',r'} (\rn)$ via $(p', q', t')$-block.
\begin{definition}
		Let $ 1 \le q \le \infty$. Let $ 1 <p <t<r <\infty $ or $ 1 <p \le t < r =\infty $. Define the function space $\mathcal H_{ p', q'}^{t',r'} (\rn) $ as the set of all measurable functions $f$ for which $f$ is realized as the sum
		\begin{equation} \label{f = block}
			f = \sum_{ (j,k) \in \mathbb Z ^{1+n} }  \lambda_{j,k} b_{j,k}
		\end{equation}
	with some $ \lambda = \{\lambda_{j,k} \} _{  (j,k) \in \mathbb Z ^{1+n}  }  \in \ell^{r'}  (\mathbb Z ^{1+n})$  and $ b_{j,k}$  is a  $(p', q', t')$-block supported on $ Q_{j,k}$ for each $  (j,k) \in \mathbb Z ^{1+n}$, where the converge occurs in $L^1_{\operatorname{loc}}$. The norm $\| f\|_{ \mathcal H_{ p', q'}^{t',r'} (\rn) }$ for $f \in \mathcal H_{ p', q'}^{t',r'} (\rn)$  is defined as
	\begin{equation*}
		\|f \|_{ \mathcal H_{ p', q'}^{t',r'} (\rn)} = \inf_\lambda \| \lambda \|_{\ell ^{r'}}, 
	\end{equation*}
where $ \lambda = \{\lambda_{j,k} \} _{  (j,k) \in \mathbb Z ^{1+n}  }  \in \ell^{r'}  (\mathbb Z ^{1+n})$ runs over all admissible expressions (\ref{f = block}) and $ b_{j,k}$  is a  $(p', q', t')$-block supported on $ Q_{j,k}$ for each $  (j,k) \in \mathbb Z ^{1+n}$.
\end{definition}

\begin{remark}
	The space $\mathcal H_{ p', q'}^{t',r'} (\rn) $ is also  invariant under translation. The proof is similar to Lemma \ref{translation BML} and we omit it. Together this and the definition of $\mathcal H_{ p', q'}^{t',r'} (\rn) $, we obtain that for all cube $Q$ with $\ell(Q) = R$, $ \|\chi_Q \|_{\mathcal H_{ p', q'}^{t',r'} (\rn)  }  \approx R^{n/t'} .$
	
		It is easy to see that $\mathcal H_{ p', q'}^{t',r'} (\rn) $ has the lattice property. That is, if $ |f | \le |g| $  and $g \in \mathcal H_{ p', q'}^{t',r'} (\rn)$, then $f \in \mathcal H_{ p', q'}^{t',r'} (\rn)$  and $ \|f \|_{\mathcal H_{ p', q'}^{t',r'} (\rn)} \le \|g \|_{\mathcal H_{ p', q'}^{t',r'} (\rn)}  $.
\end{remark}

\begin{lemma} \label{chi Q f in H}
	Let $ 1< p <\infty $ and $ 1\le  q \le \infty$.
	For any cube in $Q$ in $\rn$, and for $f \in L^{p'}_{\operatorname{loc}} (\rn)$, we have
	\begin{equation*}
		\left\| \chi_Q f \right\|_{ \mathcal H_{ p', q'}^{t',r'} (\rn)}  \lesssim   |Q|^{1/t' -1/p'}  \| f \|_{L^{p',q'}  (Q )}.
	\end{equation*}
\end{lemma}
\begin{proof}

	There are at most $2^n$ disjoint cubes $Q_i \in \D$ such that $Q \subset \cup_{i=1}^{2^n} Q_i$ and $Q_i \subset 2Q $ for all $i \in \{ 1,2, \ldots , 2^n\}$. Then $|Q_i| \approx |Q| $  and $\chi_Q f= \sum_{i=1}^{2^n} \chi_{Q_i \cap Q } f $. 
	
	 For each $i$, let 
	\begin{equation*}
		b_i (x) =
		\begin{cases}
			\frac{\chi_{Q_i \cap Q} (x) f (x) }{|Q_i|^{1/t' -1/p'}  \| f \|_{L^{p',q'}  (Q_i \cap Q  )} }, & \operatorname{if} \; \| f \|_{L^{p',q'}  (Q_i\cap Q )} \neq 0, \\
			0, & \operatorname{else.}
		\end{cases}	 
	\end{equation*}
Then $b_i$ is a $(p', q', t')$-block supported on $Q_i$. Let $ \lambda_i =|Q_i|^{1/t' -1/p'}  \| f \|_{L^{p',q'}  (Q_i\cap Q )} $. Then $  \chi_Q f =  \sum_{i=1}^{2^n} \lambda_i b_i.$
Thus, 
\begin{equation*}
	\left\| \chi_Q f \right\|_{ \mathcal H_{ p', q'}^{t',r'} (\rn)} \le  \left(  \sum_{i=1}^{2^n} \lambda_i ^{r'} \right)^{1/r'}  \approx \sup_{i \in \{ 1,2, \ldots , 2^n\} } \lambda_i \lesssim   |Q|^{1/t' -1/p'}  \| f \|_{L^{p',q'}  (Q )}.
\end{equation*}
We obtain the desired inequality.
\end{proof}

%\begin{proof}
%	Decompose $\rn$ into $2^n$ quadrants $R_1, R_2,\ldots, R_{2^n}$.
%	
%	Let $ Q_j = Q (0, 2^j) $  be the cube center at zero and the side length $2 \times 2^j $ where $ j \in \mathbb N$. Let $R_1 =  Q _1$ and for $j = 2,3,\ldots$, let $R_j = Q_j \backslash Q_{j-1}$.  Then $R_1 = \cup_{i=1}^{2^n} P_1^i$ with $P_1^i \in \mathcal D _{-1} $. And $R_2 = \cup_{i=1}^{2^{2n} - 2^n} P_1^i$ with $P_1^i \in \mathcal D _{-1} $.
%	$R_3 = $
%	
%
%\end{proof}
In \cite[Theorem 1.2]{DK24},  Dao and Krantz obtained the predual of Morrey-Lorentz space.
Now, we state a main result of this section that 
  the predual of  $M_{p,q}^{t,r} (\rn)$ is  the block space $\mathcal H_{ p', q'}^{t',r'} (\rn) $. 

\begin{theorem}\label{M is predual of H}
	Let $1 <q <\infty$.  Let $ 1 <p <t<r <\infty $ or $ 1 <p \le t < r =\infty $. 
	Then, we have
	\begin{equation} \label{M = H '}
		M_{p,q}^{t,r} (\rn) = \left( \mathcal H_{ p', q'}^{t',r'} (\rn)   \right) ^\ast.
	\end{equation}
%Then   $M_{p,q}^{t,r} (\rn) $  is  the predual space of $\mathcal H_{ p', q'}^{t',r'} (\rn) $ in the following sense.
Here (\ref{M = H '}) is  understood in the following sense.

{\rm (i)} Let $f\in M_{p,q}^{t,r} (\rn) $. Then for any   $g \in  \mathcal H_{ p', q'}^{t',r'} (\rn) $,  we have $f g \in L^1 (\rn)$  and the mapping
\begin{equation*}
	 \mathcal H_{ p', q'}^{t',r'} (\rn) \ni g \mapsto \int_\rn f(x) g(x) \d x \in \mathbb C
\end{equation*}
defines a continuous linear functional $L_f$  on $  \mathcal H_{ p', q'}^{t',r'} (\rn)$.

{\rm (ii)} Conversely, any continuous linear functional $L $ on $ \mathcal H_{ p', q'}^{t',r'} (\rn)$, can be realized as $L = L_f | \mathcal H_{ p', q'}^{t',r'} (\rn)  $ with a certain $f \in M_{p,q}^{t,r} (\rn)$. 

{\rm (iii)} The correspondence
\begin{equation*}
	\tau : M_{p,q}^{t,r} (\rn) \ni f \mapsto L_f \in \left( \mathcal H_{ p', q'}^{t',r'} (\rn) \right)^*
\end{equation*}
is an isomorphism. Furthermore,
\begin{equation} \label{f M = sup}
	\| f\|_{ M_{p,q}^{t,r} (\rn) }  \approx \sup_{ \|k\|_{ \mathcal H_{ p', q'}^{t',r'} (\rn) } =1 }  \left|  \int_\rn f(x) k (x) \d x \right|
\end{equation}
for $f \in M_{p,q}^{t,r} (\rn) $ and 
\begin{equation} \label{g H = sup}
	\| g \|_{ \mathcal H_{ p', q'}^{t',r'} (\rn)  }  \approx \sup_{ \|k\|_{ M_{p,q}^{t,r} (\rn) } =1 }  \left|  \int_\rn f(x) k (x) \d x \right|
\end{equation}
for $g \in \mathcal H_{ p', q'}^{t',r'} (\rn)$.

\end{theorem}

\begin{proof}
	(i) 
	We first prove that
	\begin{equation} \label{M hook H'}
		 M_{p,q}^{t,r} (\rn)  \hookrightarrow \left( \mathcal H_{ p', q'}^{t',r'} (\rn)   \right) ^* .
	\end{equation}
For every $f \in M_{p,q}^{t,r} (\rn)  $ and for $g \in  \mathcal H_{ p', q'}^{t',r'} (\rn) $,  by the H\"older inequality (Lemma \ref{holder lorentz}), we obtain
\begin{align*}
	\left| \int_\rn f (x) g (x) \d x \right| & \le \sum_{ (j,k) \in \mathbb Z ^{1+n} } |\lambda_{j,k} | \int_{Q_{j,k}} | f (x) b_{j,k} (x) |  \d x \\
	& \le \sum_{ (j,k) \in \mathbb Z ^{1+n} } |\lambda_{j,k} | \|f \|_{ L^{p,q} (Q_{j,k} ) }  \| b_{j,k}  \|_{ L^{p ',q'} (Q_{j,k} ) }\\
	&  \le \sum_{ (j,k) \in \mathbb Z ^{1+n} } |\lambda_{j,k} | \|f \|_{ L^{p,q} (Q_{j,k} ) }  |Q|^{ 1/ t - 1/ p} \\
	& \le \| f\|_{ M_{p,q}^{t,r} (\rn) }  \left( \sum_{ (j,k) \in \mathbb Z ^{1+n} } |\lambda_{j,k} |^{r'} \right)^{1/r'},
\end{align*}
provided that $g =\sum_{ (j,k) \in \mathbb Z ^{1+n} }  \lambda_{j,k} b_{j,k} $, 	with some $ \lambda = \{\lambda_{j,k} \} _{  (j,k) \in \mathbb Z ^{1+n}  }  \in \ell^{r'}  (\mathbb Z ^{1+n})$  and $ b_{j,k}$  is a  $(p', q', t')$-block supported on $ Q_{j,k}$ for each $  (j,k) \in \mathbb Z ^{1+n}$.
Thus, it follows from the last inequality that  
\begin{equation*}
		\left| \int_\rn f (x) g (x) \d x \right| \le  \| f\|_{ M_{p,q}^{t,r} (\rn) } \|g\|_{ \mathcal H_{ p', q'}^{t',r'} (\rn)  } .
\end{equation*}
Therefore, operator $ T_f : \mathcal H_{ p', q'}^{t',r'} (\rn) \to \mathbb C $ defined by 
\begin{equation*}
	 T_f (g) := \int_\rn f (x) g (x) \d x
\end{equation*}
is linear and continuous.

Next, define $T (f) = T_f$  for $f \in  M_{p,q}^{t,r} (\rn) $. It is clear that $T :  M_{p,q}^{t,r} (\rn) \to  \mathcal H_{ p', q'}^{t',r'} (\rn) $ is a linear operator. We claim that $T$  is injective. To obtain the result, it suffices to show that if $T (f)  =0$, then $f (x) =0$ for a.e. $x \in \rn$. We show there is a contradiction that there is $R_0>0$ such that $f (x) \neq 0$ for a.e. $x \in Q(0,R_0)$ where $Q(0,R_0)$  is the cube center at zero with side length $ 2 R_0$. 

Since $f\in M_{p,q}^{t,r} (\rn)$, then we have $f \in L^{p,q}( Q(0,R_0) )$. By duality, there exists a function $\bar g \in   L^{p',q'}( Q(0,R_0) ) $, $\bar g \neq 0$  such that 
\begin{equation} \label{0 < int fg}
	0 < \| f\|_{  L^{p,q}( Q(0,R_0) } \approx \left| \int_{ Q(0,R_0) }  f(x ) \bar g (x)  \d x     \right|.
\end{equation}
Let 
\begin{equation*}
	g(x) =  | Q(0,R_0) |^{1/p' - 1/t'}  \frac{\bar g (x)\chi_{Q(0,R_0)}  (x)}{ \| \bar g \|_{  L^{p',q'}( Q(0,R_0) )  } }.
\end{equation*}
It is clear that $g$  is multiple of  $(p', q', t')$-block supported on $  Q(0,R_0)$. (The cube $Q(0,R_0)$ is not a dyadic cube in $\mathcal D$, but it can be covered by at most $2^n$  disjoint dyadic cubes). By this fact and (\ref{0 < int fg}), we get
\begin{equation*}
	 | T(f) (g)| =c  \left|  \int_{ Q(0,R_0) }  f(x ) \bar g (x)  \d x     \right| >0,
\end{equation*}
which contradicts $T(f) =0 $ in $ \left( \mathcal H_{ p', q'}^{t',r'} (\rn)   \right)^\ast $.
This implies that linear operator $ T : M_{p,q}^{t,r} (\rn) \to \left( \mathcal H_{ p', q'}^{t',r'} (\rn)   \right) ^\ast $ is injective. Thus, we obtain (\ref{M hook H'}).

(ii) It remains to show that 
\begin{equation} \label{H' hook M}
	\left( \mathcal H_{ p', q'}^{t',r'} (\rn)   \right) ^*  \hookrightarrow  M_{p,q}^{t,r} (\rn) .
\end{equation}

Let $L  \in 	\left( \mathcal H_{ p', q'}^{t',r'} (\rn)   \right) ^* $. For
any cube $Q$,
 by Lemma \ref{chi Q f in H}, we have  function $ L \chi_Q \in L^{p,q} (\rn)$.
 Indeed,
 \begin{align*}
 	\|L \chi_Q \|_{ L^{p,q} (\rn) } & \approx \sup_{ \|g\|_{L^{p',q'} (\rn) }  =1}  \left| \int_\rn  L \chi_Q g \d x  \right| \\
 	& \le \|L\|_{  \mathcal H_{ p', q'}^{t',r'} (\rn)  \to \mathbb C }  \|\chi_Q g \|_{  \mathcal H_{ p', q'}^{t',r'} (\rn)  }  \\
 	& \lesssim \|L\|_{  \mathcal H_{ p', q'}^{t',r'} (\rn)  \to \mathbb C }  |Q|^{1/t' -1/p'}   <\infty.
 \end{align*}
  Thus, by the duality, there exists $f_Q \in   L^{p,q} (\rn) $  supported on $Q$ such that 
\begin{equation} \label{L Q fg}
	\langle  L \chi_{Q } , g \rangle_{L^{p,q} (\rn)  ,L^{p',q'} (\rn)}  =\int_\rn f_Q (x) \overline{g(x)} \d x,\; \forall g \in L^{p',q'} (\rn) .
\end{equation}
Let $\rn = \bigcup_{k\ge 1} Q_k $, with $Q_k  \Subset Q_{k+1}$ for all $k\ge 1$. (we use the notation $A \Subset B$ to denote that the closure of $A \subset \rn$ is compact and contained in $B\subset \rn$). Then we define $f (x) = f_{Q_k}$ if $x\in Q_k$, which make sense by (\ref{L Q fg}). This implies that  $f_{Q_k}(x) = f_{Q_{k+1}} (x)$ for a.e. $x\in Q_k$.
Thus, it suffices to prove that $f \in M_{p,q}^{t,r} (\rn)$.

%For each $(j,k) \in \mathbb Z^{1+n}$, the mapping
%\begin{equation} \label{g in Lt'}
%	L^{t'} (\rn)  \ni g \mapsto  L_{j,k}  (g) \equiv L (g \chi_{Q _{j,k} } )  \in \mathbb C
%\end{equation}  
%is a bounded linear functional (since $g \chi_{Q _{j,k}}  \in L^{t'}(\rn) \cap L^{p'}(\rn) $, by Lemma \ref{chi Q f in H}, $ g \chi_{Q _{j,k}} \in H_{ p', q'}^{t',r'} (\rn) $). Since $t'  \in (1,\infty)$, we see that $ L_{j,k}$ is realized
%by an $L^t_{\operatorname{loc}} (\rn)$-function $f_{j,k}$ with supp $f_{j,k} \subset \overline {Q_{j,k}}$. Since $ L_{j_1,k_1} (g)  =  L_{j_2,k_2}(g)$ holds
%by definition for all $g \in 	L^{t'} (\rn) $ with supp $g \subset \overline {Q_{j,k}} \cap \overline {Q_{j,k}} $, one sees that there
%is a function $f \in L^t_{\operatorname{loc}} (\rn)$ such that $f_{j,k}  = f \chi_{Q_{j,k} }$ for all $(j,k) \in \mathbb Z^{1+n}$. Then 
%\begin{equation*}
%	L_{j,k}  (g) = \int_{Q_{j,k}} g (x) f(x) \d x
%\end{equation*}
%for all $g \in L_{\operatorname{loc}} ^{t'} (\rn)$, and $(j,k) \in \mathbb Z^{1+n}$. We now show that $f \in M_{p,q}^{t,r} (\rn) $. 
Fix $\epsilon >0$ and a finite set $K \subset Z^{1+n}$. For each $(j,k) \in K$, there exists $g_{j,k} \in  L^{p',q'} (Q_{j,k}) $ such that $\| g_{j,k } \|_{  L^{p',q'} (Q_{j,k}) }  \le 1$ and
\begin{equation*}
	\| f \|_{L^{p,q} (Q_{j,k}) } \lesssim \int_{ Q_{j,k} } f(x) g_{j,k} (x) \d x.
\end{equation*}
Note that $ |Q_{j,k}|^{1/t-1/p} g_{j,k}$ is a $(p',q',t')$-block supported on $Q_{j,k}$.
 Take
an arbitrary nonnegative sequence $\{ \lambda_{j,k}\} \in \ell^{r'} (\mathbb Z ^{1+n})$ supported on $K$ and set
\begin{equation}\label{g_K decomposition}
	g_{K} = \sum_{ (j,k) \in K}\lambda_{j,k} |Q_{j,k}|^{1/t-1/p} g_{j,k}  \in  \mathcal H_{ p', q'}^{t',r'} (\rn) .
\end{equation}
We get 
\begin{align*}
	\sum_{ (j,k) \in K} \lambda_{j,k}  |Q_{j,k}|^{1/t-1/p}  	\| f \|_{L^{p,q} (Q_{j,k}) }   &\lesssim   \int_\rn |f(x) 	g_{K,\epsilon} (x) | \d x =   c L(h),
\end{align*}
where $h =g_{K} (x)  \overline{ \sgn  ( f(x))}$. Further, by the decomposition (\ref{g_K decomposition}), we obtain
\begin{equation*}
	| L(h)|  \le \|L \|_{ \mathcal H_{ p', q'}^{t',r'} (\rn)  \to \mathbb C}  \|h\|_{\mathcal H_{ p', q'}^{t',r'} (\rn) }  \le \|L \|_{ \mathcal H_{ p', q'}^{t',r'} (\rn)  \to \mathbb C}  \| \{\lambda_{j,k}  \}_{(j,k) \in K} \|_{\ell^{r'} (\mathbb Z^{1+n})}.
\end{equation*}
Since $r > 1$ and since $K$ and $\{ \lambda_{j,k} \}$ are arbitrary, we conclude that
\begin{align*}
	\| f\|_{M_{p,q}^{t,r} (\rn) } & = \sup_{\| \{\lambda_{j,k}  \}_{(j,k) \in K} \|_{\ell^{r'} (\mathbb Z^{1+n})} \le 1  }  \sum_{ (j,k) \in K} \lambda_{j,k}  |Q_{j,k}|^{1/t-1/p}  	\| f \|_{L^{p,q} (Q_{j,k}) }  \\
	& \le c  \|L \|_{ \mathcal H_{ p', q'}^{t',r'} (\rn)  \to \mathbb C} .
\end{align*}
%Let $\epsilon \downarrow 0$, we obtain
%\begin{equation*}
%		\| f\|_{M_{p,q}^{t,r} (\rn) }  \le  \|L \|_{ \mathcal H_{ p', q'}^{t',r'} (\rn)  \to \mathbb C}.
%\end{equation*}
%which implies (\ref{H' hook M}). 
%Furthermore, $L = L_f | \mathcal H_{ p', q'}^{t',r'} (\rn)  $ holds.

The proof of (\ref{f M = sup}) and (\ref{g H = sup}) is included in what we have proven. (The proof of  (\ref{g H = sup}) can  be obtained by using \cite[Theorem 87, existence of the norm attainer]{SFH20}). Thus we finish the proof.
\end{proof}

\begin{remark}
		Let $1 <q <\infty$.  Let $ 1 <p <t<r <\infty $ or $ 1 <p \le t < r =\infty $. 
		If $ q =p $, then Theorem \ref{M is predual of H} becomes \cite[Theorem 2.17]{M16}.
		If $r =\infty$, then Theorem \ref{M is predual of H} becomes \cite[Theorem 1.2]{DK24}. We do not consider $ q= 1$ or $q =\infty$ since we want to use the duality $  ( L^{p,q} (X) )^*  = L^{p',q'} (X) $.
\end{remark}

Below, we turn to  the Fatou property. We first recall the definition of quasi-Banach latties, which will be used in proving the boundedness of powered Hardy-Littlewood maximal operator on block spces $\mathcal H_{ p', q'}^{t',r'} (\rn)$.

\begin{definition}
	Let $L^0(\Omega, \mu)$ be the space of measurable complex-valued functions on a measure space 
	$(\Omega,  \mu)$. We say that $X (\Omega, \mu ) \subset L^0(\Omega, \mu) $  is a quasi-normed space
	if it is a quasi-normed space and satisfies the following lattice property: if $g \in L^0(\Omega, \mu)$, $|g|\le |f|$ and $ f \in X (\Omega, \mu )$, then  $g \in X (\Omega, \mu )$.
	
	A quasi-Banach lattice  $X (\Omega, \mu )$ has the Fatou property with constant $C_{\mathcal F} >0$ if $0 \le f_n \uparrow f$ for a sequence of functions $\{ f_n\}_{n=1}^\infty  \subset X (\Omega, \mu )$  and $\sup_{n\ge 1} \| f_n\|_{X} <\infty$, then $f\in X (\Omega, \mu )$ and $ \|f \|_X \le C_{\mathcal F} \sup_{n\ge 1} \| f_n\|_{X}  $.

	The Fatou property is stronger than completeness. The proof that the Fatou property of a quasi-normed lattice indeed implies its completeness can be seen in \cite[Remark 2.1(ii)]{LN24}.
\end{definition} 
\begin{theorem} \label{Fatou block}
	Let $1 < q < \infty $.
	Let $1 < p < t < r < \infty $ or $ 1 < p \le  t < r = \infty.$ Suppose that $f$ and $\{f_m\}_{m\in \mathbb N}$ are
	nonnegative, and $\|f_m\|_{ \mathcal H_{ p', q'}^{t',r'} (\rn) } \le 1$ and $f_m \uparrow f$ a.e.
 Then $f \in \mathcal H_{ p', q'}^{t',r'} (\rn)$  and $\| f\|_{\mathcal H_{ p', q'}^{t',r'} (\rn)} \le 1$.
\end{theorem}
\begin{proof}
	The case $1 < q < \infty $ and $ 1 < p \le  t < r = \infty$ is in \cite[Theorem 6.1]{DK24}. Hence we only need to consider the case $1 < q < \infty $ and $1 < p < t < r < \infty $. We use the ideas from the proof of  \cite[Theorem 1.2]{ST15}.
	We  decompose $f_m$ as 
	\begin{equation*}
		f_m = \sum_{(j,k)\in\mathbb{Z}^{n+1}}\lambda_{jk}^ { (m) } b_{jk} ^{(m)} 
	\end{equation*}
	where $b_{jk} ^{(m)}  $ is a $ (p',q',t') $-block supported on $ Q_{j,k} $ and  $ \| \lambda^{(m)} \|_{\ell^{r'}}  \le  \| f_m \|_{   \mathcal H_{ p', q'}^{t',r'} (\rn) }  +\epsilon \le 1+\epsilon $.
	Since $1 < t' < p' <\infty$, let $\eta >0$  such that $ 1<t' - \eta  <  p'$.
		 Using the H\"older inequality with $p ' /(t' - \eta) > 1$,
	 \begin{align*}
	 	\|b_{jk} ^{(m)}  \|_{L^{(t' - \eta)} (\rn) }^{(t' - \eta)} & = \int_{Q_{j,k} } | b_{jk} ^{(m)} (x) |^{(t' - \eta)}  \d x \\
	 	& \le \|  | b_{jk} ^{(m)} |^{(t' - \eta)} \|_{ L^{p ' /(t' - \eta), q'/(t' - \eta) }  (Q_{j,k})}  \| 1 \|_{ L^{ (p ' /(t' - \eta))', (q'/ (t' - \eta))' }  (Q_{j,k})  } \\
	 	& \lesssim  \|   b_{jk} ^{(m)}  \|_{ L^{p ' , q' } (Q_{j,k})} ^{(t' - \eta)}  |Q_{j,k} |^{1 / (p ' /(t' - \eta))' } \\
	 	& = 2^{ \frac{-jnp'}{p' - t' + \eta}  } \|   b_{jk} ^{(m)}  \|_{ L^{p ' , q' } (Q_{j,k})} ^{t'}  <\infty .
	 \end{align*} 
	Using the weak$*$-compactness of the Lebesgue space $L^{t'-\eta}(Q_{j,k})$, apply a diagonalization argument and, hence, we can select an infinite subsequence $ \{ f_{m_v}\}_{v=1}^\infty   \subset \{ f_m \} _{m=1} ^\infty$ such that
	\begin{equation*}
		f_{m_v} = \sum_{(j,k)\in\mathbb{Z}^{n+1}}\lambda_{jk}^ { (m_v) } b_{jk} ^{(m_v)} ,
	\end{equation*} 
	\begin{equation}
		\lim_{v\to \infty} \lambda_{jk}^ { (m_v) } = \lambda_{jk},
	\end{equation}
	\begin{equation}
		\lim_{v\to \infty} b_{jk} ^{(m_v)}=  b_{jk}  \; \operatorname{in \; the \; weak}* \operatorname{-topology\; of} L^{t'}(Q_{j,k}).
	\end{equation}
	where $b_{jk} $ is a $(p', q',t')$-block with supp $b_{jk} \subset Q_{j,k}$.
	Let 
	\begin{equation*}
		f_0 := \sum_{(j,k)\in\mathbb{Z}^{n+1}}\lambda_{jk} b_{jk}.
	\end{equation*}
	They by the Fatou Lemma,
	\begin{equation*}
		\left( \sum_{(j,k)\in\mathbb{Z}^{n+1}}  |\lambda_{jk}|^{r'} \right)^{1/r'} \le \liminf_{v\to \infty } \left( \sum_{(j,k)\in\mathbb{Z}^{n+1}}  |\lambda_{jk} ^{ (m_v) } |^{r'} \right)^{1/r'} \le 1+\epsilon,
	\end{equation*}
	which implies that $f_0  \in   \mathcal H_{ p', q'}^{t',r'} (\rn) $.

		We claim that $f = f_0 $. Once this is achieved, we obtain $ f \in  \mathcal H_{ p', q'}^{t',r'} (\rn) $ and
	\begin{equation} \label{f liminf 1+ eps}
		\|f\|_{ \mathcal H_{ p', q'}^{t',r'} (\rn)  } \le 	\left( \sum_{(j,k)\in\mathbb{Z}^{n+1}}  |\lambda_{jk}|^{r'} \right)^{1/r'} \le \liminf_{v\to \infty } \left( \sum_{(j,k)\in\mathbb{Z}^{n+1}}  |\lambda_{jk} ^{ (m_v) } |^{r'} \right)^{1/r'} \le 1+\epsilon.
	\end{equation}
	By the Lebesgue differentiation theorem, this amounts to the proof of the equality:
	\begin{equation*}
		\int_Q f (x)\d x = \int_Q f_0 (x)\d x
	\end{equation*}
	for all $Q \in \D$. Now fix $Q_0 \in \D$. Let $ j_{Q_0} := -\log_2 ( \ell (Q_0)) $. Then $ Q_0 \in \D _{ j_{Q_0}}$.
	By the definition of $ f_{ m_ v } $ and the fact that the sum defining $ f_{ m_ v } $ converges in $L^1_{\operatorname{loc}}$, we have
	\begin{align*}
		\int_{Q_0} f_{ m_ v  } (x) \d x = \sum_{(j,k)\in\mathbb{Z}^{n+1}} \lambda_{jk}^ { (m_v) }  \int_{Q_0}   b_{jk} ^{(m_v)}  (x) \d x  = \sum_{ j \in \mathbb Z } \sum_{k \in \mathbb Z^n}  \lambda_{jk}^ { (m_v) }  \int_{Q_0}   b_{jk} ^{(m_v)}  (x) \d x .
	\end{align*}
	Note that 
	\begin{align*}
		\left|  \int_{Q_0} b_{jk} ^{(m_v)}  (x)  \d x \right| &  \le \int_{Q_0}  | b_{jk} ^{(m_v)}  (x) | \d x  \\
		& \le  \left(   \int_{Q_0}  | b_{jk} ^{(m_v)}  (x) | ^{p'} \chi_{Q_{j,k} } (x) \d x \right)^{1/p'} |Q_{j,k} \cap Q_{0} |^{1/p} \\
		& \le |Q_{j,k}|^{1/t-1/p}  |Q \cap Q_{0} |^{1/p} .
	\end{align*}
	Since  $1/t-1/p <0$  and $ 1- r/t < 0$,
	\begin{align*}
		& \left(  \sum_{j = -\infty} ^\infty \sum_{k \in \mathbb Z^n,  Q_{j,k} \cap  Q_{0} \neq \emptyset } |Q_{j,k}|^{r/t-r/p}  |Q \cap Q_{0} |^{r/p} \right)^{1/r} \\
		&= \left(\sum_{j= -\infty} ^{ j_{Q_0} } 2^{ -m r n (1/t-1/p)} 2^{ - r j_{Q_0} n/p }  +\sum_{ j =j_{Q_0}  +1 } ^\infty   2^{ - j_{Q_0} n} 2^{m n  (1 -r/t) } \right)^{1/r}  \le C <\infty.
	\end{align*}
	Thus \begin{align*}
			\left|  \int_{Q_0} b_{jk} ^{(m_v)}  (x)  \d x \right|  \le C \left( \sum_{(j,k)\in\mathbb{Z}^{n+1}}  |\lambda_{jk} ^{ (m_v) } |^{r'} \right)^{1/r'}  \le  C (1+\epsilon) <\infty.
	\end{align*}
	Now we can  use  Lebesgue's convergence theorem to  obtain
	\begin{align*}
		\lim_{v \to \infty} \int_{Q_0} b_{jk} ^{(m_v)}  (x)  \d x 
		= \sum_{j= -\infty} ^\infty \left( \lim_{v \to \infty} \sum_{k \in \mathbb Z^n,  Q_{j,k} \cap  Q_{0} \neq \emptyset  } \lambda_{jk} ^{ (m_v) }  \int_{Q_0}  b_{jk} ^{(m_v)}  (x) \d x    \right).
	\end{align*}
	Since 
	$
	\sum_{k \in \mathbb Z^n,  Q_{j,k} \cap  Q_{0} \neq \emptyset }
	$
	is the symbol of summation over a finite set for each $j \in \mathbb Z$, we have
	\begin{align*}
		& \lim_{v \to \infty} \sum_{j= -\infty} ^\infty  \sum_{k \in \mathbb Z^n,  Q_{j,k} \cap  Q_{0} \neq \emptyset  } \lambda_{jk} ^{ (m_v) }  \int_{Q_0}  b_{jk} ^{(m_v)}  (x) \d x    \\
	&	 =  \sum_{(j,k) \in \mathbb Z^{n+1}} \lambda_{jk}  \int_{Q_0}  b_{jk}  (x) \d x = \int_{Q_0} f_0 (x)\d x.
	\end{align*}
		Since $f_m \uparrow f$ a.e.,
	we obtain
	\begin{equation*}
		\int_{Q_0} f (x)\d x = \lim_{v \to \infty } \int_{Q_0} f_{m_v} (x)\d x = \lim_{v \to \infty} \sum_{(j,k) \in \mathbb Z^{n+1}} \lambda_{jk} ^{ (m_v) }  \int_{Q_0}  b_{jk} ^{(m_v)}  (x) \d x  = \int_{Q_0} f_0 (x)\d x.
	\end{equation*}
 This yields $f= f_0$ a.e., by virtue of the Lebesgue differentiation theorem, and, hence $f \in  \mathcal H_{ p', q'}^{t',r'} (\rn) $.
Letting $\epsilon \to 0^+$ in  (\ref{f liminf 1+ eps}), we obtain $  	\|f\|_{ \mathcal H_{ p', q'}^{t',r'} (\rn)  }  \le 1.$
	This completes the proof of the theorem.
\end{proof}

The following result is a characterization of the predual of the block space $\mathcal H_{ p', q'}^{t',r'} (\rn)$ in terms of the K\"othe dual. 
We first recall  the K\"othe dual can be seen in \cite[Section 9.1.5]{SFH20}.
The definition of  ball Banach function norm can be seen in \cite[Definition 14]{SFH20}.

\begin{definition}
	Let $\rho$ be a Banach function norm
	over a measure space  $ ( X, \mathcal A , \mu)$. The collection $ \mathcal X =\mathcal X (\rho) $ of all  measurable functions $f$ for which $\rho (|f|)  <\infty$ is called a Banach function space. For each $f \in  \mathcal X  $, define $ \| f\|_{ \mathcal X }  : =  \mathcal X $.
\end{definition}

\begin{definition}
	Let $f, g$ be measurable  functions on $\rn$.
	If $\rho$  is a ball Banach function norm, its ``associate norm'' $\rho '$  is defined by 
	\begin{equation*}
		\rho ' (g) := \sup \left\{  \| f g \|_{L^1}  : \rho (f) \le 1 \right\}.
	\end{equation*}
\end{definition}
\begin{definition}
	Let $\rho$ be a ball Banach function norm. Let $\mathcal X =\mathcal X (\rho)$ be the ball Banach function space determined by $\rho$. Let $\rho '$ be the  associate norm of $\rho$. The Banach function space $ \mathcal X (\rho') = \mathcal X ' (\rho)$ determined by $\rho '$  is called the associate space or the K\"othe dual of $\mathcal X $.
\end{definition}

\begin{theorem}\label{M ' = H}
		Let $1 <q <\infty$.  Let $ 1 <p <t<r <\infty $ or $ 1 <p \le t < r =\infty $. Then  the  K\"othe dual space $\left(	M_{p,q}^{t,r} (\rn)  \right)' $   coincides with the block space
	  $\mathcal H_{ p', q'}^{t',r'} (\rn) $. 
\end{theorem}
%By no means does Theorem \ref{M ' = H} give any complete information on the dual
%space of $	M_{p,q}^{t,r} (\rn)$.
\begin{proof} The proof is a mere combination of Theorems \ref{M is predual of H} and \ref{Fatou block}.
	Thanks to (\ref{M = H '}),  we have 
	\begin{equation*}
		\mathcal H_{ p', q'}^{t',r'} (\rn)  \subset  \left(M_{p,q}^{t,r} (\rn) \right) ' .
	\end{equation*}
Hence, we will verify the converse.
	In fact, suppose that $f \in L^0(\rn) $  satisfies
	\begin{equation} \label{fg le 1}
		\sup \left\{ \left|\int_\rn f(x) g(x) \d x \right| : \|g\|_{ M_{p,q}^{t,r} (\rn) }  \le 1      \right\} \le 1.
	\end{equation}
	Then we observe that $|f(x)| <\infty$ for a.e. $x\in \rn $. Without loss of generality, we can assume that $f \ge 0$. Otherwise, split $f$ into its real
	and imaginary parts and each of these into its positive and negative parts; and we treat each of them.
	
	For every $k\ge 1$, let $Q_k = Q (0,k)$ and let $f_k (x) = \min\{ f(x), k\} \chi_{Q_k} (x)$. Then $f_k \in \mathcal H_{ p', q'}^{t',r'} (\rn) $ since $f_k \in L_c^\infty (\rn)$. 
	
	On the other hand, since $  M_{p,q}^{t,r} (\rn)  = \left( \mathcal H_{ p', q'}^{t',r'} (\rn) \right)^* $, and by (\ref{fg le 1}), we have
	\begin{align*}
		\| f_k\|_{  \mathcal H_{ p', q'}^{t',r'} (\rn) } &  \le C \sup_{ \|g\|_{_{ M_{p,q}^{t,r} (\rn) }  } \le 1 } \left| \int_\rn f_k (x) g(x) \d x \right| \\
		&  \le C \sup_{ \|g\|_{_{ M_{p,q}^{t,r} (\rn) }  } \le 1 } \left| \int_\rn f (x) g(x) \d x \right| \le C <\infty.
	\end{align*}
Here $C$  comes from (\ref{g H = sup}).
	By the Fatou property of $\mathcal H_{ p', q'}^{t',r'} (\rn)$(Theorem \ref{Fatou block}), we deduce  $\|f\|_{\mathcal H_{ p', q'}^{t',r'} (\rn)  } \le C$. 
	By duality and (\ref{fg le 1}), we obtain
	\begin{align*}
		\| f\|_{\mathcal H_{ p', q'}^{t',r'} (\rn)  } & \le C \sup_{ \|g\|_{_{ M_{p,q}^{t,r} (\rn) }  } \le 1 } \left| \int_\rn f (x) g(x) \d x \right| 
		\\
		& \le C   \sup_{ \|g\|_{_{ M_{p,q}^{t,r} (\rn) }  } \le 1 } \|f\|_{\left(M_{p,q}^{t,r} (\rn) \right)' }  \| g\|_{ M_{p,q}^{t,r} (\rn) } \\
		&\le C \|f\|_{\left(M_{p,q}^{t,r} (\rn) \right)' },
	\end{align*}
	which yields 	$  \left(M_{p,q}^{t,r} (\rn) \right) '  \subset  \mathcal H_{ p', q'}^{t',r'} (\rn)  .$
	Hence we have completed the proof.
\end{proof}

Finally, we consider the denseness of $\mathcal H_{ p', q'}^{t',r'} (\rn) $.
\begin{theorem} \label{dense block}
Let $1<q<\infty$.
	Let $1 < p<t<r<\infty$ 	or $1< p\le t<r=\infty$. Then $L^{\infty}  _c (\rn )$ is dense in $\mathcal H_{ p', q'}^{t',r'} (\rn)  $. In particular, by mollification, $C_c^\infty (\rn)$ is dense in $\mathcal H_{ p', q'}^{t',r'} (\rn) $.
\end{theorem}
\begin{proof}
	We use the idea from \cite[Theorem 345]{SFH20}.
	Since $ f \in \mathcal H_{ p', q'}^{t',r'} (\rn)  $, there exist a sequence $\{ \lambda_{j,k} \}_{ (j,k) \in \mathbb Z ^{1+n} }  \in \ell ^{r'} $  and a sequence  $(p', q', t')$-block $\{  b_{j,k} \}  _{ (j,k) \in \mathbb Z ^{1+n} } $  supported on  $Q_{j,k}   $
	such that $f 	=\sum_{(j,k)\in\mathbb{Z}^{n+1}}\lambda_{j,k}b_{j,k} $  and $\| \{ \lambda_{j,k} \} \|_{\ell ^{r'}}  \le \| f\|_{\mathcal H_{ p', q'}^{t',r'} (\rn) }    +\epsilon$.	
	Define 
	\begin{equation} \label{f_N dense}
		f_N = \sum_{|(j,k)|_\infty  \le N}\lambda_{j,k}b_{j,k},
	\end{equation}
	where $|(j,k)|_\infty  = \max  \{ |j|, |k_1|, \ldots, |k_n| \}  $.
	Since $r' \in [1,\infty)$, we have
	\begin{equation*}
		\| f -f_N \|_{ \mathcal H_{ p', q'}^{t',r'} (\rn) }  \le  \left( \sum_{|j| \le N ,|k|_\infty  > N} | \lambda_{j,k} |^{r'}  + \sum_{ |j| > N ,|k|_\infty  \le  N}| \lambda_{j,k} |^{r'}  \right)^{1/r'}   \to 0 
	\end{equation*}
	as $N\to \infty.$
	It is obvious that $f_N \in L^{\infty}  _c (\rn )$ since $N$ is finite.
 Hence
	$L^{\infty}  _c (\rn )$ is dense in $\mathcal H_{ p', q'}^{t',r'} (\rn)  $.

	Now we show $C_c^\infty (\rn)$ is dense in $\mathcal H_{ p', q'}^{t',r'} (\rn) $.  Let $f_N$  be  (\ref{f_N dense}).  For a fixed $\epsilon>0$, we can choose $N_0 \in \mathbb N$ such that for all $N \ge N_0$,
	\begin{equation*}
		\| f -f_N \|_{ \mathcal H_{ p', q'}^{t',r'} (\rn) }  < \epsilon.
	\end{equation*} 
	Let	
	\begin{equation*}
		\omega (x) = 	\begin{cases}
			c e ^{ 1/ (|x|^2 -1) }, & \operatorname{if}  |x| < 1; \\
			0, & \operatorname{else},
		\end{cases}
	\end{equation*} 
	where the constant $c  >0 $  is such that $ \int_{B(0,1)}  \omega (x) \d x =1 $. Set $\omega _\eta (x)  = \eta^{-n} \omega ( \eta^{-1} x )$ for $\eta>0$.
	From \cite[Theorem 2.6-1]{C13}, $\omega_\eta * f_N \in C_c^\infty (\rn) $   and 
	\begin{equation*}
		\omega_\eta *	f_N = \sum_{|(j,k)|_\infty  \le N}\lambda_{j,k} \omega_\eta * b_{j,k}.
	\end{equation*}
Since the support of $b_{j,k} $  is compact, by \cite[Theorem 11.18]{CR16},  $ \omega_\eta *b_{j,k}   $  converges to $b_{j,k}$  uniformly in the support of $b_{j,k}$. 
Since  $ \sup_{ \eta >0 } | \omega_\eta *	b_{j,k} (x)|  \le \| \omega\|_{L^1(\rn)} \M b_{j,k} (x) $, by the boundedness of Hardy-Littlewood maximal operator on  $L^{p',q'} (\rn), p', q' \in (1,\infty)$, we obtain $\omega_\eta *	b_{j,k}  \in L^{p',q'} (\rn)$. Then by Lebesgue dominated converge theorem, we get
\begin{equation*}
	\lim_{\eta \to 0} \| \omega_\eta *  b_{j,k} -  b_{j,k} \|_{L^{p',q'} (\rn)}  =0.
\end{equation*}
	There exists $\eta_0 >0 $ such that for all $0<\eta <\eta_0$, all $ |(j,k)|_\infty \le N_0  $, supp   $\omega_\eta * b_{j,k}  \subset 2 Q _{j,k}$ and 
	\begin{equation*}
		\epsilon^{-1} (2 N_0 +1)^{ (n+1) / r'}  \lambda_{j,k}  \|  \omega_\eta * b_{j,k} -  b_{j,k} \|_{L^{p',q'} (\rn)} \le     | 2Q _{j,k}|^{1/t - 1/p} .
	\end{equation*}
	Hence $ \epsilon^{-1} (N_0 +1)^{ (n+1) / r'}  \lambda_{j,k} ( \omega_\eta * b_{j,k} -  b_{j,k}) $ is a $ (p',q', t') $-block supported on $ 2Q _{j,k} $.
	Then 
	\begin{align*}
		f_{N_0}  - \omega_\eta *	f_{N_0} &=   \sum_{|(j,k)|_\infty  \le N_0}    \lambda_{j,k} \left( b_{j,k}  -  \omega_\eta * b_{j,k}  \right) \\
		& = \sum_{|(j,k)|_\infty  \le {N_0} }  \epsilon (2 N_0 +1)^{-(n+1)/r'  }     
		\\
		& \quad \times           \left(  \epsilon^{-1} ( 2 N_0 +1)^{ (n+1) / r'}  \lambda_{j,k} ( \omega_\eta * b_{j,k} -  b_{j,k})         \right)  . 
	\end{align*}
	Hence 
	\begin{equation*}
		\| f_N  - \omega_\eta *	f_N \|_{ \mathcal H_{ p', q'}^{t',r'} (\rn)  } \le  \epsilon ( 2 N_0 +1)^{-(n+1)/r'  }   \left( \sum_{|(j,k)|_\infty  \le N_0}   1     \right) ^{1/r'}  = \epsilon .
	\end{equation*}
	Then by  Minkowski's inequality, we obtain
	\begin{equation*}
		\| f  - \omega_\eta *	f_N \|_{ \mathcal H_{ p', q'}^{t',r'} (\rn)  } 	\le  \| f -f_N \|_{ \mathcal H_{ p', q'}^{t',r'} (\rn)  } + 	\| f_N  - \omega_\eta *	f_N \|_{ \mathcal H_{ p', q'}^{t',r'} (\rn) }  \le 2 \epsilon.
	\end{equation*} 
	Since  $\epsilon$  is arbitrary, we prove that  $C_c^\infty (\rn)$ is dense in $\mathcal H_{ p', q'}^{t',r'} (\rn)   $.
\end{proof}

%\subsection{when $r<\infty$}
%
%\begin{align*}
%	\| m_k (|f|) \|_{ L^{p,q} (Q) }  & = m_k (|f|)   p^{1/q} \left(  \int_0^1 |Q|^{q/p} s^q \frac{\d s}{s}  \right)^{1/q} \\
%	& =  m_k (|f|)   p^{1/q} |Q|^{1/p}   \frac{1}{q^{1/q}  } \\
%\end{align*}
%
%\begin{align*}
%	\| m_k (|g|) \|_{ L^{p,q} (Q) } 
%	& =  m_k (|g|)   p^{1/q} |Q|^{1/p}   \frac{1}{q^{1/q}  } \\
%\end{align*}
%Let $m_k (|f|)  \ge m_k (|g|) $.
%\begin{align*}
%	\| m_k (|f|) + m_k (|g|) \|_{ L^{p,q} (R) }  & =  p^{1/q} \left(  \int_0^{ m_k (|g|) }  (2|Q| )^{q/p} s^q \frac{\d s}{s}  +  \int_{ m_k (|g|) } ^{m_k (|f|) }  |Q| ^{q/p} s^q \frac{\d s}{s} \right)^{1/q} \\
%	& = p^{1/q} \left( (2|Q| )^{q/p}   \frac{  m_k (|g|) ^q }{q}    +  |Q| ^{q/p} \frac{m_k (|f|) ^q -  m_k (|g|) ^q }{q} \right)^{1/q} \\
%	& =  p^{1/q}|Q| ^{1/p} \left( (2 )^{q/p}   \frac{  m_k (|g|) ^q }{q}    +   \frac{m_k (|f|) ^q -  m_k (|g|) ^q }{q} \right)^{1/q} \\
%\end{align*}

\section{Applications of operators} \label{Applicaiton operator}
We next show the boundedness of  the classical operators on Bourgain-Morrey-Lorentz spaces and its predual, such as the
Hardy-Littlewood maximal operator, sharp maximal operator, Calder\'on-Zygmund operator, fractional integral operator, commutator on Bourgain-Morrey-Lorentz spaces and the block spaces.

\subsection{Hardy-Littlewood maximal operator}\label{HL section}

In this subsection, we consider the boundedness property of the Hardy-Littlewood maximal operator on Bourgain-Morrey-Lorentz spaces $M_{p,q}^{t,r} (\rn) $ and block spaces $\mathcal H_{ p', q'}^{t',r'} (\rn)$.

Denote by $\mathcal M_{\operatorname{dyadic}}$ the dyadic maximal operator generated by the dyadic cubes  in $\mathcal D$. We do not have a pointwise estimate to control $\mathcal M$ in terms of $\mathcal M_{\operatorname{dyadic}}$. 
The maximal operator generated by a family $\mathfrak D$ is defined by 
\begin{equation*}
	\mathcal M_{\mathfrak D} f (x) = \sup_{Q \in \mathfrak D} \frac{\chi_Q (x)}{|Q|} \int_Q |f(y) | \d y
\end{equation*}
for $f\in L^0 (\rn)$ and a dyadic grid $\mathfrak D \in \{ \mathcal D_{\vec a} : \vec a \in \{0,1,2\}^n   \}$.
As in \cite{LN19},
\begin{equation} \label{M le Mdya}
	\mathcal M f (x) \lesssim \sum_{ \vec a \in \{0,1,2\}^n }M_{\mathcal D _{\vec a}} f (x).
\end{equation}

\begin{theorem} \label{HL M}
	Let $1 \le q \le \infty$. Let $ 1 <p <t<r <\infty $ or $ 1 <p \le t < r =\infty $. Then $\mathcal M$  is bounded on  $M_{p,q}^{t,r} (\rn) $.
\end{theorem}

\begin{proof}
	We only need to prove the case $r<\infty$, since $r =\infty$ is in \cite[Lemma 3.1]{DK24}. Due to (\ref{M le Mdya}) and Lemma \ref{D equivalence}, it sufficient to show the boundedness for the dyadic maximal operator $ 	\mathcal M_{\mathfrak D}$  instead of the Hardy-Littlewood maximal function $\mathcal M$.	
	Let $f \in M_{p,q}^{t,r} (\rn) $, $Q \in \mathfrak D$, $f_1 = f \chi_Q$ and $f_2 = f \chi_{\rn \backslash Q}$. By Lemma \ref{HL Lorentz}, we have
	\begin{equation} \label{M f_1}
		\| 	\mathcal M_{\mathfrak D} f_1 \|_{ L^{p,q} (Q)} \lesssim 	\| f_1 \|_{ L^{p,q} (\rn)} = \| f \|_{ L^{p,q} (Q)}.
	\end{equation}
For $k \in \mathbb N$, let $Q_k$  be the $k^{\operatorname{th}}$ dyadic parent of $Q$, which is the dyadic cube in $\mathfrak{D}$ satisfying $Q \subset Q_k$ and $\ell (Q_k) = 2^k \ell (Q)$. 
Hence, for $x \in Q$, applying the H\"older inequality in Lorentz spaces, we obtain
\begin{align*}
		\mathcal M_{\mathfrak D} f_2 (x) & \le \sum_{k=1} \frac{1}{|Q_k|} \int_{Q_k} |f(y) | \d y \\
	& \lesssim \sum_{k=1} |Q_k|^{-1/p} \| f \|_{ L^{p,q} (Q_k)}  .
\end{align*}
Thus,
\begin{align*}
	\|	\mathcal M_{\mathfrak D} f_2 \|_{ L^{p,q} (Q) } \lesssim |Q|^{1/p} \sum_{k=1} |Q_k|^{-1/p} \| f \|_{ L^{p,q} (Q_k)} .
\end{align*}
Let $R \in \mathfrak D$ and $k\in \mathbb N$. A geometric observation shows there are $2^{kn}$ dyadic cubes $Q$ such that $Q_k = R$. Namely,
\begin{equation} \label{geo 2kn}
	 \sum_{  Q  \in \mathfrak D, Q_k = R}  |Q_k|^{r/t -r/p} \|f\|_{  L^{p,q} (Q_k)} ^r  = 2^{kn} |R|^{r/t -r/p} \|f\|_{  L^{p,q} (R)} ^r .
\end{equation}
Adding (\ref{geo 2kn}) over  $R \in \mathfrak D$ gives
\begin{equation*}
	\sum_{Q \in \mathfrak D} |Q_k| ^{r/t -r/p} \|f\|_{  L^{p,q} (Q_k)} ^r  =2^{kn} \sum_{R \in \mathfrak D } |R|^{r/t -r/p} \|f\|_{  L^{p,q} (R)} ^r  = 2^{kn}  \|f\|_{M_{p,q}^{t,r} (\rn) } ^r .
\end{equation*}
Then 
\begin{align*}
	\left(\sum_{Q\in \mathfrak D}  |Q|^{r/t-r/p} \|	\mathcal M_{\mathfrak D} f_2 \|_{ L^{p,q} (Q) } ^r  \right)^{1/r} 
& \lesssim \left(\sum_{Q\in \mathfrak D}  |Q|^{r/t} \left(  \sum_{k=1} |Q_k|^{-1/p} \| f \|_{ L^{p,q} (Q_k)} \right)^r  \right)^{1/r} \\
& \le \sum_{k=1}^\infty 2^{-kn/t}  \left(\sum_{Q\in \mathfrak D}  |Q_k|^{r/t -r/p}   \| f \|_{ L^{p,q} (Q_k)}^r  \right)^{1/r} \\
& \le \sum_{k=1}^\infty 2^{-kn/t} 2^{kn/r}  \left(\sum_{Q_k\in \mathfrak D}  |Q_k|^{r/t -r/p}   \| f \|_{ L^{p,q} (Q_k)}^r  \right)^{1/r} \\
& \lesssim \| f\|_{M_{p,q}^{t,r} (\rn)  } .
\end{align*}
Together this with (\ref{M f_1}), we prove 
\begin{equation*}
\| \mathcal M_{\mathfrak D} f\|_{M_{p,q}^{t,r} (\rn)  } 	\lesssim \| f\|_{M_{p,q}^{t,r} (\rn)  } 
\end{equation*}
as desired.
\end{proof}

To obtain the  boundedness of the Hardy-Littlewood maximal operator on  block spaces $\mathcal H_{ p', q'}^{t',r'} (\rn)$, we need some preparation.
\begin{definition}
	Let  $1\le q \le \infty$.
	Let $1\le p\le t\le r\le\infty$. The function space $\mathcal{H}_{p', q',*}^{t',r'} (\rn)$
	is the set of all measurable functions $f$ such that $f$ is realized
	as the sum
	\begin{equation}
		f=\sum_{v = 1}^\infty  \lambda_{v}g_{v}\label{eq:block f 2}
	\end{equation}
	with some $\lambda=\{\lambda_{v}\}_{v \in\mathbb{N}}\in\ell^{r'}(\mathbb{N})$
	and $g_{jk}$ is a $(p',q',t')$-block, where
	the convergence is in $L_{\mathrm{loc}}^{1}.$ The norm of $\mathcal{H}_{p', q' ,*}^{t',r'}$
	is defined by
	\[
	\|f\|_{\mathcal{H}_{p',q',*}^{t',r'}}:=\inf_{\lambda}\|\lambda\|_{\ell^{r'}},
	\]
	where the infimum is taken over all admissible sequence $\lambda$
	such that (\ref{eq:block f 2}) holds.
\end{definition}

\begin{remark}
	Let $\D := \{ Q^ {(j)} \}_{j=1}^\infty  $ be an enumeration of $\D$. Then we  immediately get
	\begin{equation*}
		\| f \| _{ \mathcal{H}_{p',q',*}^{t',r'} (\rn) } \le  	\| f \| _{ \mathcal{H}_{p',q'}^{t',r'} (\rn) } .
	\end{equation*}
Thus, $ \mathcal{H}_{p',q'}^{t',r'} (\rn) \subset   \mathcal{H}_{p',q',*}^{t',r'} (\rn) $. 
\end{remark}

\begin{lemma}\label{r'=1 equi *}
	The spaces $ \mathcal{H}_{p',q'}^{t',,r'} (\rn)$ and $ \mathcal{H}_{p',q',*}^{t',r'} (\rn)$ are equivalent when $r' =1$. 
\end{lemma}

\begin{proof}
	The proof is similar to \cite[Lemma 343]{SFH20}. For the convenience of the reader, we give the proof here. Let  $ f \in \mathcal{H}_{p',q',*}^{t',r'} (\rn)$. Then $f$ can be decomposed as 
\begin{equation*}
	f = \sum_{v \in K} \lambda_{v} b_v
\end{equation*}
where $K \subset \mathbb N$ is an index set, and 
\begin{equation*}
	0 < \sum_{v \in K}  | \lambda_{v} |\le (1+\epsilon)\| f\|_{  \mathcal{H}_{p',q',*}^{t',r'} (\rn) },
\end{equation*}
and each $b_v$ is a $ (p',q',t')$-block supported on some cube $Q_v$. We divide $K$ into the disjoint sets $K(Q) \subset \mathbb N$, $Q \in \D$, as 
\begin{equation*}
	K = \bigcup_{ Q \in \D } K(Q)
\end{equation*}
and $K(Q)$ satisfies   supp $(b_v) \subset 3Q$ and  $ |Q_v | \ge |Q| $ when $v \in K(Q)$.

Let $\D := \{ Q^ {(j)} \}_{j=1}^\infty  $ be an enumeration of $\D$. For each $v \in K$, we define
\begin{equation*}
	j_v :=\min\{ j :  \operatorname{supp} (b_v) \subset 3Q^{(j)} , |Q_v | \ge |Q^{(j)} |\}.
\end{equation*}
We write $ K ( Q ^{(j)} ) :=  \{ v\in K : j_v = j \}$. Set
\begin{equation*}
	\lambda (Q) := 3^n   \sum_{v\in K (Q)} |\lambda_{v} |,
\end{equation*}
and
\begin{equation*}
	b(Q)  := \begin{cases}
		\frac{1}{	\lambda (Q)} \sum_{v\in  K(Q)}  \lambda_{v} b_v , & \operatorname{if} \; 	\lambda (Q) \neq 0, \\
		0,  & \operatorname{else.} 
	\end{cases}
\end{equation*}
Now rewrite $f$ as
\begin{align*}
	f & = \sum_{v \in K} \lambda_{v} b_v = \sum_{Q \in \D} \left(\sum_{v\in K (Q)} \lambda_{v} b_v  \right) \\
	& = \sum_{Q \in \D}  3^n   \sum_{v\in K (Q)} |\lambda_{v} |
	\left\{
	\left( 3^n   \sum_{v\in K (Q)} |\lambda_{v} |
	\right)^{-1}  \sum_{v\in K (Q)} \lambda_{v} b_v  \right\} \\
	& = \sum_{Q \in \D} \lambda (Q) b (Q).
\end{align*}
Hence,
\begin{align*}
	\sum_{Q \in \D}  |\lambda (Q)|&= 3^n \sum_{Q \in \D} \sum_{v\in K (Q)} |\lambda_{v} |  
	= 3^n \sum_{v \in K} |\lambda_{v} |\le 3^n (1+\epsilon)\| f\|_{  \mathcal{H}_{p',q',*}^{t',r'} (\rn) } .
\end{align*}
Since each $b_v$ is a $(p',q',t')$-block,  we obtain
\begin{align*}
	 \left( 3^n  \sum_{v\in K (Q)} |\lambda_{v} |
	\right)^{-1}  \left\| \sum_{v\in K (Q)} \lambda_{v} b_v \right\|_{L^{p',q'}} 
	& \le 3^{-n} \left (  \sum_{v\in K (Q)} |\lambda_{v} |\right)^{-1}  \sum_{v\in K (Q)} | \lambda_{v} | \| b_v\|_{L^{p',q'}} \\
	& \le 3^{-n}  |Q|^{ 1/ p' - 1/ t'}  \lesssim    | 3Q |^{1/ p' - 1/ t'} ,
\end{align*}	
which implies that $b (Q)$ is  a $(p',q'.t')$-block supported on $ 3Q $ modulo a multiplicative
constant.  Letting $\epsilon \to 0$, we  complete the proof.
\end{proof}

 To get the relation of $ \mathcal{H}_{p',q'}^{t',,r'} (\rn)$ and $ \mathcal{H}_{p',q',*}^{t',r'} (\rn)$ when $ 1< r' <\infty$, we need the definition of finite overlapping.

\begin{definition} \label{def finite over}
	Let $\D := \{ Q^ {(j)} \}_{j=1}^\infty  $ be an enumeration of $\D$. 
	Let $f  \in  \mathcal{H}_{p',q',*}^{t',r'} (\rn)$ with $ \| f\|_{\mathcal{H}_{p',q',*}^{t',r'} (\rn)  } >0 $. 
	Let $0< \epsilon < 1/2^{10}$.
	Decompose $f$ as 	\begin{equation}\label{condi 1}
		f = \sum_{v \in K} \lambda_{v} b_v ,
	\end{equation}
	where $K \subset \mathbb N$ is an index set, and 
	\begin{equation} \label{condi 2}
		0 < \left( \sum_{v \in K}  | \lambda_{v} |^{r'}      \right)^{1/r'} \le (1+\epsilon) \| f\|_{  \mathcal{H}_{p',*}^{t',r'} (\rn) }
	\end{equation}
	and
	\begin{equation} \label{condi 3}
		\operatorname{each} \; b_v \; \operatorname{is \; a\;} (p',q',t') \operatorname{-block \; supported \; on \; some \;  cube \; } Q_v.
	\end{equation}
	% each $b_v$ is a $ (p',t')$-block supported on some cube $Q_v$.
	
	We define a map $T : K \subset \mathbb N \to \mathbb N$ as follows. For each $ v \in K$,
	\begin{equation*}
		T(v)  :=\min\{ j :  \operatorname{supp} (b_v) \subset 3Q^{(j)} , |Q_v | \ge |Q^{(j)} |\}.
	\end{equation*}
	For each $Q^{(j)} \in \D = \{ Q^ {(j)} \}_{j=1}^\infty  $, we define 
	\begin{equation*}
		T ^{(-1)}( Q^{(j)} )  := \{ v \in K :  T(v)   =j \}.
	\end{equation*}

	Then for each  dyadic cube $Q^{(j)}$ may have three cases: $\sharp 	T ^{(-1)}( Q^{(j)} ) =0 $,  $0< \sharp 	T ^{(-1)}( Q^{(j)} ) <\infty  $,  $\sharp 	T ^{(-1)}( Q^{(j)} ) =\infty  $. 
	
	We say that the decomposition  	$f$ satisfying   (\ref{condi 1}), (\ref{condi 2}),  (\ref{condi 3}) is finite overlapping of $C$ if 
	\begin{equation*}
		\sup_{ Q^{(j)} \in \D }  \sharp 	T ^{(-1)}( Q^{(j)} )  \le C <\infty.
	\end{equation*}
	
	We say that $f \in  \mathcal{H}_{p',q',*}^{t',r'} (\rn) $ is finite overlapping  of $C$ if for each  decomposition of $f$  satisfying   (\ref{condi 1}), (\ref{condi 2}),  (\ref{condi 3})  is finite overlapping of $C$. 
\end{definition}

Similar to Lemma \ref{r'=1 equi *},  we obtain the similar result under the condition of  finite overlapping.
\begin{theorem} \label{finite over}
	Let $ 1<p < t< r <\infty $. Let $f \in \mathcal{H}_{p',q',*}^{t',r'} (\rn) $ be finite overlapping  of $C$. Then 
	\begin{equation*}
		\| f \| _{ \mathcal{H}_{p',q'}^{t',r'} (\rn) }   \le   c_{n,p',t',r',C} \| f \| _{ \mathcal{H}_{p',q',*}^{t',r'} (\rn) } .
	\end{equation*}
	where the constant  $c_{n,p',t',r',C}$ only depends on $n,p',t',r',C$.
\end{theorem}

The block spaces $ \mathcal{H}_{p',q'}^{t',r'} (\rn)$ have the similar property in Lemma \ref{D equivalence} as $M_{p,q}^{t,r} (\rn)  $.

\begin{theorem} \label{M block}
	Let $1 <q' <\infty$.
	Let $ 1 <r' <t' <p' < \infty $ or $ 1 = r' <  t' \le p' < \infty $ .
	Let $b $  be a  $ (p',q',t')$-block. 
	Then $  \M (b)  $  has a decomposition:
	\begin{equation*}
		\M (b) =  \sum_{k=0}^\infty c  2^{-kn/t} m_k,
	\end{equation*}
	where $m_k$ is a $(p',q',t')$-block, $  \left(  \sum_{k=0}^\infty | c  2^{-kn/t} |^{r'} \right)^{1/r'} \le C_{ t, n,\eta,p,r }  <\infty$ and $ C_{ t, n,\eta,p,r }$  depends only on the norm of $\M $ on $L^{p',q'} (\rn)$, $n$, $\eta$, $p$, $r$. 
	Furthermore,  this decomposition satisfies the finite overlapping property in Definition \ref{def finite over}.
	Hence 	$  \M (b)  $ 
	belongs to $ \H _{p',q'} ^{t' , r'} (\rn) $.
\end{theorem}

\begin{proof}
%	If $t' =\infty$, then $ \H _{\infty,\infty} ^{\infty , 1} (\rn) =L^\infty (\rn)$. This case is trivial.
	
%	Let $1< t' <\infty $.	
	Let $b $  be a  $ (p',q',t')$-block supported on a cube $Q \in \D_v$, $v\in \mathbb Z$. Note that $\| b\|_{\H _{p',q'} ^{t' , r'} (\rn)} \le 1 $.

	For $k=0$, define $\chi_k := \chi_Q$; for $ k\ge 1$, let $ \chi_k := \chi_{2^k Q \backslash 2^{k-1}Q } $.
	Then
	\begin{equation*}
		\M (b) =  \sum_{k=0} \M (b)\chi_k .
	\end{equation*}
	For $k=0$,
	the function $ \M (b)\chi_k$ is supported on $Q$ and using the boundedness of $	\M$ on $L^{p',q'} (\rn)$, we have
	\begin{equation*}
		\| \M (b)\chi_k\|_{L^{p',q'} (\rn)} \le \|\M \|_{ L^{p',q'} (\rn)} 	\|b \|_{L^{p',q'} (\rn)} \le \|\M \|_{ L^{p',q'} (\rn) }  |Q|^{1/p' -1/t'}.
	\end{equation*}
	Thus $ \frac{1}{\|\M \|_{ L^{p',q'} (\rn)} } M (b)\chi_k   $ is a $ (p',q',t') $-block on $Q$. 
	
	For $k\ge 1$, $M (b)\chi_k$ is supported on $ 2^k Q  $ and
	\begin{align*}
		\M (b)\chi_k \le c \frac{ \int _Q |b| \d x }{| 2^k Q |} \le c \frac{ \| b\|_{L^{p',q'} (\rn) }}{| 2^k Q |} |Q|^{ 1/p}.
	\end{align*}
	Hence
	\begin{align*}
		\left\|      \M (b)\chi_k \right\|_{ L^{p',q'} (\rn) } & \le c  \frac{ \| b\|_{L^{p',q'} (\rn) }}{| 2^k Q |} |Q|^{ 1/p} | 2^{k}Q |^{1/p'} \\
		& \le c 
		\frac{ |Q|^{1/p' - 1/t'}     }{| 2^k Q |} |Q|^{1/p}  | 2^{k}Q |^{1/p'} \\
		& = c 2^{-kn (1-1 / t') }   | 2^{k}Q |^{1/p'   - 1/t'}.
	\end{align*}
	Thus, $  c 2^{kn/t}  \M (b)\chi_k  $ is a $ (p',q',t')$-block supported on $2^k Q$. 
	Let $ b_k =  c 2^{kn/p}   \M (b)\chi_k  $ for $k\ge 1$ and $b_0 = c \M (b)\chi_0  $. Let $c_k = c   2^{-kn/t}  $ for $k\ge 0$.
	Then 
	\begin{equation} \label{decom M b}
		\M (b) =  \sum_{k=0}^\infty c_k m_k
	\end{equation}
	and for any $s >0$,
	\begin{equation*}
		\left(  \sum_{k=0} |c_k| ^{s} \right)^{1/s}  \le c.
	\end{equation*}
	There are at most $2^n$ cubes $Q_i  \in \D _{ v -k}$  such that 
	\begin{equation*}
		2^k Q \subset   \cup_{i=1}^{2^n} Q_i.
	\end{equation*}
	Set $ m_{ki} =m_k \chi_{Q_i}  $. $ m_{ki} $ is a  $ (p',t')$-block supported on $Q_i  \in \D _{ v -k}$.
	Hence
	\begin{equation} \label{decom M b 2}
		\M (b) =  \sum_{k=0}^\infty c_k  \sum_{i=1}^{2^n} m_{ki}
	\end{equation}
	Finally, we have prove that $\|  \M  (b) \|_{\H _{p',q'} ^{t' , r'} (\rn)  } \lesssim  \left(  \sum_{k=0} |c_k| ^{r'} \right)^{1/r'}  <\infty$.
\end{proof}
Then we have the boundedness property of the Hardy-Littlewood maximal operator on  $ \mathcal{H}_{p',q'}^{t',r'} (\rn)$.
\begin{theorem} \label{HL block}
	Let $1<q' <\infty$.
	Let $ 1 <r' <t' <p' < \infty $ or $ 1 = r' <  t' \le p' < \infty $.  
	Then $\M$ is  bounded  on  $ \H _{p',q'} ^{t' , r'} (\rn) $.
\end{theorem}

\begin{proof}
	Without loss of generality,
	we may assume that   $\|  f\|_ { \H _{p', q'} ^{t' , r'} (\rn) } = 1 $. 
	Then there  exist a sequence $\{ \lambda_{j,k} \}_{ (j,k) \in \mathbb Z ^{1+n} }  \in \ell ^{r'} $  and a sequence  $(p',q',t')$-block $\{  b_{j,k} \}  _{ (j,k) \in \mathbb Z ^{1+n} } $  supported on  $Q_{j,k}   $
	such that $f 	=\sum_{(j,k)\in\mathbb{Z}^{n+1}}\lambda_{j,k}b_{j,k} $  and $\| \{ \lambda_{j,k} \} \|_{\ell ^{r'}}  \le 1   +\epsilon$.	
	Using Theorem  \ref{M block} to  each block $ b_{j,k}$, we have 
	\begin{equation*}
		\M (b_{j,k} )   =c  \sum_{v=0}^\infty 2^{-vn/t}  \sum_{i=1}^{2^n} m_{j,k,v,i}
	\end{equation*}
	and for each $v\ge 0$, each $i = 1,\ldots, 2^n$, $m_{j,k,v,i}$ is  $(p',q',t')$-block.
	
	We then arrange the sum. For each $Q_{j,k}$, 
	there is only one $i \in \{  1, \ldots, 2^n\}$ such that supp $m_{j,k,v,i} \subset Q_{j,k}$.
	Let 
	\begin{equation*}
		b_{Q_{j,k}}  =  \sum_{v=0}^{ \infty } 2^{-vn/t} m_{ j+v, k, v, i' }, \; \operatorname{for \; some\;}  i' \in \{  1, \ldots, 2^n\}.
	\end{equation*}
	Then  supp  $b_{Q_{j,k}} \subset Q_{j,k}$ and
	\begin{align*}
		\|	b_{Q_{j,k}}\|_{L^{p',q'} (\rn)}  &\le  \sum_{v=0}^{ \infty } 2^{-vn/t} \|	 m_{ j+v, k, v, i' }\|_{L^{p',q'} (\rn)} \\
		& \le  c |Q_{j,k}|^{1/p' - 1/t'}.
	\end{align*}
	Hence 	$b_{Q_{j,k}} $ is a  $(p',q',t')$-block supported on  $Q_{j,k}   $ modulo a multiplicative
	constant.
	
	Then
	\begin{align*}
		\M 	(f)  & \le \sum_{(j,k)\in\mathbb{Z}^{n+1}} | \lambda_{j,k}|  \M  (  b_{j,k} ) \\
		& = c \sum_{(j,k)\in\mathbb{Z}^{n+1}} | \lambda_{j,k}| \sum_{v=0}^\infty  2^{-vn/t} \sum_{i=1}^{2^n} m_{j,k,v,i} \\
		& \le c \sum_{i=1}^{2^n}  \sum_{(j,k)\in\mathbb{Z}^{n+1}} | \lambda_{j,k}|  \sum_{v=0}^{ \infty } 2^{-vn/t} m_{ j+v, k, v, i }  \\
		& = c   \sum_{(j,k)\in\mathbb{Z}^{n+1}} | \lambda_{j,k}|  b_{Q_{j,k}}
		. 
	\end{align*}
	By the  lattice property, we have
	\begin{equation*}
		\| \M	(f) \|_{  \H _{p',q'} ^{t' , r'} (\rn) }  \le c \left(   \sum_{(j,k)\in\mathbb{Z}^{n+1}} | \lambda_{jk}|^{r'}          \right) ^{1/r'} \le c (1+\epsilon) = c (\|f\|_{\H _{p',q'} ^{t' , r'} (\rn)} + \epsilon).
	\end{equation*}

	Letting $\epsilon \to 0$, we obtain  	$\| \M 	(f) \|_{  \H _{p',q'} ^{t' , r'} (\rn) } \le c\|f\|_{\H _{p',q'} ^{t' , r'} (\rn)}.$ Thus the proof is complete.
\end{proof}

Now we consider the boundedness of the powered Hardy-Littlewood maximal operator $\M _\eta ( f) : =\M ( f^\eta ) ^{1/\eta}$ for $0<\eta<\infty$ on block spaces  $\H _{p',q'} ^{t' , r'} (\rn)$. The boundedness will be used in the proof of Sharp maximal operator.

\begin{remark} \label{banach lattic fatou}
	Let $1<q' <\infty$.
Let $ 1 <r' <t' <p' \le \infty $ or $ 1 = r' <  t' \le p' \le \infty $.
By the definition of spaces $\H _{p', q'} ^{t' , r'} (\rn) $, it is not hard to show that spaces $\H _{p',q'} ^{t' , r'} (\rn) $ have lattice property. 
Together  with Theorem \ref{Fatou block}, we have shown that each block space $\H _{p',q'} ^{t' , r'} (\rn) $ is a Banach lattices with  the Fatou property with constant $C_{\mathcal F} =1 $.
\end{remark}

\begin{lemma} [Theorem 1.1,  \cite{S24}] \label{char M on X}
	Let $X(\Omega,d,\mu)$ be a quasi-Banach lattice over a space of homogeneous type $(\Omega,d,\mu)$. If $ X(\Omega,d,\mu) $ has the Fatou property, the following statements are equivalent:
	
	{\rm (i)} $\M$ is bounded on $X(\Omega,d,\mu)$;
	
	{\rm (ii)} There exists $r_0 \in (0,1)$  such that if $r \in [ r_0,1 )$, then  $\M$ is bounded on $X^{(r)}  (\Omega,d,\mu)$ where
	\begin{equation*}
		X^{(r)}  (\Omega,d,\mu) := \{ f : |f|^r \in  X(\Omega,d,\mu) \}.
	\end{equation*}
\end{lemma}

	Let $X(\Omega,d,\mu)$ be a quasi-Banach lattice over a space of homogeneous type $(\Omega,d,\mu)$ with the Fatou property.
	From the proof of \cite[Theorem 4.2]{S24}, we conclude that there exists $\eta >1$ such that $\M_\eta $ is bounded on $ X(\Omega,d,\mu)$ provided that $\M$  is bounded on $X(\Omega,d,\mu)$. Therefore, from Theorem \ref{HL block}, Remark \ref{banach lattic fatou}, we have the following corollary.

\begin{corollary} \label{power HL block}
	Let $1<q' <\infty$.
Let $ 1 <r' <t' <p' < \infty $ or $ 1 = r' <  t' \le p' < \infty $.  Then there exists $\eta > 1$  such that
$\M_\eta$ is a bounded operator on  $ \H _{p', q'} ^{t' , r'} (\rn) $.
\end{corollary}

\subsection{Fractional integral operators}

Next, we consider the boundedness of fractional integral operators in Bourgain-Morrey-Lorentz spaces.
Let $ 0 <\alpha <n$, the fractional integral operator $I_\alpha$ of order $\alpha$ is defined by
\begin{equation*}
	I_\alpha f (x) = \int_\rn \frac{f(y)}{|x-y|^{n-\alpha} }  \d y, \quad  x\in \rn,
\end{equation*}
for $f \in L^1_{\operatorname{loc}} (\rn)$ as long as the right-hand side makes sense. By the standard argument for fractional integral operators on Morrey spaces (for example, see \cite{SFH20}), we know the following pointwise estimate for measurable functions $f$ 
\begin{equation*}
		I_\alpha f (x) \lesssim \| f \|_ {M_{ 1,1 }^{ t,\infty } (\rn) }  ^{ t\alpha /n }  \mathcal M f (x) ^{ 1 - t \alpha /n },  \quad   x \in \rn
\end{equation*}
for $1 \le t < n /\alpha$. Combining this pointwise estimate, the embedding property and the boundedness of the Hardy-Littlewood maximal function, we obtain the following result.

\begin{theorem}
	Let $ 1< q_i  <\infty $ for $i =1,2$.
	Let $ 0 <\alpha <n$ and $t_1, t_2 \in (1,\infty) $ satisfy 
	\begin{equation*}
		\frac{1}{t_2} = \frac{1}{t_2} -\frac{\alpha}{n}.
	\end{equation*}
Assume that either
\begin{equation*}
 1<p_i <t_i <r_i <\infty,  \; \operatorname{for}\; i =1,2 \; \operatorname{and} \;	\frac{p_1}{p_2} = 	\frac{q_1}{q_2}	 = \frac{t_1}{t_2} = \frac{r_1}{r_2}
\end{equation*}
or 
\begin{equation*}
	1<p_i \le t_i <r_i = \infty,  \; \operatorname{for}\; i =1,2 \; \operatorname{and} \;	\frac{p_1}{p_2} = 	\frac{q_1}{q_2}	 = \frac{t_1}{t_2} .
\end{equation*}
Then for $f\in M_{ p_1, q_1 }^{ t_1,r_1 } (\rn)$, we have
\begin{equation*}
		\| I_\alpha f\|_{ M_{ p_2, q_2 }^{ t_2,r_2 } (\rn) }   \lesssim \| f \|_ {	M_{ p_1, q_1 }^{ t_1,r_1 } (\rn) }.
\end{equation*}
	
\end{theorem}

\begin{proof}
	By the H\"older inequality for Lorentz spaces, we obtain
	\begin{align*}
		|Q|^{1/t-1} \int_Q |f(x) |\d x \lesssim |Q|^{1/t-1} \|f\|_{L^{p,q} (Q)} |Q|^{1/p'} =  |Q|^{1/t-1/p} \|f\|_{L^{p,q} (Q)} .
	\end{align*}
	Thus using the embedding 
	\begin{equation*}
		M_{ p_1, q_1 }^{ t_1,r_1 } (\rn)       \hookrightarrow  M_{ 1,1 }^{ t_1,r_1 } (\rn)    \hookrightarrow   M_{ 1,1 }^{ t_1,\infty } (\rn),
	\end{equation*}
we obtain
		\begin{align*}
		|I_\alpha f  (x)| \le 	I_\alpha (|f|)  (x)  & \lesssim 
		\| f \|_ {M_{ 1,1 }^{ t_1,\infty } (\rn) }  ^{ t_1  \alpha /n }  \mathcal M f (x) ^{ 1 - t_1 \alpha /n } \\
		& \lesssim  	\| f \|_ {	M_{ p_1, q_1 }^{ t_1,r_1 } (\rn) }  ^{ t_1  \alpha /n }  \mathcal M f (x) ^{ 1 - t_1 \alpha /n }.
	\end{align*}
Thus 
\begin{align*}
	\| I_\alpha f\|_{ M_{ p_2, q_2 }^{ t_2,r_2 } (\rn) }  & \lesssim 	\| f \|_ {	M_{ p_1, q_1 }^{ t_1,r_1 } (\rn) }  ^{ t_1  \alpha /n }   \|  \mathcal M f ^{ 1 - t_1 \alpha /n }\|_{ M_{ p_2, q_2 }^{ t_2,r_2 } (\rn) } \\
	& = 	\| f \|_ {	M_{ p_1, q_1 }^{ t_1,r_1 } (\rn) }  ^{ t_1  \alpha /n }   \|  \mathcal M f \|_{ M_{ p_2 (1 - t_1 \alpha /n) , q_2(1 - t_1 \alpha /n)  }^{ t_2(1 - t_1 \alpha /n) ,r_2 (1 - t_1 \alpha /n) } (\rn) } ^{(1 - t_1 \alpha /n) } \\
		& = 	\| f \|_ {	M_{ p_1, q_1 }^{ t_1,r_1 } (\rn) }  ^{ t_1  \alpha /n }   \|  \mathcal M f \|_{ M_{ p_1, q_1  }^{ t_1 ,r_1 } (\rn) } ^{(1 - t_1 \alpha /n) } \\
		& \lesssim \| f \|_ {	M_{ p_1, q_1 }^{ t_1,r_1 } (\rn) }.
\end{align*}
Hence we obtain the desired result.
\end{proof}
\begin{remark}
	Using duality and 
	\begin{align*}
		\int_\rn I_\alpha f (x) g(x) \d x &= 
	   \int_\rn \int_\rn \frac{f(y)}{|x-y|^{n-\alpha} }  \d y  g(x) \d x \\
	   &
		  = \int_\rn  f (y) \int_\rn  \frac{g (x)}{|x-y|^{n-\alpha} }   \d x    \d y    = \int_\rn f(y) I_\alpha (y) \d y,
		\end{align*}
one obtains the boundedness of fractional integral operators on block spaces. We omit it here and leave the details to the reader.
\end{remark}
As a corollary, we can also show the boundedness of the fractional maximal operator $\mathcal M_\alpha$. The fractional maximal operator  $\mathcal M_\alpha$ is defined by 
\begin{equation*}
	 \mathcal M_\alpha f (x) = \sup_{x \in Q}  \frac{1}{|Q|^{1- \alpha/n}}  \int_Q |f(y) | \d y, x\in \rn,
\end{equation*}
where the supremum is taken over all cubes $Q$ in $\rn$ containing $x$. The pointwise inequality $ \mathcal M_\alpha f (x) \lesssim I_\alpha ( |f| ) (x) $, $x\in \rn $  leads to the following corollary.
\begin{corollary}
		Let $ 0 <\alpha <n$ and $t_1, t_2 \in (1,\infty) $ satisfy 
	\begin{equation*}
		\frac{1}{t_2} = \frac{1}{t_2} -\frac{\alpha}{n}.
	\end{equation*}
	Assume that either
	\begin{equation*}
		1<p_i <t_i <r_i <\infty,  p_i \le q_i <\infty \; \operatorname{for}\; i =1,2 \; \operatorname{and} \;	\frac{p_1}{p_2} = 	\frac{q_1}{q_2}	 = \frac{t_1}{t_2} = \frac{r_1}{r_2}
	\end{equation*}
	or 
	\begin{equation*}
		1<p_i \le t_i <r_i = \infty,  p_i \le q_i <\infty \; \operatorname{for}\; i =1,2 \; \operatorname{and} \;	\frac{p_1}{p_2} = 	\frac{q_1}{q_2}	 = \frac{t_1}{t_2} .
	\end{equation*}
	Then for $f\in M_{ p_1, q_1 }^{ t_1,r_1 } (\rn)$, we have
	\begin{equation*}
		\| \mathcal M_\alpha f\|_{ M_{ p_2, q_2 }^{ t_2,r_2 } (\rn) }   \lesssim \| f \|_ {	M_{ p_1, q_1 }^{ t_1,r_1 } (\rn) }.
	\end{equation*}
\end{corollary}

\subsection{Sharp maximal function}
In this subsection, we prove a Bourgain-Morrey-Lorentz bound for the sharp maximal function, introduced by Fefferman-Stein \cite{FS72}
\begin{equation*}
	\mathcal M ^{\sharp} (f) (x) = \sup_{ x\in Q} \frac{1}{|Q|} \int_Q |f(y)- f_Q | \d y,
\end{equation*}
where $f_Q = \int_Q f (y) \d y /|Q| $, and the supremum is taken over all cubes $Q$ containing $x$.

Let $ \mathbb M ^+ (\rn) $ be  the cone of all non-negative Lebesgue measurable functions on $\rn$.
For $j \ge 1$, a cube $Q \in \D $, let $j_{\operatorname{pa}} Q$  be the $j$-th parent of $Q$, that is $Q \subset j_{\operatorname{pa}} Q \in \D$ and $2^{j} \ell (Q) =  \ell ( j_{\operatorname{pa}} Q  )$.

Now, we have the following sharp maximal inequality for Bourgain-Morrey-Lorentz spaces.

\begin{theorem} \label{sharp BML}
	Let $1 < q < \infty$. Let $ 1 <p <t<r <\infty $ or $ 1 <p \le t < r =\infty $. Then for $f \in M_{p,q}^{t,r} (\rn)$  and $g \in  \H _{p',q'} ^{t' , r'} (\rn)$, we have
	\begin{equation*}
		\left| \int_\rn f(x) g (x) \d x \right|  \lesssim \int_\rn  \M ^\sharp f (x) \M_\D g (x) \d x \lesssim \int_\rn  \M ^\sharp f (x) \M g (x) \d x .
	\end{equation*}
	Furthermore, by duality, we obtain  for $ f \in M_{p,q}^{t,r} (\rn)$,
	\begin{equation*}
		\| f\|_{ M_{p,q}^{t,r} (\rn) } \lesssim \|  \M ^\sharp f \|_{ M_{p,q}^{t,r} (\rn) } .
	\end{equation*} 
\end{theorem}

\begin{proof}
	We use the idea from \cite[Theorem 389]{SFH20}.
	Let $f \in M_{p,q}^{t,r} (\rn)$  and $g \in  \H _{p',q'} ^{t' , r'} (\rn)$.
	We may assume that $g \in L_c^\infty (\rn)$ and $f \in L^\infty (\rn) \cap \mathbb M ^+ (\rn) $  by the truncation. 
	
	Let $k \in \mathbb Z$ be  fixed. We partition $\{ x \in \rn: \M _{\D} g (x) > 2^k \} $ into a collection of maximal dyadic cubes such that
	\begin{equation*}
		m_{ Q_j^k } (|g|) : = \frac{1}{|Q_j^k|} \int_{Q_j^k }  |g(y) | \d y >2^k.
	\end{equation*}
	That is 
	\begin{equation*}
		\{ x \in \rn: \M _{\D} g (x) > 2^k \} = \bigcup_{j \in J^k} Q_j^k ,
	\end{equation*}
	where $ \{ Q_j^k \}_{ j \in J^k} $ is a disjoint collection of dyadic cubes.
	Since any two dyadic cubes are either disjoint, or one is contained in the other, dyadic cubes that are maximal with respect to some property are necessarily disjoint.
	
	Define 
	\begin{equation*}
		G_k (x):= \chi_{ \{ \M ^{\D} g \le 2^k \} } (x) | g(x) |  + \sum_{j\in J^k} m_{ Q_j^k } (|g|)  \chi_{  Q_j^k  } (x), \quad  B_k := |g| - G_k.
	\end{equation*}
	Note that $G_k  \le G_{k+1}, B_k \ge  B_{k+1}$  for all $k \in \mathbb Z$.
	
	Then we move $k$ now. For each $j ' \in J^{k+1} $, there exists $j \in J^k$ such
	that  $  Q_{j'}^{k+1}  \subset Q_j^k  $. Therefore, for all $j \in J^k$, we have
	\begin{equation*}
		\int_{  Q_j^k }  \left( B_{k+1} (x) - B_k (x)  \right) \d x = 0.
	\end{equation*}
	Note that 
	\begin{equation*}
		\frac{1}{ 2^{-n}|  2_{\operatorname{pa}}  Q_j^k  |}  \int_{  Q_j^k  } | g (x) | \d x  \le 	\frac{2^n}{ |  2_{\operatorname{pa}}  Q_j^k  |}  \int_{ 2_{\operatorname{pa}}  Q_j^k  } | g (x) | \d x \le 2^{ k +n } .
	\end{equation*}	
	In view of the Lebesgue differentiation theorem,
	\begin{align*}
		| B_{k+1} (x) - B_k (x) | = |G_k (x) - G_{k+1} (x) | \le 3 \cdot 2^{n+k}
	\end{align*}
	for almost $x\in \rn$. Since $g\in L^\infty (\rn)$, $ B_k \equiv 0 $ when $ k \ge  \log_2 (\|g\|_{L^\infty (\rn)})  +1 $. Therefore
	\begin{align*}
		\left| \int_\rn f (x) g (x) \d x \right| & \le   \int_\rn f (x) G_k (x) \d x + \int_\rn f (x) B_k (x) \d x \\
		& = \int_\rn f (x) G_k (x) \d x  +  \sum_{l = k} ^\infty \int_\rn f (x) ( B_l (x)  -  B_{l +1} (x)   )\d x  \\
		& = : I_k + II_k .
	\end{align*}
	Next we show $ \lim _{k\to -\infty} I_k = 0$ for $f \in M_{p,q}^{t,r} (\rn)$  and $g \in \H _{p', q'} ^{q',r'} (\rn) $.
	Note that 
	\begin{equation*}
		G_k (x) \le \M_{\D} g (x) \lesssim 2^{ k \theta} \M_{\D} ^{1-\theta} g (x)
		\lesssim
		2^{ k \theta}  \M g (x)^{1-\theta} \quad (x\in \rn).
	\end{equation*}
	(The dyadic Hardy-Littlewood maximal operator can be dominated by the Hardy-Littlewood maximal operator).
	By Corollary \ref{power HL block},	 $ \M g (x)^{1-\theta} \in   \H _{p', q'} ^{q',r'} (\rn)$  as long as $\theta$  is sufficiently small. 
	Thus 
	\begin{align*}
		I_k \lesssim 2^{k\theta } \| f\|_{  M_{p,q}^{t,r} (\rn)}  \| \M g ^{1-\theta}  \|_{\H _{p', q'} ^{q',r'} (\rn)   } = O (2^{k\theta}) .
	\end{align*}
	Hence  $ \lim _{k\to -\infty} I_k = 0$. As for $ II_k $, we have
	\begin{align*}
		II_k & \le  \sum_{l = k} ^\infty  \sum_{ j \in J^l } \int_{Q_j^l }  |f (x)  -m_{ Q_j^l } (f)  |  \cdot | B_l (x)  -  B_{l +1} (x)  | \d x  \\
		& \le 3 \sum_{l = k} ^\infty  2^{n+l} \sum_{ j \in J^l } \int_{Q_j^l }  |f (x)  -m_{ Q_j^l } (f)  |  \d x .
	\end{align*}
	Hence, 
	\begin{align*}
		II_k & \le 3 \sum_{ l = -\infty } ^\infty   2^{n+l} \sum_{ j \in J^l } \int_{  \{\M g >2^l \} }  \M ^\sharp (f)   \d x \\ 
		& \lesssim \| \M ^\sharp (f) \M_\D g \|_{L^1 (\rn)} .
	\end{align*}
	Thus, we obtain
	\begin{equation} \label{sharp fg}
		\left| \int_\rn f(x) g (x) \d x \right|  \lesssim \int_\rn  \M ^\sharp f (x) \M_\D g (x) \d x \lesssim \int_\rn  \M ^\sharp f(x) \M g(x) \d x .
	\end{equation}

In the general case,   considering the real part and the imaginary part of $f$, we may suppose that $f$ is a real function.
 let 
\begin{equation*}
	f_j(x) = \begin{cases}
		-j , &  \operatorname{if \;} f (x) < -j , \\
		f (x), &  \operatorname{if\;} |f(x) | \le j, \\
		j , &  \operatorname{if\;} f (x) >j,
	\end{cases}
\end{equation*}
 and
 \begin{equation*}
 	g_j(x) =  \begin{cases}
 		-j  \chi_{B(0,j) } (x) , &  \operatorname{if \;} g (x) < -j , \\
 		g (x) \chi_{B(0,j) } (x), &  \operatorname{if\;} |g(x) | \le j ,\\
 		j \chi_{B(0,j) } (x) , &  \operatorname{if\;} g (x) >j.
 	\end{cases}
 \end{equation*}
  Then from \cite[exercie 3.1.4]{G142}, we obtain
\begin{equation*}
		\left| \int_\rn f_j (x) g_j (x) \d x \right|  \lesssim \int_\rn  \M ^\sharp f_j  (x) \M_\D g_j (x) \d x \lesssim \int_\rn  \M ^\sharp f(x) \M g (x)\d x .
\end{equation*}
Then by the Fatou Lemma, we have (\ref{sharp fg}).

	Finally, by duality and the boundedneess of Hardy-Littlewood maximal operator on $\H _{p', q'} ^{q',r'} (\rn)$, we get
	\begin{equation*}
		\| f\|_{ M_{p,q}^{t,r} (\rn) } \lesssim \|  \M ^\sharp f \|_{ M_{p,q}^{t,r} (\rn) } .
	\end{equation*} 
	Thus we finish the proof.
\end{proof}

%A weight $\omega$  is a non-negative locally integrable function. The condition $( f )^* (+\infty) = 0$  is equivalent to that the distribution function $d_f (\alpha ) = |\{x \in \rn : |f (x)| > \alpha \}| $ is finite for all $\alpha  > 0.$
%
%\begin{lemma}
%	[Theorem 3,\cite{L04}] For any locally integrable function $f$ with $(\mathcal M f )^*
%	(+\infty) = 0$ and all weights $\omega$,
%	\begin{equation*}
	%			\int_\rn \mathcal M f (x) \omega (x) \d x \lesssim \int_\rn 	\mathcal M ^{\sharp} (f) (x) \mathcal M \omega (x) \d x .
	%		\end{equation*}
%\end{lemma}
%
%
%\begin{theorem}
%		Let $1 < q < \infty$. Let $ 1 <p <t<r <\infty $ or $ 1 <p \le t < r =\infty $. Then for any $f \in M_{p,q}^{t,r} (\rn) $, we have
%		\begin{equation*}
	%				\| f \|_{M_{p,q}^{t,r} (\rn)}  \le \|\mathcal M f \|_{M_{p,q}^{t,r} (\rn)} \lesssim \| \mathcal M^{\sharp} f \|_{M_{p,q}^{t,r} (\rn)} .
	%			\end{equation*}
%\end{theorem}
%
%\begin{proof}
%	Let $f \in M_{p,q}^{t,r} (\rn) $.
%	Since $ \mathcal M$ is bounded on $M_{p,q}^{t,r} (\rn)$(Theorem \ref{HL M}), we have $\mathcal M f \in M_{p,q}^{t,r} (\rn)  $.
%\end{proof}

%\section{Calder\'on-Zygmund operator}
\subsection{Calder\'on-Zygmund operators}
In this subsection, we consider the Calder\'on-Zygmund operator. We first recall their definition.

\begin{definition}
	A function $K(x, y)$ defined on $\{(x,y) \in \rn \times \rn : x \neq y\}$ is called a standard kernel if there exists $\delta, A > 0$ satisfying the size condition
	\begin{equation*}
		|K(x,y)| \le \frac{A}{|x-y|^n}
	\end{equation*}
and the regularity conditions
\begin{equation*}
	|K(x,y) - K (x',y) | \le \frac{A |x-x'|^\delta}{(|x-y|+|x'-y|)^{n+\delta}},
\end{equation*}
whenever $|x- x'| \le  \max( |x-y| , |x' -y |  )  /2 $ and 
\begin{equation*}
	|K(x,y) - K (x,y') | \le \frac{A |y-y'|^\delta}{(|x-y|+|x-y'|)^{n+\delta}},
\end{equation*}
whenever $|y - y'| \le 1/2 \max( |x-y| , |x -y '|  )  /2$. The class of all standard kernels with
constants $\delta$ and $A$ is denoted by $\operatorname{SK}(\delta, A).$
\end{definition}

\begin{definition}
	\label{def CZ op}
	Let $\delta, A > 0$, and $K \in \operatorname{SK}(\delta, A)$. A continuous linear operator $T$ from $\mathscr S (\rn)$ to $\mathscr S' (\rn)$ is said a Calder\'on-Zygmund operator associated with  $K$ if it satisfies
	\begin{equation*}
		T(f) (x) = \int_\rn K (x,y) f (y) \d y
	\end{equation*}
for all $f \in C_c^\infty (\rn)$ and $x$  not in the support of $f$,
and admits a bounded extension on $L^2(\rn)$, i.e., there exists a constant $C>0$ such that 
\begin{equation*}
	\|Tf \|_{L^2(\rn)}  \le C \|f \|_{L^2(\rn)} 
\end{equation*}
for all $f \in L^2(\rn)$.	
\end{definition}

Let $T$ be a  Calder\'on-Zygmund operator as in Definition \ref{def CZ op}. Define $T^*$ by 
\begin{equation*}
	\int_\rn T f(x) g(x) \d x = \int_\rn f (y) T^* g (y) \d y
\end{equation*}
via all $f,g \in C_c^\infty (\rn)$. It is easy to get that  the kernel $K^* $ of $T^*$ belongs to $\operatorname{SK}(\delta, A)$ and $K^* (x,y) = K (y,x)$.

Since $T$  is weak type $(s,s)$ for $s \in[1,\infty)$,  by \cite[Theorem 1.4.19]{G14}, we obtain $T$  is bounded on  $L^{p,q} (\rn)$ for $p,q \in (1,\infty)$.

We first show that  a singular integral operator $T$, which is initially defined in $L^s (\rn)$, for $s \in (1,\infty)$, can be extended to a bounded linear operator on $M_{p,q}^{t,r} (\rn)$. 
\begin{theorem} \label{T r = infty H}
	Let $1<q<\infty$. Let  $ 1 <p \le t < r =\infty $. Let $T$ be a  Calder\'on-Zygmund operator as in Definition \ref{def CZ op}. Then $T$  is bounded on $\H _{p',q'} ^{t' , r'} (\rn)$.
\end{theorem}
\begin{proof}
	Since $r=\infty$, $r'=1$.  Given a block $a$, which is supported on  a cube $Q$  and  satisfy    $  \|a \|_{L^{p',q'}(\rn ) }  \le |Q|^{1/t-1/p}  $, we need to show $\|T a \|_{  \H _{p',q'} ^{t' , 1} (\rn) }  \lesssim 1 $. 
	Since $T$ is bounded on $L^{p',q'}(\rn )$, we obtain $\chi_{2Q}  Ta$  is a $  (p',q', t') $-block modulo a multiplicative constant. Consequently, 
	\begin{equation*}
		\| \chi_{2Q}  Ta \|_{ \H _{p',q'} ^{t' , r'} (\rn) }  \lesssim 1.
	\end{equation*}
For each $k \in \mathbb N$,
\begin{align*}
	| \chi_{2^{k +1} Q  \backslash 2^k Q }  (x) Ta (x) | & \lesssim  \chi_{2^{k +1} Q  \backslash 2^k Q } (x) \int_Q  |x-y|^{-n}  |a (y)| \d y \\
	& \lesssim \frac{1}{2^{kn} |Q| }\int_Q    |a (y)| \d y \\
	& \lesssim \frac{1}{2^{kn} |Q| }    |Q|^{1/t}.
\end{align*}
Hence 
\begin{equation*}
	\| \chi_{2^{k +1} Q  \backslash 2^k Q }  Ta  \|_{ L^{p',q'}  (\rn )}  \lesssim  \frac{1}{2^{kn} |Q| }    |Q|^{1/t} |Q|^{1/p'} \approx 2^{-kn}  |Q|^{1/t - 1/p}.
\end{equation*}
Hence $2^kn  \chi_{2^{k +1} Q  \backslash 2^k Q }  Ta$ is a $  (p',q', t') $-block modulo a multiplicative constant.
Then let $T_a =  \chi_{2Q}  Ta  +  \sum_{k=1}^\infty       2^{-kn}  2^{kn} \chi_{2^{k +1} Q  \backslash 2^k Q }  (x) Ta (x)  $. By Lemma \ref{r'=1 equi *}
\begin{equation*}
	\| T_a \|_{\H _{p',q'} ^{t' , r'} (\rn) }  \lesssim 1+ \sum_{k=1}^\infty  2^{-kn} \lesssim 1.
\end{equation*}
Thus the proof is complete.
\end{proof}

%\begin{lemma}
%	Let $ 1 <p \le t < r =\infty $. Let $ 1 \le q_1 \le q_2 \le \infty$. Then
%	\begin{equation*}
%		M_{p,q_1}^{t,r} (\rn)  \subset M_{p,q_2}^{t,r} (\rn).
%	\end{equation*} 
%If $\theta \le p$, we have
%\begin{equation*}
%		M_{p,\theta}^{t,r} (\rn)  \subset M_{s,z}^{\mu,r} (\rn)
%\end{equation*}
%where $p >s$, $0\le \theta <z$, and  $  t/ \theta <\mu$.
%\end{lemma}
Then we discuss how to extend the  Calder\'on-Zygmund operator   $T$  on $M_{p,q}^{t,r} (\rn)$.
\begin{theorem}
	Let $1<q<\infty$.  Let $ 1 <p <t<r <\infty $ or $ 1 <p \le t < r =\infty $. Let $T$ be a  Calder\'on-Zygmund operator as in Definition \ref{def CZ op}. Then $T$ can be  extended to a 	bounded linear operator on $M_{p,q}^{t,r} (\rn)$. 
\end{theorem}

\begin{proof}
	We use the idea from \cite[Theorem 401]{SFH20}.
	
	We first consider the case $1 <q <\infty $  and $ 1 <p \le t < r =\infty $.
	We divide it into two subcases. 
	
	Subcase $1<p \le q <\infty $ and  $ 1 <p \le t < r =\infty $.
	By the embedding of Lorentz spaces and the H\"older inequality, we obtain
	\begin{equation*}
		|Q|^{1/t-1/p}  \| f \|_{ L^{p,q} (Q)} \lesssim 	|Q|^{1/t-1/p}  \| f \|_{ L^{p,p} (Q)} \le \| f \|_{ L^{t} (Q)}.
	\end{equation*}
Taking the supremum over $Q \in \D$, we get
\begin{equation*}
	\|f\|_{M_{p,q}^{t,r} (\rn) }  \lesssim \| f \|_{L^t (\rn) }.
\end{equation*}
Let $f  \in L^t (\rn)$  and note that $T$ can be extended to a bounded linear operator  on  $L^t (\rn)$.
	
	 Define the functional 
	\begin{equation*}
		g \in \H _{p',q'} ^{t' , r'} (\rn)  \mapsto  \int_\rn f(x)  T^\ast  g (x)  \d x \in \mathbb C. 
	\end{equation*}
Note that this functional is bounded on $\H _{p',q'} ^{t' , r'} (\rn)$ by Theorem \ref{T r = infty H}. Therefore, thanks to Theorem \ref{M is predual of H}, there  exists a unique element $Tf \in  M_{p,q}^{t,r} (\rn) $  such that 
\begin{equation*}
	\int_\rn Tf  (x)  g (x) \d x =  \int_\rn f(x)  T^\ast  g (x)  \d x
\end{equation*}
for all $g \in\H _{p',q'} ^{t' , r'} (\rn) $. We claim that the mapping $f  \in L^t (\rn) \mapsto Tf \in  M_{p,q}^{t,r} (\rn)$  is the desired extension. 

By the Hahn-Banach theorem (for example, see \cite[Theorem 87]{SFH20}), for all $f \in M_{p,q}^{t,r} (\rn)$, we can find $g \in \H _{p',q'} ^{t' , r'} (\rn)$  with $\| g\|_{ \H _{p',q'} ^{t' , r'} (\rn) }  \le 1$  such that 
\begin{equation*}
	\| Tf\|_{  M_{p,q}^{t,r} (\rn) } \approx \int_\rn Tf (x) g (x) \d x.
\end{equation*}
By Theorems \ref{M is predual of H} and \ref{T r = infty H}, we obtain
\begin{equation*}
		\| Tf\|_{  M_{p,q}^{t,r} (\rn) } \lesssim \| f\|_{  M_{p,q}^{t,r} (\rn)}  \|T^* g \|_{ \H _{p',q'} ^{t' , r'} (\rn) }  \lesssim \| f\|_{  M_{p,q}^{t,r} (\rn)}  \| g \|_{ \H _{p',q'} ^{t' , r'} (\rn) } .
\end{equation*}
Taking into account that $g$ has unit norm, we see that $	\| Tf\|_{  M_{p,q}^{t,r} (\rn) }\lesssim \| f\|_{  M_{p,q}^{t,r} (\rn)} $, which is the desired result.

Subcase $1<q < p <\infty $ and  $ 1 <p \le t < r =\infty $. 
From \cite[Theorem 401]{SFH20}, we know that $T$ can be extended to a
bounded linear operator on Morrey spaces  $M_{p_0,p_0}^{t_0,\infty} (\rn)  $ for all  $1<p_0 \le  t_0 <\infty $.
By the embedding (\ref{embedding p - epsilon}),
\begin{equation*}
		M_{p,q}^{t,\infty} (\rn) \hookrightarrow M_{p,\infty}^{t,\infty} (\rn) \hookrightarrow M_{p- \epsilon,p- \epsilon}^{t,\infty} (\rn),
\end{equation*}
we see that   $T$, initially defined on  $M_{p- \epsilon,p- \epsilon}^{t,\infty} (\rn)$, can be defined on $M_{p,q}^{t,\infty} (\rn)$ by restriction.

 Case $1<q  <\infty $ and $ 1 <p < t < r <\infty $. 
 
 By the embedding 
 \begin{equation*}
 	M_{p,q}^{t,r} (\rn) \hookrightarrow M_{p,q}^{t,\infty} (\rn),
 \end{equation*}
 we see that   $T$, initially defined on  $ M_{p,q}^{t,\infty} (\rn)$, can be defined on $M_{p,q}^{t,r} (\rn)$ by restriction.
 Hence we complete the proof. 
\end{proof}

%According to the Calder\'on-Zygmund theory, such an operator $T$ extends to a linear operator defined on any $L^p (\rn)$ for $p \in [1,\infty)$ such that
%\begin{equation*}
%	\| T f\|_{L^p (\rn ) }  \le c_p 	\| f\|_{L^p (\rn ) }  , \quad   \| T f\|_{L^{1,\infty} (\rn ) }  \le c_1 	\| f\|_{L^1 (\rn ) } 
%\end{equation*} 
%for all $f \in L^2 (\rn)$ such that the right-hand side is finite. As a result, for any $p\in [1,\infty) $, $T$ extends to a bounded linear operator on $ L^p (\rn )$, which we still denote by $T$. In \cite[Lemma 3.2]{SE18}, Sawano and El-Shabrawy obtained that $T$, initially defined on $L^p (\rn)$  via $L^2 (\rn) \cap L^p (\rn)$, extends to a bounded linear operator on the closure of $L_c^\infty (\rn)$  in the norm of weak Morrey spaces.
%
%Hence in the sequel, let $\widetilde{ M_{p,q}^{t,r} (\rn)  }$ denote the closure of $ L_c^\infty (\rn)$  in $M_{p,q}^{t,r} (\rn)$.
Now we have the boundedness of Calder\'on-Zygmund operator on  $ M_{p,q}^{t,r} (\rn)  $.
\begin{theorem}
		Let $1 \le q \le \infty$. Let $ 1 <p <t<r <\infty $ or $ 1 <p \le t < r =\infty $. 
		Suppose that $T$ is a  Calder\'on-Zygmund operator as in Definition \ref{def CZ op}. 
		Then $T$ is a bounded  operator from  $ M_{p,q}^{t,r} (\rn)  $ to $M_{p,q}^{t,r} (\rn)  $ .
\end{theorem}
\begin{proof}
	The case $r=\infty$ is in \cite[Lemma 3.3]{DK24}. We only need to prove $r<\infty$. 	We use the idea from \cite[Theorem 4.7]{HNSH23}. 
	We will show 
	\begin{equation*}
		\| Tf\|_{M_{p,q}^{t,r} (\rn) } \lesssim 	\| f\|_{M_{p,q}^{t,r} (\rn) }.
	\end{equation*} 
	Fix $Q \in \mathcal D$. Denote by $Q_k$ be the $k^{\operatorname{th}}$ dyadic parent of $Q$. Then we decompose $f= f_1 +f_2$ where $f_1 = f \chi_Q$. By the $L^{p,q} (\rn) $-boundedness of the Calder\'on-Zygmund operator (see \cite[Theorem 1.1]{CLSS21}), we obtain
	\begin{equation*}
		\|T f_1 \|_{L^{p,q} (Q)} \lesssim 	\| f_1 \|_{L^{p,q} (\rn)} = 	\| f_1 \|_{L^{p,q} (Q)} .
	\end{equation*}
Meanwhile, by the H\"older inequality, for $x\in Q$,
\begin{align*}
	|Tf_2 (x) | & \lesssim \int_{Q^c} \frac{ |f(y) |}{|x-y|^n} \d y = \sum_{k=1}^\infty \int_{Q_k \backslash Q_{k-1} } \frac{ |f(y) |}{|x-y|^n} \d y \\
	& \lesssim \sum_{k=1}^\infty |Q_k|^{-1} \|f \|_{ L^{p,q} (Q_k)} |Q_k|^{1/p'} = \sum_{k=1}^\infty |Q_k|^{-1/p} \|f \|_{ L^{p,q} (Q_k)}.
\end{align*}
Then we conclude
\begin{align*}
 |Q|^{1/t-1/p}	\|Tf_2 \|_{ L^{p,q} (Q) }  & \lesssim  |Q|^{1/t}  \sum_{k=1}^\infty |Q_k|^{-1/p} \|f \|_{ L^{p,q} (Q_k)}  \\
 & =   \sum_{k=1}^\infty 2^{-kn/t} |Q_k|^{1/t-1/p} \|f \|_{ L^{p,q} (Q_k)}  .
\end{align*}
Hence, by Minkowski's inequality, we get
\begin{align*}
	 & \left(\sum_{Q\in \D} \left( |Q|^{1/t-1/p}  \| Tf\|_{ L^{p,q} (Q) }   \right) ^r \right)^{1/r} \\
	& \lesssim \left(\sum_{Q\in \D} \left( |Q|^{1/t-1/p}  \| f\|_{ L^{p,q} (Q) }   \right) ^r \right)^{1/r}  + \sum_{k=1}^\infty 2^{-kn/t}  \left(\sum_{Q\in \D} \left( |Q_k|^{1/t-1/p}  \| f\|_{ L^{p,q} (Q_k) }   \right) ^r \right)^{1/r} \\
	& \lesssim \| f\|_{M_{p,q}^{t,r} (\rn) } + \sum_{k=1}^\infty 2^{-kn/t + kn/r} \| f\|_{M_{p,q}^{t,r} (\rn) } \approx \| f\|_{M_{p,q}^{t,r} (\rn) } .
\end{align*}
Thus, we obtain $ \|Tf\|_{ M_{p,q}^{t,r} (\rn) }  \lesssim \|f\|_{ M_{p,q}^{t,r} (\rn) }$  as desired.
\end{proof}
%Since $L_c^\infty (\rn)$  is dense in $ \H _{p',q'} ^{t' , r'} (\rn)  $ (Theorem \ref{dense block}), we have  $ \widetilde{   \H _{p',q'} ^{t' , r'} (\rn)   } =   \H _{p',q'} ^{t' , r'} (\rn)  $.
By duality, we obtain the Calder\'on-Zygmund operator $T$ is bounded on $ \H _{p',q'} ^{t' , r'} (\rn)  $.
\begin{corollary}
	Let $1 < q < \infty$. Let $ 1 <p <t<r <\infty $ or $ 1 <p \le t < r =\infty $. 	Suppose that $T$ is a  Calder\'on-Zygmund operator as in Definition \ref{def CZ op}. 
	Then $T$ is bounded on $\H _{p',q'} ^{t' , r'} (\rn)  $.
\end{corollary}
\begin{proof}
	Let $T$ be a Calder\'on-Zygmund operator associated with  $K \in \operatorname{SK}(\delta, A)$. Then $T^*$  is also a Calder\'on-Zygmund operator associated with  $K^* (x,y)  = {K (y,x)}  \in \operatorname{SK}(\delta, A)$.
	Then
	\begin{align*}
		\|T g\|_{ \H _{p',q'} ^{t' , r'} (\rn) }  & \approx  \sup_{ \|f\|_{ M_{p,q}^{t,r} (\rn) =1 } }   \left| \int_\rn Tg (x) f (x )  \d x \right| \\
		& = \sup_{ \|f\|_{ M_{p,q}^{t,r} (\rn) } =1 }   \left| \int_\rn g (y) T^* f (y )  \d y \right| \\ 
		& \le \sup_{ \|f\|_{ M_{p,q}^{t,r} (\rn) } =1 }  \| g \|_{ \H _{p',q'} ^{t' , r'} (\rn)  }  \| T^* f \| _{  M_{p,q}^{t,r} (\rn)}  \\ 
		& \lesssim  \| g \|_{ \H _{p',q'} ^{t' , r'} (\rn)  }.
	\end{align*}
Thus, the proof is finished.
\end{proof}

\subsection{Commutators}
In this section, we consider  the boundedness of the commutators generated by the Calder\'on-Zygmund operators and BMO functions on Bourgain-Morrey-Lorentz spaces. To do so, we first recall the definition of bounded mean oscillation.

\begin{definition}
	For $f$ a complex valued locally integrable function on $\rn$, set 
	\begin{equation*}
		\|f \|_{\BMO} = \sup_Q \frac{1}{|Q|} \int_Q |f(y)- f_Q | \d y,
	\end{equation*} 
where $f_Q = \int_Q f (y) \d y /|Q| $, and the supremum is taken over all cubes $Q$ in $\rn$. The function $f$ is of bounded mean oscillation if $	\|f \|_{\BMO} <\infty$. $\BMO$ is the set of locally integrable such that $ 	\|f \|_{\BMO} <\infty $.

The $\BMO$-closure of $C_c^\infty (\rn)$, is denoted by $\CMO$. 
\end{definition}

Let $T$ be a  Calder\'on-Zygmund operator as in Definition \ref{def CZ op} and $b \in \BMO$. The commutator $[b,T]$ is defined by
\begin{equation*}
	[b,T] f (x) = b (x) T (f) (x) - T(bf) (x),  \quad  x\in \rn.
\end{equation*}
It is easy see that for $f , g\in C_c^\infty (\rn)$ with supp $f \cap$ supp $g = \emptyset$,  
\begin{align*}
	\int_\rn  	[b,T] f (x) g(x)\d x & = \int_\rn f (y) \int_\rn \left( b (x) - b (y)\right) K (x,y) g (x) \d x \d y\\
	& =\int_\rn f (y)  \left(  T^* (bg) (y) - b (y) T^* (g) (y) \right) \d y.
\end{align*}
Since $L^\infty_c (\rn)$  is dense in $\H _{p',q'} ^{t' , r'} (\rn)$,  the commutator $[b,T] $, which is initially defined in $L^\infty_c (\rn)$, can be  extended  on  block spaces $\H _{p',q'} ^{t' , r'} (\rn)$.
\begin{definition}
	Let $1 < q < \infty$. Let $ 1 <p <t<r <\infty $ or $ 1 <p \le t < r =\infty $. Let $b \in \BMO$ and  $T$ be a  Calder\'on-Zygmund operator as in Definition \ref{def CZ op}. Then the  commutator $[b,T] : M_{p,q}^{t,r} (\rn) \to M_{p,q}^{t,r} (\rn) $  is defined by way of the duality  $ \H _{p',q'} ^{t' , r'} (\rn)  - M_{p,q}^{t,r} (\rn)  $; for $f \in M_{p,q}^{t,r} (\rn) $, $ [b,T] f \in M_{p,q}^{t,r} (\rn) $   is a unique element $h \in M_{p,q}^{t,r} (\rn)$ such that for all $g \in  \H _{p',q'} ^{t' , r'} (\rn)$
	\begin{equation*}
		\int_\rn h (x) g (x) \d x = - \int_\rn f (x) [b,T^\ast]  g (x) \d x .
	\end{equation*}
\end{definition}

Next we recall the following lemma. Its proof is similar with  \cite[p. 418-419]{T86}. The similar result  for Morrey spaces  can be seen in \cite[Lemma 1]{DGR91}. Dao and Krantz used \cite[Lemma 1]{DGR91} to prove the boundedness of commutator $[b,T]$ on Morrey-Lorentz spaces. 
\begin{lemma} \label{est commutator}
	Let $1 < q < \infty$. Let $ 1 <p <t<r <\infty $ or $ 1 <p \le t < r =\infty $.
	Let $T$ be a  Calder\'on-Zygmund operator as in Definition \ref{def CZ op} and $b\in \BMO$. Then for $f \in M_{p,q}^{t,r} (\rn) $ and $1<s,t<\infty$,
	\begin{equation*}
		\M ^{\sharp} ( [b,T] f )  (x) \lesssim \| b\|_{\BMO} \left(  \M _s  ( T f ) (x) + \M_t ( f) (x) \right), \quad x\in \rn.
	\end{equation*}
\end{lemma}
Now we have the boundedness of commutators  on Bourgain-Morrey-Lorentz spaces.
\begin{theorem}
	Let $1 < q < \infty$. Let $ 1 <p <t<r <\infty $ or $ 1 <p \le t < r =\infty $. Suppose that $T$ is a  Calder\'on-Zygmund operator as in Definition \ref{def CZ op} and $b\in \BMO$. Then  $[b,T]$ is a bounded operator on  $ M_{p,q}^{t,r} (\rn)  $.
\end{theorem}
\begin{proof}
	The case $r=\infty$ is in \cite[Theorem 1.5]{DK24}. We only need to prove $r<\infty$. Let $ f \in  M_{p,q}^{t,r} (\rn)$.
	Chose $\eta \in (1, \min(p,q))$. By Lemma \ref{est commutator}, we have
	\begin{equation*}
			\M ^{\sharp} ( [b,T] f )  (x) \lesssim \| b\|_{\BMO} \left(  \M _\eta  ( T f ) (x) + \M_\eta ( f) (x) \right).
	\end{equation*}
	By Lemma \ref{sharp BML}, we obtain
	\begin{align*}
		\| [b,T] f \|_{ M_{p,q}^{t,r} (\rn) } & \lesssim \| \M^{\sharp}  [b,T] f \|_{ M_{p,q}^{t,r} (\rn) } \\
		& \lesssim \| b\|_{\BMO}  \left( \| \M _\eta  ( T f )  \|_{ M_{p,q}^{t,r} (\rn) }  + \|  \M_\eta ( f)\|_{ M_{p,q}^{t,r} (\rn) }  \right)\\
		& \lesssim \| b\|_{\BMO}  \left( \|    T f   \|_{ M_{p,q}^{t,r} (\rn) }  + \|  f\|_{ M_{p,q}^{t,r} (\rn) }  \right)\\
		& \lesssim \| b\|_{\BMO}  \|  f\|_{ M_{p,q}^{t,r} (\rn) }  .
	\end{align*}
	Thus, we obtain $ \| [b,T]f\|_{ M_{p,q}^{t,r} (\rn) }  \lesssim \|f\|_{ M_{p,q}^{t,r} (\rn) }$  as desired. 
\end{proof}
\begin{remark}
	By duality, we also obtain the boundedness of $[b,T]$ in $\H _{p',q'} ^{t' , r'} (\rn)$. In addition, we have
	\begin{equation*}
	 \| [b,T]f \|_{ \H _{p',q'} ^{t' , r'} (\rn) }  \lesssim \| b\|_{BMO} \|f\|_{ \H _{p',q'} ^{t' , r'} (\rn) } .
	\end{equation*}
\end{remark}

Next, we have the following  Minkowski type inequality for Bourgain-Morrey-Lorentz spaces. 
\begin{lemma} \label{Minkowski type BML}
	Let  $1 < q < \infty$. Let $ 1 <p <t<r <\infty $ or $ 1 <p \le t < r =\infty $. Then
	\begin{equation*}
			\left\|   \int_\rn f (\cdot, y) \d y \right\|_{ M_{p,q}^{t,r} (\rn) } \lesssim  \int_\rn     \| f (\cdot, y) \|_{  M_{p,q}^{t,r} (\rn) } \d y  .
	\end{equation*}
\end{lemma}
\begin{proof}
	By Theorem \ref{M is predual of H}, we obtain
\begin{align*}
	\left\|   \int_\rn f (\cdot, y) \d y \right\|_{ M_{p,q}^{t,r} (\rn) } & \approx \sup_{ \|g\|_{ \H _{p',q'} ^{t' , r'} (\rn) }  \le 1 }  \left| \int_\rn g (x) \int_\rn f (x, y) \d y  \d x \right| \\
	& =  \sup_{ \|g\|_{ \H _{p',q'} ^{t' , r'} (\rn) }  \le 1 }   \int_\rn\int_\rn  |  g (x) f (x, y) | \d x \d y    \\
	& \le  \sup_{ \|g\|_{ \H _{p',q'} ^{t' , r'} (\rn) }  \le 1 }   \int_\rn   \|  g\|_{\H _{p',q'} ^{t' , r'} (\rn)   }   \| f (\cdot, y) \|_{  M_{p,q}^{t,r} (\rn) } \d y    \\
	& =  \int_\rn     \| f (\cdot, y) \|_{  M_{p,q}^{t,r} (\rn) } \d y  .
\end{align*}
Thus, we obtain the desired result.
\end{proof}
\begin{remark}
	Lemma \ref{Minkowski type BML} can not be obtained from Corollary \ref{convolution BML} since $\chi_\rn \notin L^1 (\rn)$.
\end{remark}

\section{Hardy factorization in Bourgain-Morrey-Lorentz spaces} \label{Hardy factorization Section}
In this section, we establish the Hardy factorization in Bourgain-Morrey-Lorentz spaces. For the theory of Hardy spaces, we refer the reader to   \cite[Chapter 2]{G142}. To our purpose, we use the  $L^\infty$-atom definition of the Hardy space $H^1(\rn)$.

\begin{definition}
	A complex-valued function $a$  is an $L^\infty$-atom if it is supported on a cube $Q$, $\int_\rn a(x) \d x =0$ and $\|a\|_{L^\infty (\rn)}  \le |Q|^{-1}$. We now denote the Hardy space $H^1(\rn)$ by 
	\begin{equation*}
		H^1(\rn) = \left\{    \sum_{k \ge 1} \lambda_k a_k : a_k \; \operatorname{atoms}, \lambda\in \mathbb C, \sum_{k\ge 1} |\lambda_k | <\infty  \right\} .
	\end{equation*}
We define a norm on $H^1(\rn)$ by 
\begin{equation*}
	\| f \|_{	H^1(\rn)}  = \inf \left\{ \sum_{k\ge 1} |\lambda_k |  : f =   \sum_{k \ge 1} \lambda_k a_k      \right\}.
\end{equation*}
\end{definition}

The following lemma is a basic result of $H^1 (\rn)$. 
\begin{lemma}[Lemma 4.3, \cite{KT06}] \label{H1 basic}
	Let $x_0, y_0 \in \rn$ be such that $|x_0 - y_0 | = MR$  for some $R>0$  and $M >10$. If $ \int_\rn F (x) \d x =0$  and for all $x\in \rn $, $ |F(x) | \le R^{-n}\left(  \chi_{B (x_0,R)}  +  \chi_{B (y_0,R)} \right)  $
then there is a constant $C = C_n >0$ such that
\begin{equation*}
	\| F\|_{H^1 (\rn)}  \le C \log  M.
\end{equation*}
\end{lemma}

\begin{definition}
	We say that $T$ is homogeneous if, for any cube $Q = Q(c_Q, \ell(Q)/2)$  in $\rn$, $T$ satisfies
	\begin{equation*}
		|T (\chi_\Omega) | (x) | \ge  \frac{|\Omega|}{M^n |Q|}
	\end{equation*}
	for any Borel set $\Omega \subset Q$, and for all $|x -c_Q | = M \ell (Q) /2$, $M >10$.
\end{definition}
\begin{remark}
	The Riesz transforms and the Cauchy integral operators associated with Lipschitz curves are homogeneous. 
%	The Szeg\"{o}
%operators, defined on the smooth boundary of strongly pseudo-convex domains
%in $\mathbb C^n$ is also  homogeneous.
\end{remark}

Then we have the following result.

\begin{lemma} \label{H1 lemma}
	Let $T$ be a Calder\'on-Zygmund operator associated with  $K  \in \operatorname{SK}(\delta, A)$. 
	 Suppose that $T$ is homogeneous.
	If $f \in H^1 (\rn)$ can be written as $ f =\sum_{k\ge 1 }  \lambda_k a _k $, then there exists $ \{ g_k\}_{k\ge 1 } , \{ h_k\}_{k\ge 1 }  \subset L_c^\infty (\rn)$ such that 
	\begin{equation*}
		\| a_k - [ g_k T^* (h_k) - h_k T (g_k) ] \|_{H^1 (\rn)}  \lesssim \frac{\log M}{ M^\delta}
	\end{equation*}
and \begin{equation*}
	\sum_{k\ge 1} |\lambda_k| \| g_k \|_{\H _{p',q'} ^{t' , r'} (\rn) } 	\| h_k\|_{ M_{p,q}^{t,r} (\rn) } \le C M^n \|f\|_{H^1(\rn)},
\end{equation*}
where $M >0$ is sufficiently large. Furthermore, we have
\begin{equation*}
	\left\|  f - \sum_{k \ge 1}  \lambda_k \left(   g_k T^* (h_k) - h_k T (g_k)   \right)  \right\|_{ H^1 (\rn)}  \le \frac{1}{2} \| f\|_{H^1(\rn)} .
\end{equation*}
\end{lemma}

\begin{proof}
	We use the idea from \cite[Lemma 4.2]{DK24}.
	Let $a$ be an $L^\infty$-atom for $H^1 (\rn)$, supported on $Q $ in $\rn$ such that 
	\begin{equation*}
		\|a\|_{L^\infty}  \le |Q|^{-1} \quad  \operatorname{and} \quad \int_\rn a (x) \d x=0.
	\end{equation*}
Let $M \ge 10$  be a real number, which will be determined later. Let $y_0 \in \rn$ be such that $|c_Q - y_0 | = M \ell (Q) /2$. Next, we set
\begin{equation*}
	g (x) = \chi_{ Q (y_0, \ell(Q) /2) } (x)  \quad \operatorname{and }
\quad h (x) = -\frac{ a(x)}{T(g) (c_Q)}.
 \end{equation*}
It is obvious that these functions are in $L_c^\infty  (\rn) $. In addition, since $T$ is homogeneous, then there exists a constant $C>0$ such that
\begin{equation*}
	| T (g) (c_Q)  | \ge C M^{-n}.
\end{equation*}
Note that   $\| g \|_{\H _{p',q'} ^{t' , r'} (\rn) }  \approx \ell (Q) ^{n/t'} $ and
\begin{equation*}
	\| h\|_{ M_{p,q}^{t,r} (\rn) } = \frac{\|a\|_{ M_{p,q}^{t,r} (\rn) }}{ |  T (g) (c_Q) |} \le C M^n |Q|^{-1} |Q|^{1/t} .
\end{equation*}
Therefore,
\begin{equation*}
	\| g \|_{\H _{p',q'} ^{t' , r'} (\rn) } 	\| h\|_{ M_{p,q}^{t,r} (\rn) }  \le C M^n.
\end{equation*}
Next we claim that 
\begin{equation} \label{F le log M / M^delta}
		\| a - (  g T^* (h) - h T (g) ) \|_{H^1 (\rn)}  \le C \frac{\log M}{ M^\delta}.
\end{equation}
Indeed, let us set
\begin{equation*}
	F =a - (   g T^* (h) - h T (g) ).
\end{equation*}
Then,
\begin{align*}
	|F (x) | & \le |a (x) +  h T (g) | + |g T^* (h) | \\
	& =\left|  \frac{ a(x) \left( T(g) (c_Q) -T(g) (x) \right)   }{T(g) (c_Q)} \right|  +  |g T^* (h) | =: J_1 (x) + J_2 (x).
\end{align*}
For $J_1$, since  supp $a \subset Q$, then it suffices to consider $x\in Q$. Therefore, using the regularity conditions of $K$, we obtain
\begin{align*}
	|J_1 (x) | & \lesssim  M^n \chi_{Q} (x) \|a\|_{L^\infty(\rn)} \int_{Q (y_0, \ell(Q) /2)} |K(c_Q, y) -K (x,y)| \d y \\
	& \lesssim M^n \chi_{Q} (x) |Q|^{-1}\int_{Q (y_0, \ell(Q) /2)} \frac{|c_Q - x|^\delta}{ |y-c_Q|^{n+\delta}} \d y \\
	& \lesssim M^n \chi_{Q} (x) |Q|^{-1} \frac{\ell(Q)^\delta }{ ( M \ell (Q) )^{n+ \delta}    }  \ell(Q) ^n \\
	& \approx \chi_{Q} (x) M^{-\delta} \ell(Q) ^{-n} .
\end{align*}
As for $J_2 (x)$, since $ \int_\rn a (x) \d x =0$  and supp $a \subset Q$, then we have
\begin{align*}
	|J_2 (x) | & \le  \chi_{ Q (y_0, \ell(Q) /2) } (x)  \left| \int_Q  K (z,x)  \left(\frac{-a(z)}{ T(g) (c_Q) }\right) \d z \right| \\
	& =\chi_{ Q (y_0, \ell(Q) /2) } (x)  \left| \int_Q  \left(    K(c_Q, x) -K (z,x) \right)  \left(\frac{a(z)}{ T(g) (c_Q) }\right) \d z \right| .
\end{align*}
Similar to $J_1$, we obtain $J_2 (x) \lesssim \chi_{ Q (y_0, \ell(Q) /2) } (x) M^{-\delta} \ell(Q) ^{-n} .$ Hence
\begin{equation*}
		|F (x) | \le J_1 (x) + J_2 (x) \lesssim  M^{-\delta} \ell(Q) ^{-n} \left(  \chi_{Q} (x)    + \chi_{ Q (y_0, \ell(Q) /2) } (x)\right).
\end{equation*}
By applying  Lemma \ref{H1 basic} to the function $F(x)$, we obtain (\ref{F le log M / M^delta}).

Now, we can apply (\ref{F le log M / M^delta}) to $a = a_k$, for $k \ge 1$. Then, there exist functions $ \{ g_k\}_{k \ge 1}$, $ \{ h_k\}_{k \ge 1} \subset L_c^\infty (\rn)$, such that
	\begin{equation*}
	\| a_k - [ g_k T^* (h_k) - h_k T (g_k) ] \|_{H^1 (\rn)}  \lesssim \frac{\log M}{ M^\delta} . 
\end{equation*}
With this inequality above, we obtain
\begin{align*}
 & 	\left\|  f - \sum_{k \ge 1}  \lambda_k \left(   g_k T^* (h_k) - h_k T (g_k)   \right)  \right\|_{ H^1 (\rn)} \\
	& \le \sum_{k\ge 1} |\lambda_k| 	\| a_k - [ g_k T^* (h_k) - h_k T (g_k) ] \|_{H^1 (\rn)}  \\
	& \le C\frac{\log M}{ M^\delta} \sum_{k\ge 1} |\lambda_k|  \le \frac{1}{2} \| f\|_{H^1(\rn)}
\end{align*}
if we choose $M$ large enough. Thus the proof is finished.
\end{proof}

The following theorem is the main result of this section.

\begin{theorem} \label{Hardy factorization}
		Let  $1 < q < \infty$. Let $ 1 <p <t<r <\infty $ or $ 1 <p \le t < r =\infty $. Let $T$ be a Calder\'on-Zygmund operator associated with  $K  \in \operatorname{SK}(\delta, A)$. Suppose that $T$ is homogeneous. Then for each $f\in  H^1(\rn)$, there exists a sequence $  \{\lambda_{k,j}\} \in \ell^1 $  and functions $\{g_{k,j} \} , \{ h_{k,j} \}  \subset L_c^\infty (\rn)$ such that
		\begin{equation} \label{f = sum lambda}
			f =  \sum_{j=1}^\infty \sum_{k=  1} ^\infty \lambda_{k,j} \left(   g_{k,j} T^* (h_{k,j}) - h_{k,j} T (g_{k,j})   \right)
		\end{equation}
	in the sense of $ H^1(\rn)$. Moreover, we have
	\begin{equation*}
		\| f\|_{H^1(\rn)}  \approx \inf \left\{\sum_{j=1}^\infty \sum_{k=  1} ^\infty | \lambda_{k,j}|  \| g_{k,j} \|_{\H _{p',q'} ^{t' , r'} (\rn) } 	\| h_{k,j}\|_{ M_{p,q}^{t,r} (\rn) }
		\right\},
	\end{equation*}
where the infimum above is taken over all possible representations of $f$ that satisfy (\ref{f = sum lambda}).
\end{theorem}

\begin{proof}
	We can suppose that $f \in H^1(\rn)$ can be written
	\begin{equation*}
		f = \sum_{k\ge 1}  \lambda_{k,1} a_{k,1} ,
	\end{equation*}
with $\{\lambda_{k,1} \}_{k\ge 1}  \in\ell^1$, and $ \{a_{k,1} \}_{k\ge 1}$ are $L^\infty$-atoms.
Thanks to Lemma \ref{H1 lemma}, there exist functions $\{g_{k,1}\} , \{ h_{k,1} \}  \subset L_c^\infty (\rn)$ such that
\begin{equation*}
	\begin{cases}
		\sum_{k\ge 1} |\lambda_{k,1}| \| g_{k,1} \|_{\H _{p',q'} ^{t' , r'} (\rn) } 	\| h_{k,1}\|_{ M_{p,q}^{t,r} (\rn) } \le C M^n \|f\|_{H^1(\rn)}	\\
			\left\|  f - \sum_{k \ge 1}  \lambda_{k,1} \left(   g_{k,1} T^* (h_{k,1}) - h_{k,1} T (g_{k,1})   \right)  \right\|_{ H^1 (\rn)}  \le \frac{1}{2} \| f\|_{H^1(\rn)} .
	\end{cases}
\end{equation*}
Then let 
\begin{equation*}
	f_1 = f - \sum_{k \ge 1}  \lambda_{k,1} \left(   g_{k,1} T^* (h_{k,1}) - h_{k,1} T (g_{k,1})   \right).
\end{equation*}
Since $f_1 \in  H^1(\rn) $, then we can decompose 
\begin{equation*}
	f_1 = \sum_{k \ge 1} \lambda_{k,2} a_{k,2},
\end{equation*}
with $\{\lambda_{k,2} \}_{k\ge 1}  \in\ell^1$, and $ \{a_{k,2} \}_{k\ge 1}$ are $L^\infty$-atoms.
Applying Lemma \ref{H1 lemma} to $f_1$, there exist functions $\{g_{k,2}\} , \{ h_{k,2} \}  \subset L_c^\infty (\rn)$ such that
\begin{equation*}
	\begin{cases}
		\sum_{k\ge 1} |\lambda_{k,2}| \| g_{k,2} \|_{\H _{p',q'} ^{t' , r'} (\rn) } 	\| h_{k,2}\|_{ M_{p,q}^{t,r} (\rn) } \le C M^n \|f_1\|_{H^1(\rn)} \le  C M^n 2^{-1} \|f\|_{H^1(\rn)}	\\
		\left\|  f - \sum_{k \ge 1}  \lambda_{k,2} \left(   g_{k,2} T^* (h_{k,2}) - h_{k,2} T (g_{k,2})   \right)  \right\|_{ H^1 (\rn)}  \le \frac{1}{2} \| f_1\|_{H^1(\rn)}   \le 2^{-2} \| f\|_{H^1(\rn)}.
	\end{cases}
\end{equation*}
Similarly, we can repeat the above argument to 
\begin{align*}
	f_2 & = f_1 - \sum_{k \ge 1}  \lambda_{k,2} \left(   g_{k,2} T^* (h_{k,2}) - h_{k,2} T (g_{k,2})   \right) \\
	& = f  -\sum_{k \ge 1}  \lambda_{k,1} \left(   g_{k,1} T^* (h_{k,1}) - h_{k,1} T (g_{k,1})   \right)
	-\sum_{k \ge 1}  \lambda_{k,2} \left(   g_{k,2} T^* (h_{k,2}) - h_{k,2} T (g_{k,2})   \right) .
\end{align*}
By induction, we can construct sequence $\{\lambda_{k,j}\} \in \ell^1$, and functions $\{g_{k,j} \} , \{ h_{k,j} \}  \subset L_c^\infty (\rn)$ such that
\begin{equation*}
	\begin{cases}
		f = \sum_{j=1}^N \sum_{k\ge 1} \lambda_{k,j} \left(   g_{k,j} T^* (h_{k,j}) - h_{k,j} T (g_{k,j})   \right) +f_N, \\
		\sum_{j=1}^N  \sum_{k\ge 1} | \lambda_{k,j}|  \| g_{k,j} \|_{\H _{p',q'} ^{t' , r'} (\rn) } 	\| h_{k,j}\|_{ M_{p,q}^{t,r} (\rn) } \le C M^n  \sum_{j =1} ^N  2^{1-j} \|f\|_{H^1(\rn)}, \\
		\| f_N\|_{H^1(\rn)} \le 2^{-N}   \|f\|_{H^1(\rn)} ,
	\end{cases}
\end{equation*}
which yields the desired result when $N \to \infty$. Thus, the proof is complete.
\end{proof}

Using the Hardy factorization, we obtain a characterization of functions $b \in \BMO$ through the boundedness of commutators  on the Bourgain-Morrey-Lorentz space.

\begin{theorem} \label{char b commutator}
	Let  $1 < q < \infty$. Let $ 1 <p <t<r <\infty $ or $ 1 <p \le t < r =\infty $. Let $T$ be a Calder\'on-Zygmund operator associated with  $K  \in \operatorname{SK}(\delta, A)$. Let $b \in L_{\operatorname{loc}}^1 (\rn)$ and $T$ be homogeneous. Then $b\in \BMO$ and
	\begin{equation*}
			\|b \|_{BMO} \lesssim  \|[b,T]\|_{  M_{p,q}^{t,r} (\rn)   \to  M_{p,q}^{t,r} (\rn)  }.
	\end{equation*}
\end{theorem}

\begin{proof}
	To obtain the desired result, we use the Hardy factorization in Theorem \ref{Hardy factorization}, and the duality between $\BMO $ and $H^1 (\rn)$. Note that $ H^1 (\rn) \cap L_c^\infty (\rn)$ is dense in  $ H^1 (\rn)$. 
	Next, for every $L >0$, put
	\begin{equation*}
		b_L (x)  = b(x) \chi_{B(0,L)} (x).
	\end{equation*}
For each $f\in  H^1 (\rn) \cap L_c^\infty (\rn)$, using Theorem \ref{Hardy factorization}, there exist functions  $\{g_{k,j} \} , \{ h_{k,j} \}  \subset L_c^\infty (\rn)$  and   a sequence $  \{\lambda_{k,j}\} \in \ell^1 $  such that
		\begin{equation*}
	f =  \sum_{j=1}^\infty \sum_{k=  1} ^\infty \lambda_{k,j} \left(   g_{k,j} T^* (h_{k,j}) - h_{k,j} T (g_{k,j})   \right)
\end{equation*}
in the sense of $ H^1(\rn)$. Furthermore, we have
\begin{equation*}
	\|f\|_{H^1 (\rn)}  \approx \sum_{j=1}^\infty \sum_{k=  1} ^\infty | \lambda_{k,j}|  \| g_{k,j} \|_{\H _{p',q'} ^{t' , r'} (\rn) } 	\| h_{k,j}\|_{ M_{p,q}^{t,r} (\rn)}.
\end{equation*}
Since $f \in H^1 (\rn) \cap L_c^\infty (\rn) $, then we obtain
\begin{equation*}
	\lim_{L \to \infty} \langle b_L, f \rangle_{re} = \langle b, f \rangle_{re}
\end{equation*}
where we denote $\langle b, f \rangle_{re} = \int_\rn  b (x) f (x) \d  x $.
Thus,\begin{align*}
	 \langle b, f \rangle_{re} =\lim_{L \to \infty} \langle b_L, f \rangle_{re} &= \lim_{L \to \infty} \left\langle b_L,  \sum_{j=1}^\infty \sum_{k=  1} ^\infty \lambda_{k,j} \left(   g_{k,j} T^* (h_{k,j}) - h_{k,j} T (g_{k,j})   \right) \right\rangle_{re}  \\
	 & = \sum_{j=1}^\infty \sum_{k=  1} ^\infty \lambda_{k,j}  \lim_{L \to \infty} \left\langle b_L,    g_{k,j} T^* (h_{k,j}) - h_{k,j} T (g_{k,j})    \right\rangle_{re}  \\
	 	 & = \sum_{j=1}^\infty \sum_{k=  1} ^\infty \lambda_{k,j}  \left\langle b,   \left(   g_{k,j} T^* (h_{k,j}) - h_{k,j} T (g_{k,j})   \right) \right\rangle_{re}  \\
	 	 	 	 & = \sum_{j=1}^\infty \sum_{k=  1} ^\infty \lambda_{k,j}  \left\langle [b,T]( g_{k,j}) ,   h_{k,j}   \right\rangle_{re}  .
\end{align*}
Note that in the fourth step, we use the fact that $  g_{k,j} T^* (h_{k,j}) - h_{k,j} T (g_{k,j})    \in H ^1 (\rn)$ and the support of $  g_{k,j} T^* (h_{k,j}) - h_{k,j} T (g_{k,j}) $ is compact.

Since $[b,T]$ maps $\H _{p',q'} ^{t' , r'} (\rn) \to \H _{p',q'} ^{t' , r'} (\rn)  $, it follows that
\begin{align*}
	|\langle b, f \rangle_{re}| & \le\sum_{j=1}^\infty \sum_{k=  1} ^\infty \lambda_{k,j}  \|  [b,T]( g_{k,j}) \|_{ \H _{p',q'} ^{t' , r'} (\rn)  } \| h_{k,j}\|_{ M_{p,q}^{t,r} (\rn) } \\
	& \lesssim \|[b,T]\|_{ \H _{p',q'} ^{t' , r'} (\rn) \to \H _{p',q'} ^{t' , r'} (\rn)  } \sum_{j=1}^\infty \sum_{k=  1} ^\infty \lambda_{k,j}  \|  g_{k,j} \|_{ \H _{p',q'} ^{t' , r'} (\rn)  } \| h_{k,j}\|_{ M_{p,q}^{t,r} (\rn) } \\
	& \lesssim \|[b,T]\|_{ \H _{p',q'} ^{t' , r'} (\rn) \to \H _{p',q'} ^{t' , r'} (\rn)  } \|f\|_{  H^1 (\rn) }.
\end{align*}
Next by (\cite[exercise 3.2.1]{G142})
\begin{equation*}
	\|b\|_{\BMO} \approx \sup_{ f \in H^1(\rn) \le 1 } \left| \int_\rn b (x) f (x) \d x \right| \approx \sup_{ f \in H^1(\rn) \le 1 } \left| \int_\rn b (x)  \overline{f (x)} \d x \right|,
\end{equation*}
we obtain
\begin{equation*}
	\|b \|_{\BMO} \lesssim  \|[b,T]\|_{ \H _{p',q'} ^{t' , r'} (\rn) \to \H _{p',q'} ^{t' , r'} (\rn)  }.
\end{equation*}
Hence we complete the proof.
\end{proof}

\begin{remark}
	Thanks to  the duality, we observe that the conclusion in Theorem \ref{char b commutator} also holds for $\H _{p',q'} ^{t' , r'} (\rn) $  in place of  $M_{p,q}^{t,r} (\rn) $.
\end{remark}

\section{Compactness of commutators} \label{Compactness}
In the last section, we consider the compactness characterization of commutators in Bourgain-Morrey-Lorentz spaces.

We have proved that the Bourgain-Morrey-Lorentz space is a translation-invariant
lattices on $\rn$. 
By  dominated converge theorem, for every sequence $\{ S_k\}_{k \in \mathbb N} \subset \rn$ satisfying $S_k \subset S_{k-1}  \subset \cdots \subset S_1$ and $\cap_{k \in \mathbb N} S_k = \emptyset$, we have
\begin{equation*}
	\lim_{n \to \infty}  \| f \chi_{S_n} \|_{M_{p,q}^{t,r} (\rn) }  =0.
\end{equation*} 
By \cite[Theorem 2.3]{B15}, we have the following compactness criterion in Bourgain-Morrey-Lorentz spaces.

\begin{lemma} \label{precompact}
	Let $1<q <\infty  $.
	Let $1 \le  p < t < r <\infty$ or $ 1 \le p \le t < r= \infty $. 
	A subset $\mathcal F \subset M_{p,q}^{t,r} (\rn) $ is totally bounded if the following conditions are valid:
	
	{\rm (ii)} $\mathcal F$ is a bounded set, that is, $\sup_{f \in \mathcal F}  \|f\|_{M_{p,q}^{t,r} (\rn) } <\infty;$
	
	{\rm  (ii)} $\mathcal F$ uniformly vanishes at infinity, that is,
	\begin{equation*}
		\lim_{R\to \infty} \sup _{f \in  \mathcal F} 	\|  f \chi_{B^c (0,R)}  \|_{  M_{p,q}^{t,r} (\rn)  }  =0;
	\end{equation*}
	
	{\rm  (iii)} $\mathcal F$ is equicontinuous, that is,
	\begin{equation}\label{trans con}
		\lim _{b \to 0} \sup _{ f \in  \mathcal F} \sup_{y\in B(0,b)} 	\|  f  - \tau_y   f  \|_{  M_{p,q}^{t,r} (\rn) } =0.
	\end{equation}
	Here and what follows, $\tau _y$ denotes the translation operator: $\tau_y  f (x) :=  f (x-y)$.
	
\end{lemma}
Similar to \cite[Theorem 1.6]{DK24}, we have the following result about the compactness of the commutator of Calder\'on-Zygmund operator and functions $b \in \CMO$ for Bourgain-Morrey-Lorentz spaces.

\begin{theorem}
		Let  $1 < q < \infty$. Let $ 1 <p <t<r <\infty $ or $ 1 <p \le t < r =\infty $.  
		Let $T$ be a Calder\'on-Zygmund operator associated with  $K  \in \operatorname{SK}(\delta, A)$.
		If $b \in \CMO$, then $[b,T] $  is a compact  operator on  $  M_{p,q}^{t,r} (\rn) $. Conversely, for $b \in L_{\operatorname{loc}}^1 (\rn)$, 
		if $T$ is homogeneous, and $[b,T]$ is a compact operator  on $  M_{p,q}^{t,r} (\rn) $, then $b\in \CMO$.
\end{theorem}

\begin{proof}
		 Assume that $b \in \CMO$. Let $E$ be a bounded set in $M_{p,q}^{t,r} (\rn) $. It is enough to show that $[b,T] (E)$ is relatively compact in $M_{p,q}^{t,r} (\rn) $. 
	Since $b \in \CMO$, then for every $\epsilon>0$, there exists a function $b_\epsilon \in C_c^\infty (\rn)$ such that $\| b- b_\epsilon\|_{BMO} < \epsilon$. Then for each $f \in E$, we have
	\begin{align*}
		\| [b,T] (f) \|_{ M_{p,q}^{t,r} (\rn)}  &\le \| [b - b_\epsilon,T] (f) \|_{ M_{p,q}^{t,r} (\rn)} + \| [b_\epsilon,T] (f) \|_{ M_{p,q}^{t,r} (\rn)} \\
		& \lesssim \| b- b_\epsilon\|_{BMO} \| f \|_{ M_{p,q}^{t,r} (\rn)} + \| [b_\epsilon,T] (f) \|_{ M_{p,q}^{t,r} (\rn)} \\
		& \le C \epsilon + \| [b_\epsilon,T] (f) \|_{ M_{p,q}^{t,r} (\rn)} .
	\end{align*}
	By this inequality, it suffices to show that $[ b_\epsilon,T] (E)$  is relatively compact in $ M_{p,q}^{t,r} (\rn)$.
	
	Since $E$ is a bounded set in $M_{p,q}^{t,r} (\rn)$, by the upper bound of $[ b_\epsilon,T] (E)$, we obtain $[ b_\epsilon,T] (E)$ satisfies (i) in Lemma \ref{precompact}.
	
	Next, we prove that $[ b_\epsilon,T] (E)$ satisfies (ii) in Lemma \ref{precompact}. Indeed, suppose that supp $b_\epsilon \subset Q(0,R_\epsilon)$ for some $R_\epsilon >0$. Then for any $f\in E$, and for $x \in  \rn \backslash Q(0,R)$ with $R > 10 R_\epsilon$, we obtain $|x-y| \approx |x| $ for any $y \in Q(0,R_\epsilon)$.
	Hence, 
	\begin{align*}
		| [b_\epsilon,T] (f) (x)| & = | T(b_\epsilon f) (x) |  \\
		& \lesssim \| b_\epsilon\|_{L^\infty (\rn)}  \int_{ Q(0,R_\epsilon)  }  \frac{|f(y)|}{ |x-y|^n } \d y\\
		& \lesssim \| b_\epsilon\|_{L^\infty (\rn)} |x|^{-n} \int_{ Q(0,R_\epsilon)  }  |f(y)|\d y \\
		& \lesssim \| b_\epsilon\|_{L^\infty (\rn)} |x|^{-n}  \| f\|_{M_{p,q}^{t,r} (\rn)  }  |Q(0,R_\epsilon)|^{1/t'} \\
		& \lesssim  |x|^{-n} R_\epsilon ^{n /t'}.
	\end{align*}
	Let us fix $k \in \mathbb Z$ and fix $Q(0,R)$. 
	\begin{align*}
		&\| [b_\epsilon,T] f \chi_{B^c (0,R)}  \|_{  M_{p,q}^{t,r} (\rn)  } \\
		&\approx  \sup_{ \|g\|_{ \H _{p',q'} ^{t' , r'} (\rn)} =1 }  \int_\rn \left| [b_\epsilon,T] f (x) \chi_{B^c (0,R)}  (x) g (x) \right|  \d x \\
		& \lesssim     R_\epsilon ^{n /t'} \sup_{ \|g\|_{ \H _{p',q'} ^{t' , r'} (\rn)} =1 }  \int_\rn |x|^{-n} \left|  \chi_{B^c (0,R)}  (x) g (x) \right|  \d x \\ 
		& \lesssim   R_\epsilon ^{n /t'} \sup_{ \|g\|_{ \H _{p',q'} ^{t' , r'} (\rn)} =1 }  \sum_{k=1}^\infty  \int_{B(0,2^k R   )\backslash B(0,2^{k-1} R   )} |x|^{-n} \left| g (x) \right|  \d x \\ 
		& \lesssim R_\epsilon ^{n /t'} \sup_{ \|g\|_{ \H _{p',q'} ^{t' , r'} (\rn)} =1 }  \sum_{k=1}^\infty  (2^k R) ^{-n} \int_{B(0,2^k R   )\backslash B(0,2^{k-1} R   )} \left|   g (x) \right|  \d x \\ 
		& \lesssim  R_\epsilon ^{n /t'}  \sum_{k=1}^\infty  (2^k R) ^{-n}  (2^{k} R)^{n/t} \\ 
		& \lesssim  R_\epsilon ^{n /t'}  R^{-n/t'} . 
	\end{align*}
	Hence
	\begin{equation*}
		\| [b_\epsilon,T] f \chi_{B^c (0,R)}  \|_{  M_{p,q}^{t,r} (\rn)  } \lesssim R_\epsilon ^{n /t'}  R^{-n/t'} \| b_\epsilon\|_{L^\infty (\rn)}  \| f\|_{M_{p,q}^{t,r} (\rn)  }.
	\end{equation*}
	Since $ n/t' >0 $, we obtain $ \| [b_\epsilon,T] f \chi_{B^c (0,R)}  \|_{  M_{p,q}^{t,r} (\rn)  }  \to 0 $  as $R \to \infty$. Therefore, (ii) in Lemma \ref{precompact} follows.
	
	Next, we prove the equicontinuity of   $[b_\epsilon,T]$. It suffices to show that for every $\zeta >0$, if $|z|$ is sufficiently small (merely depending on $\zeta$) then, for every $f \in E$,
	\begin{equation*}
		\| [b_\epsilon,T] f (\cdot + z)  - [b_\epsilon,T] f (\cdot) \|_{ M_{p,q}^{t,r} (\rn)  }  \le C \zeta^\delta,
	\end{equation*}
	where  the constant $C >0$ is independent of $ f, \zeta, |z| $.
	
	To obtain the desired result, we recall the maximal operator of $T$, defined by 
	\begin{equation*}
		\mathcal T (f) (x) = \sup_{\zeta>0} |T_\zeta (f) (x) |,
	\end{equation*}
	where $T_\zeta$, the truncated operator of $T$, is defined by
	\begin{equation*}
		T_\zeta (f) (x) = \int_{ \{y\in \rn :|x-y| >\zeta  \} }  K (x,y)  f(y) \d y.
	\end{equation*}
	
	We first recall  Cotlar's  inequality (see \cite[p. 291, Lemma 6.1]{T86}). For all $\ell >0$,
	\begin{equation*}
		\mathcal T (f) (x) \lesssim \mathcal M_{\ell} (T (f) )  (x)  + \mathcal  M f (x).
	\end{equation*}
	Now, for $x\in \rn$, we have
	\begin{align*}
		& [b_\epsilon,T] f (x + z)  - [b_\epsilon,T] f (x) \\
		&= \int_\rn \left( b_\epsilon (y) - b_\epsilon (x+z) \right)  K (x+z , y) f(y) \d y \\
		&  \quad - \int_\rn \left( b_\epsilon (y) - b_\epsilon (x) \right)  K (x , y) f(y) \d y \\
		& = \int_{ \{ y:|x-y| > \zeta^{-1} |z| \} } \left( b_\epsilon (x) - b_\epsilon (x+z) \right)  K (x , y) f(y) \d y \\
		& \quad + \int_{\{ y:|x-y| > \zeta^{-1} |z| \}} \left( b_\epsilon (y) - b_\epsilon (x+z) \right) \left(  K (x +z , y)  -K (x,y) \right) f(y) \d y \\
		& \quad + \int_{\{ y:|x-y| \le \zeta^{-1} |z|\} } \left( b_\epsilon (x) - b_\epsilon (y) \right) K (x  , y)  f(y) \d y \\
		& \quad + \int_{\{ y:|x-y| \le \zeta^{-1} |z| \} } \left( b_\epsilon (y) - b_\epsilon (x+z) \right)  K (x +z , y)  f(y) \d y \\
		& = : I_1 + I_2 + I_3 + I_4.
	\end{align*}
	We first estimate $I_1$.
	\begin{equation*}
		|I_1| \le  \left|b_\epsilon (x) - b_\epsilon (x+z) \right|  \left| \int_{ \{ y:|x-y| > \zeta^{-1} |z|  \} } K (x , y) f(y) \d y\right| \le \| \nabla b_\epsilon \|_{L^\infty(\rn)}  |z| 	\mathcal T (f) (x).
	\end{equation*}
	Therefore, by  Cotlar's  inequality, we obtain
	\begin{align*}
		\| I_1 \|_{M_{p,q}^{t,r} (\rn)  } \le \| \nabla b_\epsilon \|_{L^\infty(\rn)}  |z| 	\| \mathcal T (f) \|_{M_{p,q}^{t,r} (\rn)  }  \lesssim \| \nabla b_\epsilon \|_{L^\infty(\rn)}  |z| \|f\|_{ M_{p,q}^{t,r} (\rn)  }.
	\end{align*} 
	As for $I_2$, by the smoothness of kernel $K$, we get
	\begin{equation*}
		|I_2| \lesssim \| b_\epsilon\|_{L^\infty(\rn)}  |z|^\delta \int_{\{ y:|y| > \zeta^{-1} |z| \} } \frac{|f(x-y)|}{|y|^{n+\delta}} \d y.
	\end{equation*}
	Applying Minkowski's inequality, we obtain
	\begin{align*}
		\| I_2 \|_{M_{p,q}^{t,r} (\rn)  } & \lesssim  \| b_\epsilon\|_{L^\infty(\rn)}  |z|^\delta \|f\|_{M_{p,q}^{t,r} (\rn)  } \int_{\{ y:|y| > \zeta^{-1} |z| \} } \frac{1}{|y|^{n+\delta}} \d y  \\
		& \lesssim  \| b_\epsilon\|_{L^\infty(\rn)}  \zeta^\delta \|f\|_{M_{p,q}^{t,r} (\rn)  } .
	\end{align*}
	For $I_3$, from the size condition of $K$ and the change of variable, we have
	\begin{align*}
		|I_3| &\lesssim \| \nabla b_\epsilon \|_{L^\infty(\rn)}  \int_{\{ y:|x-y| \le \zeta^{-1} |z|\} } |x-y|  |x-y|^{-n} |f(y)|\d y \\
		&= \| \nabla b_\epsilon \|_{L^\infty(\rn)}  \int_{\{ y:|y| \le \zeta^{-1} |z|\} }  |y|^{-n+1} |f(x-y)|\d y .
	\end{align*}
	Applying Minkowski's inequality again, we obtain
	\begin{align*}
		\| I_3 \|_{M_{p,q}^{t,r} (\rn)  } & \lesssim  \| b_\epsilon\|_{L^\infty(\rn)}  \|f\|_{M_{p,q}^{t,r} (\rn)  } \int_{\{ y:|y| \le \zeta^{-1} |z|\} }  |y|^{-n+1}\d y \\
		& \lesssim  \| b_\epsilon\|_{L^\infty(\rn)}  \|f\|_{M_{p,q}^{t,r} (\rn)  } \zeta^{-1} |z| .
	\end{align*}
	Finally, for $I_4$, arguing as in the proof of $I_3$, we also obtain
	\begin{align*}
		|I_4| & \lesssim \| \nabla b_\epsilon \|_{L^\infty(\rn)}  \int_{\{ y:|x-y| \le \zeta^{-1} |z|\} } |x+z-y|  |x+z-y|^{-n} |f(y)|\d y \\
		& = \| \nabla b_\epsilon \|_{L^\infty(\rn)}  \int_{\{ y:|x-y| \le \zeta^{-1} |z|\} }  |x+z-y|^{-n+1} |f(y)|\d y \\
		& \lesssim \| \nabla b_\epsilon \|_{L^\infty(\rn)}  \int_{\{ y:|y| \le \zeta^{-1} |z| +|z| \} }  |y|^{-n+1} |f(y)|\d y .
	\end{align*}
	Thus, applying Minkowski's inequality again, we obtain
	\begin{align*}
		\| I_4 \|_{M_{p,q}^{t,r} (\rn)  } & \lesssim  \| b_\epsilon\|_{L^\infty(\rn)}  \|f\|_{M_{p,q}^{t,r} (\rn)  } \int_{\{ y:|y| \le \zeta^{-1} |z| +|z| \} }  |y|^{-n+1} |f(y)|\d y \\
		& \lesssim  \| b_\epsilon\|_{L^\infty(\rn)}  \|f\|_{M_{p,q}^{t,r} (\rn)  } \left(  \zeta^{-1} |z| +|z| \right).
	\end{align*}
	Hence, we obtain
	\begin{equation*}
		\| [b_\epsilon,T] f (\cdot + z)  - [b_\epsilon,T] f (\cdot) \|_{ M_{p,q}^{t,r} (\rn)  } \lesssim |z| +  \zeta^\delta + \zeta^{-1} |z| +  \left(  \zeta^{-1} |z| +|z| \right) .
	\end{equation*}
	Let $|z| < \zeta^2 $,  and we get
	\begin{equation*}
		\| [b_\epsilon,T] f (\cdot + z)  - [b_\epsilon,T] f (\cdot) \|_{ M_{p,q}^{t,r} (\rn)  } \lesssim  \zeta^\delta .
	\end{equation*}
	Therefore,  $[b_\epsilon,T]$ satisfies (ii) in Lemma \ref{precompact}. Thus, we prove that $[b_\epsilon,T]$  is a compact operator on $ M_{p,q}^{t,r} (\rn) $.
	
	Conversely, 	
	suppose that $T$ is homogeneous, and  $[b_\epsilon,T]$  is a compact operator on $ M_{p,q}^{t,r} (\rn) $. By Theorem \ref{char b commutator}, we have that $b \in \BMO$. Next we show that $b \in \CMO$. To get this result,  we first recall a characterization of a function in $\CMO$.
	
	\begin{lemma}
		[Lemma 3, \cite{U78}] \label{CMO char}
		Let $b \in  \BMO $. Then $b \in \CMO$ if and only if $b$ satisfies the following three conditions:
		
		{ \rm (i)}   $ \lim_{R \to \infty}   \sup_{ Q:  |Q| = R }  \frac{1}{|Q| }\int_Q  |b(z) - b_Q | \d z  =0$;
		
		{ \rm (ii)}   $ \lim_{R \to \infty}   \sup_{ Q :  Q \subset B (0,R)^c }  \frac{1}{|Q| }\int_Q  |b(z) - b_Q | \d z  =0$;
		
		{ \rm (iii)}   $ \lim_{R \to 0}   \sup_{ Q:  |Q| = R }  \frac{1}{|Q| }\int_Q  |b(z) - b_Q | \d z  =0$;
	\end{lemma}
	
	\begin{lemma}[Lemma 5.3, \cite{DK24}]
		There exists a constant $M \ge 10$ such that for any ball $B(x_0,R_0) \subset \rn$, there is a ball $B(y_0,R_0)$, $ |x_0 -y_0 | = M R_0 $; and for any $x  \in B(x_0,R_0) $ , $T (\chi_{B(y_0,R_0)})  (x)$  does not change sign and 
		\begin{equation*}
			M^{-n} \lesssim | T (\chi_{B(y_0,R_0)})  (x)|.
		\end{equation*}
	\end{lemma}
	
	Now, we show that $b \in  \CMO$. Seeking a contradiction, we assume that $b \notin \CMO $. Therefore, $b$ violates (i)-(iii).
	
	Case 1: Suppose that $b$ violates (i). Then there exists a sequence of cubes $\{ Q_k = Q ( c_Q, \ell(Q_k) /2 ) \}_{k\ge 1}$ such that $\ell(Q_k) \to \infty$ as $k \to \infty$ and 
	\begin{equation*}
		\frac{1}{|Q_k|} \int_{Q_k}  |b(x) - b_{Q_k} | \d x \ge c_0 >0 \;  \operatorname{for \; each\;} k \ge 1.
	\end{equation*}
	Since  $\ell(Q_k) \to \infty$, we can choose a subsequence of $ \{ \ell(Q_k) \}_{k\ge 1}$ (still denoted by  $ \{ \ell(Q_k)  \}_{k\ge 1}$) such that
	\begin{equation*}
		\ell(Q_k)  \le C^{-1} \ell(Q_{k+1}) ,  \; \forall k \ge 1 
	\end{equation*}
	for some constant $C>10$.
	
	We denote $m_b (\Omega)$, by the median value of function $b$ on a bounded set $\Omega \subset \rn$ (possibly non-unique)  such that
	\begin{equation*}
		\begin{cases}
			| \{  x\in \Omega : b(x) > m_b (\Omega) \}|  \le \frac{1}{2} |\Omega|, \\
			| \{  x\in \Omega : b(x) < m_b (\Omega) \}|  \le \frac{1}{2} |\Omega|.
		\end{cases}
	\end{equation*} 
	Next, for any $k \ge 1$, let $y_k \in \rn$ be such that $ |c_{Q_k}  - y_k | = M \ell (Q_k) /2 $, $M >10$, and put
	\begin{align*}
		& \tilde B_k = B (y_k, R_k) ,  \quad \tilde B_{k,1} = \{ y \in \tilde B_k : b (y) \le m_b ( \tilde B_k) \}, \\
		& \tilde B_{k,1} = \{ y \in \tilde B_k : b (y) \ge m_b ( \tilde B_k) \}, \\
		& B_{k,1}= \{ x\in B_k : b(x) \ge m_b ( \tilde B_k) \}, B_{k,2}= \{ x\in B_k : b(x) < m_b ( \tilde B_k) \}, \\
		& F_{k,1} = \tilde B_{k,1} \backslash \bigcup_{j=1}^{k-1} \tilde B_j , \quad F_{k,2} = \tilde B_{k,2} \backslash \bigcup_{j=1}^{k-1} \tilde B_j .
	\end{align*}
	
	Then from \cite[(5.20)]{DK24}, we obtain
	\begin{equation*}
		\frac{c_0}{4} M^{-n} \lesssim \frac{1}{|B_k|} \int_{B_{k,1} }  \left|   [b,T] (\chi_{F_{k,1}})  (x) \right| \d x.
	\end{equation*}
	Set $ \phi_k (x) = R_k ^{ - n/t} \chi_{F_{k,1}}  (x) $ for each $k \ge 1$. Then for all $k\ge 1$,
	\begin{equation} \label{phi_k gtrsim 1}
		\| \phi_k \|_{  M_{p,q}^{t,r} (\rn) } \gtrsim  R_k ^{ - n/t} R_k^{n/t} =1.
	\end{equation}
	By the compactness of $[b,T]$ on $ M_{p,q}^{t,r} (\rn)$, there exists a subsequence of $\{ [b,T] (\phi_k) \}_{k\ge 1}  $ (still denoted as $\{ [b,T] (\phi_k) \}_{k\ge 1}  $) such that 
	\begin{equation} \label{bT phi_k to G}
		[b,T] (\phi_k) \to G \in M_{p,q}^{t,r} (\rn) 
	\end{equation}
	as $k \to \infty$. By (\ref{phi_k gtrsim 1}), we obtain
	\begin{equation} \label{1 lesssim G}
		1 \lesssim \|G\|_{ M_{p,q}^{t,r} (\rn) }.
	\end{equation}
	Next let $t_2  \in (t,r)$. Since $[b,T]$  maps  $M_{p,q}^{t_2,r} (\rn) \to M_{p,q}^{t_2,r} (\rn) $, then we get
	\begin{equation*}
		\|	[b,T] (\phi_k)  \|_{ M_{p,q}^{t_2,r} (\rn) } \lesssim \|b\|_{BMO (\rn)} \|\phi_k \|_{ M_{p,q}^{t_2,r} (\rn) } \lesssim  \|b\|_{BMO (\rn)} R_k ^{ - n/t} R_k^{n/t_2}.
	\end{equation*}
	This implies that $	[b,T] (\phi_k) \to 0 $ in $M_{p,q}^{t_2,r} (\rn)$ when $k \to \infty$. This contradicts (\ref{1 lesssim G}). Thus, $b$ cannot violate (i).
	
	Case 2: Assume that $b$ violate (ii). The proof of this case is similar to the one of Case 1. We omit it here.
	
	Case 3: Assume that $b$ violate (iii). The proof of this case is similar to the one of Case 1 by considering $\delta_k $ in place of $R_k$, with $\delta_k \to 0$. Since $\delta_k \to 0$, then for every $C >10$, there is a subsequence of $\{ \delta_k\}_{k\ge 1}$ (still denoted as $\{\delta_k\}_{k\ge 1}$)  such that $\delta_{k+1}  \le C^{-1} \delta_k $.
	
	Furthermore, we need to redefine $F_{k,1}$ and $F_{k,2}$:
	\begin{equation*}
		F_{k,1} = \tilde B_{k,1} \backslash \bigcup_{j=K+1}^{\infty} \tilde B_j , \quad  F_{k,2} = \tilde B_{k,2} \backslash \bigcup_{j=k+1}^{\infty} \tilde B_j .
	\end{equation*}
	By the definition of the median value, it is not difficult to verify that $ |F_{k,1}| \approx | \tilde B _k | $, and $ |F_{k,2}| \approx | \tilde B _k | $. This allow us to mimic the proof of Case 1 to get (\ref{bT phi_k to G}) and (\ref{1 lesssim G}).
	
	Now let fix $w $ such that  $1< w<t$. Then we have
	\begin{align*}
		\| [b,T] (\phi_k) \|_{ L^w (\rn) } & \lesssim  \|b\|_{BMO (\rn)} \|\phi_k \|_{ L^w (\rn) } =  \|b\|_{BMO (\rn)} \|\delta_k ^{ - n/t} \chi_{F_{k,1}}   \|_{ L^w (\rn) } \\
		& \lesssim  \|b\|_{BMO (\rn)} \delta_k ^{ - n/t} \delta_k^{n/w} .
	\end{align*}
	This shows  $[b,T] (\phi_k) \to 0$ in $L^w (\rn)$ when $k \to \infty$. As a result, we obtain $G =0$, which contradicts (\ref{1 lesssim G}). Therefore, $b$ must satisfy (iii) in the Lemma \ref{CMO char}.
	
	From the above cases, we conclude that $b \in \CMO $. Hence, we complete the proof.
\end{proof}

\section*{Ethical Approval} 

No applicable for both human and/ or animal studies.

\section*{Competing interests}

The authors have no competing interests to declare that are relevant to the content of this article.

\section*{Authors' contributions}

T. Bai and J. Xu wrote the draft. All authors reviewed the manuscript.

%\section*{Funding}
%The corresponding author Jingshi Xu is supported by the National Natural Science Foundation of China (Grant No. 12161022) and the Science and Technology Project of Guangxi (Guike AD23023002).
%Pengfei Guo is supported by the National Natural Science Foundation of China (Grant No. 12061030).

\section*{Availability of data and materials} No data and materials was used for the research described in the article.

%\bibliographystyle{plain}
%\bibliography{BML}

\end{document}